\documentclass[11pt,a4paper]{amsart}

\usepackage{amsthm, amsfonts, amssymb, amsmath, latexsym, enumerate, enumitem, array}
\usepackage[all]{xy}
\SelectTips{cm}{}

\usepackage{amssymb,eucal,graphicx}
\usepackage{enumerate}
\usepackage{colortbl}
\usepackage{xcolor}
\usepackage{mathtools}

\usepackage{hyperref}

 \newtheorem*{akn}{Acknowledgments}
 \newenvironment{akn*}{\begin{akn}\em}{\end{akn}}

\definecolor{darkgreen}{rgb}{.1,.7,.3}

\numberwithin{equation}{section}
\newtheorem{theo}{Theorem}[section]
\newtheorem{lemma}[theo]{Lemma}
\newtheorem{cor}[theo]{Corollary}
\newtheorem{prop}[theo]{Proposition}
\newtheorem{dfntn}[theo]{Definition}
\newtheorem{rem}[theo]{Remark}
\newtheorem{claim}[theo]{Claim}

\newtheorem{ass}[theo]{Assumptions}
\newtheorem{case1}[theo]{Case I:}
\newtheorem{case2}[theo]{Case II:}
\newtheorem{ex}[theo]{Example}
\newtheorem{fact}[theo]{FACT}
\newtheorem{facts}[theo]{FACTS}
\newenvironment{rem*}{\begin{rem}\em}{\end{rem}}
\newenvironment{ex*}{\begin{ex}\em}{\end{ex}}
\newenvironment{claim*}{\begin{claim}\em}{\end{claim}}
\newenvironment{facts*}{\begin{facts}\em}{\end{facts}}
\newenvironment{fact*}{\begin{fact}\em}{\end{fact}}
\newenvironment{case1*}{\begin{case1}\em}{\end{case1}}
\newenvironment{case2*}{\begin{case2}\em}{\end{case2}}

\newcommand{\brref}[1]{(\ref{#1})}

\newcommand{\Pin}[1]{{\mathbb P}^{#1}}

\newcommand{\Projcal}[1]{\mathbb{ P}({\mathcal #1})}

\newcommand{\oofp}[2]{{\mathcal O}_{\mathbb{ P}^{#1}}({#2})}

\newcommand{\scrollcal}[1]{(\Projcal{#1},\tautcal{#1})}
\newcommand {\xel} {(X, L)}

\newcommand{\tautcal}[1]{{\mathcal O}_{\mathbb{P}({\mathcal#1})}(1)}

\newcommand{\num}{\equiv}

\newcommand{\Pp}{\mathbb P}
\newcommand{\Oc}{\mathcal O}

\newcommand{\FF}{\mathbb{F}}

\newcommand{\Ee}{\mathcal{E}_e}
\newcommand{\Eo}{\mathcal{E}_0}
\newcommand{\E}{\mathcal{E}}
\newcommand{\Ext}{{\rm Ext}}

\newcommand{\cA}{\mathcal{A}}

\newcommand{\cD}{\mathcal{D}}
\newcommand{\cH}{\mathcal{H}}

\newcommand{\cK}{\mathcal{K}}
\newcommand{\cL}{\mathcal{L}}

\newcommand{\cF}{\mathcal{F}}
\newcommand{\cM}{\mathcal{M}}
\newcommand{\cG}{\mathcal{G}}
\newcommand{\cV}{\mathcal{V}}

\newcommand{\cQ}{\mathcal{Q}}
\newcommand{\cU}{\mathcal{U}}
\newcommand{\cO}{\mathcal{O}}

\makeatletter
\@namedef{subjclassname@2020}{%
  \textup{2020} Mathematics Subject Classification}
\makeatother

\allowdisplaybreaks[1]

\title{Ulrich Bundles on  some threefold scrolls over ${\FF}_e$}

\author{Maria Lucia Fania}
\address{Maria Lucia Fania\\ Dipartimento di Ingegneria e Scienze dell'Informazione e Matematica\\
Universit\`{a} degli Studi di L'Aquila\\
Via Vetoio Loc. Coppito\\67100 L'Aquila\\Italy}
\email{marialucia.fania@univaq.it}

\author{Flaminio Flamini}
\address{Flaminio Flamini\\Dipartimento di Matematica\\ Universit\`a degli Studi di Roma
Tor Vergata \\ Viale della Ricerca Scientifica, 1 - 00133 Roma\\Italy}
\email{flamini@mat.uniroma2.it}



\thanks{The first author is supported by PRIN 2017SSNZAW. The second author has been partially supported by  the MIUR Excellence Department Project MatMod@TOV awarded to the Department of Mathematics, University of Rome Tor Vergata.  Both authors are members of  INdAM-GNSAGA}

\begin{document}



\begin{abstract} We investigate the existence of Ulrich vector bundles on suitable $3$-fold scrolls $X_e$ over Hirzebruch surfaces $\mathbb{F}_e$, for any $e \geqslant  0$, which arise as tautological embeddings of  projectivization of very-ample vector bundles  on $\mathbb{F}_e$ that are {\em uniform} in the sense of Brosius and Aprodu--Brinzanescu, cf. \cite{bro} and \cite{ApBr} respectively. 

We explicitely describe components of moduli spaces of rank $r \geqslant 1$ vector bundles  which are Ulrich with respect to the tautological polarization on $X_e$ and whose general point is a slope-stable, indecomposable vector bundle. We moreover determine the dimension of such components, proving also that they are generically smooth. As a direct consequence of these facts, we also compute the {\em Ulrich complexity} of any such $X_e$ and give an effective proof of the fact that these $X_e$'s turn out to be {\em geometrically Ulrich wild}. 

At last, the machinery developed for $3$--fold scrolls $X_e$ allows us to deduce Ulrichness results on rank $r \geqslant 1$ 
vector bundles on $\FF_e$, for any $e \geqslant 0$, with respect to a naturally associated (very ample) polarization. 
\end{abstract}

\maketitle

\section*{Introduction} Let $X$ be a smooth irreducible projective variety of dimension $n \geqslant1$ and let $H$ be a very ample divisor on $X$. A vector bundle $\cU$  on $X$ is said to be 
\emph{Ulrich with respect to $H$} if it satisfies suitable cohomological conditions involving some multiples of the polarization induced by $H$ 
(cf. Definition\;\ref{def:Ulrich} below for precise statement and, e.g. \cite[Thm.\;2.3]{b}, for equivalent conditions).  

Ulrich vector bundles first appeared in Commutative Algebra in the paper \cite{U} by B. Ulrich  from  1984, since these bundles enjoy suitable extremal cohomological properties.  
After that, the attention on Ulrich bundles entered in the realm of Algebraic Geometry with the paper \cite{ESW} where, among other things, the authors compute 
the Chow form of a projective variety $X$ using Ulrich vector bundles on $X$, under the assumption that $X$ supports Ulrich bundles.  

In recent years there has been a huge amount of work on Ulrich bundles (for nice surveys the reader is referred to e.g. \cite{Co,CMP}), mainly investigating the following problems: 

\begin{itemize}
\item Given any polarization $H$ on a variety $X$, does there exist a vector bundle $\cU$ which is Ulrich with respect to $H$?
\item Or even more generally, given a variety $X$ does there exist a very ample divisor $H$, inducing a polarization on $X$, 
and a vector bundle $\cU$ on $X$ which is Ulrich with respect to $H$? 
\item What is the smallest possible rank for an Ulrich bundle on a given polarized variety $(X,H)$ (the so called 
{\em Ulrich complexity of $X$ w.r.t. $H$}, denoted by $uc_H(X)$, cf. Remark \ref{Rmk1}-(i) below)? 
\item If Ulrich bundles on $(X,H)$ do exist, are they {\em stable} bundles? If not empty, are their moduli spaces $\mathcal M$ either smooth or at 
least reduced? 
\item What is $\dim (\mathcal M)$? 
\end{itemize}

Although something is known about these problems for some specific classes of varieties (e.g. curves, Segre, Veronese, Grassmann varieties, rational normal 
scrolls, hypersurfaces, some classes of surfaces and threefolds, cf. e.g. \cite{b,c-h-g-s,Co,CMP} for overviews) the above questions  are still open in their full  generality even for surfaces. 

In the present paper we investigate the case when $X$ is a $3$-fold scroll over a Hirzebruch surface $\mathbb{F}_e$, with $e \geqslant 0$. 
More precisely we focus on $3$-fold scrolls $X_e$ arising as embedding, via very-ample tautological 
line bundles $\Oc_{\Pp(\Ee)}(1)$, of projective bundles $\Pp(\E_e)$, where $\mathcal E_e$ are very-ample 
rank-$2$ vector bundles on  $\mathbb{F}_e$ with Chern classes 
$c_1 (\mathcal E_e)$ numerically equivalent to $3 C_e + b_ef$ and $c_2(\mathcal E_e) =k_e$, where $C_e$ and $f$  are the generators 
of ${\rm Num}(\mathbb{F}_e)$ and where $b_e$ and $k_e$ are integers satisfying some natural numerical conditions (cf. Assumptions \ref{ass:AB} and 
Remark \ref{rem:assAB} below). 

In this set-up one gets $3$-fold scrolls $X_e \subset \mathbb P^{n_e}$, with $n_e = 4 b_e - k_e - 6 e + 4$, which 
are non--degenerate and of degree $\deg(X_e) = 6b_e - 9e - k_e$ (cf. \eqref{eq:nde} below), whose hyperplane section divisor 
we denote by $\xi$. The aim of this paper is to study the behaviour of $3$-fold scrolls $(X_e, \xi)$ as above 
in terms of Ulrich bundles they can support. 

A reason for such interest comes from the fact that the existence of Ulrich bundles on geometrically ruled surfaces has been considered in  \cite{a-c-mr,ant,cas} while  in  \cite{f-lc-pl} the existence of Ulrich bundles of rank one and two on  low degree smooth $3$--fold scrolls over a surface was investigated and among such $3$--folds there are scrolls over $\FF_e$ with $e=0, 1$ (results for other polarizations are contained in \cite{hoc}). Hence it is reasonable to explore what happens for  $3$-fold scrolls $(X_e, \xi)$, for any $e \geqslant 0$. On the other hand  in  \cite{be-fa}, \cite{be-fa-fl}, \cite{fa-fl} the Hilbert schemes of $3$--fold scrolls $(X_e, \xi)$ were studied and so it is natural to understand how the Ulrich bundles would behave in the irreducible components  constructed therein. 

Our main result is the following:

\bigskip

\noindent
{\bf Main Theorem} 
 {\em For any integer $e \geqslant0$, consider the Hirzebruch surface 
$\mathbb{F}_e$ and let $\Oc_{\FF_e}(\alpha,\beta)$ denote the line bundle 
$\alpha C_e + \beta f$ on $\mathbb{F}_e$, where $C_e$ and $f$ are the generators 
of ${\rm Num}(\mathbb{F}_e)$.

Let $(X_e, \xi)$ be a $3$-fold scroll over $\mathbb{F}_e$ as above, 
where $\varphi: X_e \to \FF_e$ denote the scroll map. Then: 

\smallskip

\noindent 
(a) $X_e$  does  not support any Ulrich line bundle  w.r.t. $\xi$ unless $e = 0$. In this latter case, the unique Ulrich line bundles on 
$X_0$ are the following: 
\begin{itemize}
\item[(i)]  $L_1 :=\xi+\varphi^*\Oc_{\FF_0}(2,-1) $ and $ L_2 :=\xi+\varphi^*\Oc_{\FF_0}(-1,b_0-1)$;  
\item[(ii)]  for any integer $t\geqslant  1$, $M_1 :=2\xi+\varphi^*\Oc_{\FF_0}(-1,-t-1)$ and $M_2:=\varphi^*\Oc_{\FF_0}(2,3t-1)$, which only occur for $b_0=2t, k_0=3t$.
\end{itemize}

\smallskip 

\noindent
(b) Set $e=0$ and let $r \geqslant2$ be any integer. Then the moduli space of rank-$r$ vector bundles $\cU_r$ on $X_0$ which are Ulrich w.r.t. $\xi$ and with first Chern class
\begin{eqnarray*}c_1(\cU_r) =
    \begin{cases}
      r \xi + \varphi^*\Oc_{\FF_0}(3, b_0-3) + \varphi^*\Oc_{\FF_0}\left(\frac{r-3}{2}, \frac{(r-3)}{2}(b_0-2)\right),  & \mbox{if $r$ is odd}, \\
      r \xi + \varphi^*\Oc_{\FF_0}(\frac{r}{2},\frac{r}{2}(b_0-2)), \color{black} & \mbox{if $r$ is even}.
    \end{cases}\end{eqnarray*}
    is not empty and it contains a generically smooth component $\mathcal M(r)$ of dimension 
		\begin{eqnarray*}\dim (\mathcal M(r) ) = \begin{cases} \frac{(r^2 -1)}{4}(6 b_0 -4), & \mbox{if $r$ is odd}, \\
			 \frac{r^2}{4} (6b_0-4) +1 , & \mbox{if $r$ is even}.
    \end{cases}
    \end{eqnarray*}
    The general point $[\cU_r] \in \mathcal M(r)$  
corresponds to a  slope-stable vector bundle, of slope w.r.t. $\xi$ given by 
$\mu(\cU_r) = 8b_0 - k_0 -3$. If moreover $r=2$, then $\cU_2$ is also {\em special} (cf. Def. \ref{def:special} below).

\smallskip

\noindent
(c) When $e >0$, let $r \geqslant 2$ be any integer. Then the moduli space of rank-$r$ vector bundles $\cU_{r}$ on $X_e$ which are Ulrich w.r.t. $\xi$ and with first Chern class
\begin{eqnarray*}c_1(\cU_r) =
    \begin{cases}
      r \xi + \varphi^*\Oc_{\FF_e}(3,b_e -3) + \varphi^*\Oc_{\FF_e}\left(\frac{r-3}{2}, \frac{(r-3)}{2}(b_e - e -2)\right), & \mbox{if $r$ is odd}, \\
      r \xi + \varphi^*\Oc_{\FF_e}\left(\frac{r}{2}, \frac{r}{2}(b_e-e-2)\right), & \mbox{if $r$ is even}.
    \end{cases}\end{eqnarray*}
    is not empty and it contains a generically smooth component $\mathcal M(r)$ of dimension 
		\begin{eqnarray*}\dim (\mathcal M(r) ) = \begin{cases} \left(\frac{(r -3)^2}{4}+ 2 \right)(6 b_e - 9e -4) + \frac{9}{2}(r-3) (2b_e-3e), & \mbox{if $r$ is odd}, \\
			 \frac{r^2}{4} (6b_e- 9e-4) +1 , & \mbox{if $r$ is even}.
    \end{cases}
    \end{eqnarray*} The general point $[\cU_{r}] \in \mathcal M(r)$  
corresponds to a  slope-stable vector bundle, of slope w.r.t. $\xi$ given by 
$\mu(\cU_r) = 8 b_e - k_e - 12 e - 3$.  If moreover $r=2$, then $\cU_2$ is also special. 
}

\bigskip

\noindent
The proof of  the {\bf Main Theorem}  will be  the collection of those of Theorems \ref{prop:LineB}, \ref{prop:rk 2 simple Ulrich vctB e=0;I}, \ref{thm:rk 2 vctB e>0}, \ref{thm:general0}, \ref{thm:antonelli3} and \ref{thm:generale}. 
\vspace{2mm}

 We like to point out that our result is not covered by those in \cite{hoc}, since the very ample polarizations considered therein are of the form $\xi+\varphi^*(A)$, where $A$ is a very ample polarization on the base surface (cf. \cite[Theorem B, Theorem 5.1, Corollary 5.17]{hoc}. 

Recall  
that for a given polarized variety $X$ 
 there is the notion of {\em Ulrich wildness}, as suggested by an analogous definition in \cite{DG}.  To be more precise for a projective variety $X \subset \Pp^n$ the notion of being {\em Ulrich wild} can be defined both: 

\smallskip

\noindent
$\bullet$ {\em algebraically}, i.e. in terms of functorial behavior of suitable modules over the homogeneous coordinate ring of the variety $X$, we refer 
the reader to  \cite[Section 2.2]{f-pl} for more precise details,

\noindent
$\bullet$ {\em geometrically}, namely if it possesses families of dimension $r$ of pairwise non--isomorphic, indecomposable, Ulrich vector bundles for arbitrarily large $r$, 
cf. e.g. \cite[Introduction]{DG}.

\smallskip

\noindent
Moreover, if $X$ is Ulrich wild in the {\em algebraic sense}, then it is also Ulrich wild in the {\em geometric sense} (cf. \cite[Rem.\;2.6--(iii)]{f-pl}).

We must mention that the $3$-fold scrolls $(X_e, \xi)$ studied in this paper  are {\em algebraically Ulrich wild} (and  thus, from above, also {\em geometrically Ulrich wild}) and this follows from  the results in \cite{f-pl}. In fact, when $e=0$,  the Ulrich line bundles $L_1$ and $L_2$ as in  {\bf Main Theorem}--(a)  satisfy the conditions of \cite[Theorem A,\;Corollary 3.1]{f-pl} as well as, when $e>0$, two general Ulrich rank $2$ vector bundles  as in Theorem \ref{thm:rk 2 vctB e>0}, which are not isomorphic,  satisfy the same conditions in 
\cite[Theorem A,\;Corollary 3.1]{f-pl}. These facts imply that $X_e$ is {\em (strictly) algebraically Ulrich wild} for any $e \geqslant0$, see \cite[Def.\;2.5]{f-pl} for precise definition.

In this perspective, {\bf Main Theorem}  not only computes the Ulrich complexity of the $3$-fold scrolls $(X_e, \xi)$ that we are considering but it also gives a constructive proof of the fact that  
$(X_e, \xi)$ is {\em geometrically Ulrich wild}, for  any integer $e \geqslant0$, explicitly describing families of  pairwise non--isomorphic, indecomposable, Ulrich vector bundles of arbitrarily large dimension and rank, with further details  concerning possible ranks that can actually occur. Indeed one has: 

\bigskip

\noindent
{\bf Main Corollary} {\em For any $e \geqslant0$,  the moduli spaces $\mathcal M(r)$ constructed in Main Theorem, (a)-(b)-(c), give rise to explicit families of arbitrarily 
large dimension and rank of slope-stable,  pairwise non--isomorphic, indecomposable, Ulrich vector bundles on $(X_e, \xi)$, which gives an effective proof of the geometric Ulrich wildness of such varieties. Moreover, 

\smallskip

\noindent
(a) when $e=0$, the Ulrich complexity of $X_0$ w.r.t. $\xi$ is $uc_{\xi}(X_0) = 1$ and $X_0$ supports Ulrich vector bundles w.r.t. $\xi$ 
of any rank $r \geqslant 1$, with no gaps on $r$; 

\smallskip

\noindent
(b) for $ e>0$, the Ulrich complexity of $X_e$ w.r.t. $\xi$ is $uc_{\xi}(X_e) = 2$ and $X_e$ supports  Ulrich vector bundles w.r.t. $\xi$ of any rank $r \geqslant 2$, with no gaps (except for $r=1$).

}

\bigskip

Notice that $3$-fold scrolls $(X_e, \xi)$ as above are varieties not of {\em minimal degree} in $\Pp^{n_e}$, being $d_e \neq n_e-2$ (see \eqref{eq:nde}), which are moreover {\em (strictly) Ulrich wild} as in \cite[Def.\;2.5]{f-pl}. This observation in particular implies that, for any $e \geqslant 0$, $3$-folds scrolls $(X_e, \xi)$ as above give rise to a class of varieties satisfying \cite[Conjecture\;1.]{f-pl}.

As a consequence of the previous results, we moreover deduce Ulrichness results for vector bundles on the base surface $\FF_e$ with respect to a naturally associated very ample polarization, see Theorem \ref{thm:UlrichFe}, whose proof directly follows from {\bf Main Theorem}, {\bf Main Corollary} and a natural one-to-one correspondence among rank $r$ vector bundles on $X_e$, of the form $\xi \otimes \varphi^*(\mathcal F)$, which are Ulrich w.r.t. $\xi$ on $X_e$, and rank $r$ vector bundles on $\FF_e$, of the form $\mathcal F (c_1(\mathcal E_e))$, which are Ulrich w.r.t. $c_1(\mathcal E_e) = 3 C_e + b_e f$  (cf. Theorem \ref{pullback} and Section\,\ref{S:final} below).

\medskip

An open question is certainly concerned with irreducibility of moduli spaces of rank $r$ vector bundles on $X_e$ as in {\bf Main Theorem} above. 


\bigskip

The paper consists of five sections. In Section 1 we recall some generalities on 
Ulrich vector bundles on projective varieties, which will be used in the sequel, as well as preliminaries from 
\cite{al-be,be-fa-fl,bro} to properly define $3$-fold scrolls $(X_e, \xi)$ which are the core of the paper. 
Sect.\;\ref{Ulrich lb} deals with Ulrich line bundles on scrolls $(X_e, \xi)$, cf. Theorem\;\ref{prop:LineB}, whereas 
Sect.\;\ref{Ulrich rk 2 vb} focuses on the rank-$2$ case, using extensions suitably defined (cf.\;Theorem \ref{prop:rk 2 simple Ulrich vctB e=0;I}) as well as pull-back of appropriate 
bundles coming from the base  (cf. Theorem \ref{pullback} and Theorem \ref{thm:rk 2 vctB e>0}). Section \;\ref{Ulrich higher rk  vb} deals with the general case of any rank $r \geqslant1$, via 
inductive processes, extensions, deformation and modular theory (cf. Theorems \ref{thm:general0} and \ref{thm:generale}). Finally, in Section \ref{S:final}, we deal with the aforementioned Ulrichness result Theorem \ref{thm:UlrichFe} dealing with rank $r$ vector bundles on the base surface $\FF_e$ which turn out to be Ulrich w.r.t. $c_1(\mathcal E_e)$, deduced from {\bf Main Theorem} and {\bf Main Corollary} above.

\begin{akn*} We would like to thank Juan  Pons--Llopis, for pointing out reference \cite{f-pl} and for useful conversation, Antonio Rapagnetta, for pointing out references 
\cite{Drezet}, \cite{OGrady} together with precise explanations and advices concerning Claim \ref{cl:merd} below. We are also grateful to Daniele Faenzi for some of his comments.   Finally, we would like to deeply thank the anonymous referee for his/her enthusiastic report, full of encouragement and with important advices to improve the presentation. 
\end{akn*}

\subsection*{Notation and terminology} We work throughout over the field $\mathbb{C}$ of complex numbers. All schemes will be endowed with the Zariski topology. By \emph{variety}   we mean an integral algebraic scheme. We say that a property holds for a \emph{general}  point of a variety $V$ if it holds for any point in a Zariski open non--empty subset of $V$. We will  interchangeably use the terms {\em rank-$r$ vector bundle} on a variety $V$ and {\em rank-$r$ locally free sheaf} on $V$; in particular for 
the case $r=1$ of line bundles (equiv. invertible sheaves), to ease notation and if no confusion arises, we sometimes identify line bundles with Cartier divisors 
interchangeably using additive notation instead of multiplicative notation and tensor products. Thus, if $L$ and $M$ are line bundles on $V$, the {\em dual} of $L$ will be 
denoted by either $L^{\vee}$, or $L^{-1}$ or even $-L$, so that $M \otimes L^{\vee}$ will be also denoted by either $M \otimes L^{-1}$ or just $M-L$. 
If $\mathcal P$ is either a {\em parameter space} of a flat family of geometric objects $\mathcal E$ defined on $V$ (e.g. vector bundles, extensions, etc.) 
or a {\em moduli space} parametrizing geometric objects modulo a given equivalence relation, we will denote by $[\mathcal E]$ the parameter point (resp., the moduli point) corresponding to the geometric object $\mathcal E$ (resp., associated to the equivalence class of $\mathcal E$). For further non-reminded terminology, we refer the reader to \cite{H}.


\section{Preliminaries} We first remind some general definitions concerning Ulrich bundles on projective varieties.

\begin{dfntn}\label{def:Ulrich} Let $X\subset \Pp^N$ be a smooth variety of dimension $n$  and let $H$  be a hyperplane section of $X$.
A vector bundle $\cU$  on $X$ is said to be  {\em Ulrich} with respect to $H$ if
\begin{eqnarray*}
H^{i}(X, \cU(-jH))=0 \quad \text{for}  \quad i=0, \cdots, n \quad  \text{and} \quad 1 \leqslant j \leqslant \dim X.
\end{eqnarray*}
 \end{dfntn}

  \begin{rem*}\label{Rmk1}  (i) If $X$ supports Ulrich bundles w.r.t. $H$ then one sets $uc_H(X)$, called the {\em Ulrich complexity of $X$ w.r.t. $H$}, to be the minimum rank among possible Ulrich vector bundles w.r.t. $H$ on $X$. 

\noindent
(ii) If $\cU_1$ is a vector bundle on $X$, which is Ulrich w.r.t. $H$,  then 
$\cU_2 :=\cU_1^{\vee}(K_X +(n+1)H)$ is also Ulrich w.r.t. $H$. The vector bundle $\cU_2$ is called the {\em Ulrich dual} of $\cU_1$. 
From this we see that,  if Ulrich  bundles of some rank $r$ on $X$ do exist, then they come in pairs. 
 \end{rem*}
 
\begin{dfntn}\label{def:special} Let $X\subset \Pp^N$ be a smooth variety of dimension $n$  polarized by $H$, where $H$ is a hyperplane section of $X$,   and let $\cU$  be  a rank-$2$ Ulrich  vector bundle  on  $X$. Then $\cU$ is said to be {\em special} if $c_1(\cU) = K_{X} +(n + 1)H.$
 \end{dfntn}

\noindent
Notice that,  because $\cU$ in Definition \ref{def:special} is of rank-$2$, then $\cU^{\vee} \cong \cU (- c_1(\cU))$ therefore 
being special is equivalent for $\cU$ to be isomorphic to its Ulrich dual bundle.

We now remind facts  concerning (semi)stability and slope-(semi)stability properties of these
bundles (cf. \cite[Def.\;2.7]{c-h-g-s}). Let $\mathcal E$ be a vector bundle on $X$; recall that $\mathcal E$ is said to be {\em semistable} if for every non-zero coherent subsheaf
$\mathcal F \subset \mathcal E$, with $0 < {\rm rk}(\mathcal F) := \mbox{rank of} \; \mathcal F < {\rm rk}(\mathcal E)$, the inequality
$\frac{P_{\mathcal F}}{{\rm rk}(\mathcal F)} \leqslant  \frac{P_{\mathcal E}}{{\rm rk}(\mathcal E)}$ holds true, where 
$P_{\mathcal F}$ and $P_{\mathcal E}$ are the Hilbert polynomials of the sheaves. Furthermore, $\mathcal E$ is {\em stable} if 
the strict inequality above holds. 

Similarly, recall that the {\em slope} of a vector  bundle $\mathcal E$ (w.r.t. $\mathcal O_X(H)$) is defined to be \linebreak 
$\mu(\mathcal E) := \frac{c_1(\mathcal E) \cdot H^{n-1}}{{\rm rk}(\mathcal E)}$; the bundle $\mathcal E$ is said 
to be {\em $\mu$-semistable}, or even {\em slope-semistable}, if for every non-zero coherent subsheaf
$\mathcal F \subset \mathcal E$ with $0 < {\rm rk}(\mathcal F)   < {\rm rk}(\mathcal E)$, one has 
$\mu (\mathcal F) \leqslant \mu(\mathcal E)$. The bundle $\mathcal E$ is {\em $\mu$-stable}, or  {\em slope-stable}, if the strict inequality holds. 

The two definitions of (semi)stability are related as follows (cf. e.g. \cite[\S\;2]{c-h-g-s}): 
\begin{eqnarray*}\mbox{slope-stability} \Rightarrow \mbox{stability} \Rightarrow \mbox{semistability} \Rightarrow \mbox{slope-semistability}.\end{eqnarray*}
When the bundle in question is in particular Ulrich, the following more precise situation holds:

\begin{theo}\label{thm:stab} (cf. \cite[Thm.\;2.9]{c-h-g-s}) Let $X\subset \Pp^N$ be a smooth variety of dimension $n$ and let $H$ be a hyperplane section of $X$. Let 
$\cU$ be a rank-$r$ vector bundle on $X$ which is Ulrich w.r.t. $H$. Then: 

\noindent
(a) $\mathcal U$ is semistable, so also slope-semistable;

\noindent
(b) If $0 \to  \mathcal F \to \mathcal  U  \to \mathcal G \to 0$ is an exact sequence of coherent sheaves with $\mathcal G$ 
torsion-free, and $\mu(\mathcal F) = \mu(\mathcal U)$, then $\mathcal F$ and $\mathcal G$ are both Ulrich vector 
bundles.

\noindent
(c) If $\mathcal U$ is stable then it is also slope-stable. In particular, the notions of stability and slope-stability coincide 
for Ulrich bundles. 
\end{theo} 
\vspace{3mm}

 We like to point out that the property of being Ulrich in a family of vector bundles is an open condition.   
Indeed if $\cU$ is a deformation  of an Ulrich bundle $\cF$ then  $\cU$ is  also Ulrich and this  because the cohomology vanishing of $\cF(-j)$, for $1 \leqslant j \leqslant \dim X$, implies the cohomology vanishing of $\cU(-j)$, by semi--continuity. 
\vspace{3mm}

We also like to remark  that because Ulrich bundles are semistable, then any family of Ulrich bundles with given rank and Chern classes is bounded, see for instance A. Langer \cite{langer}.

\vspace{3mm}
In particular, if the  bundles in a bounded family are simple, then Casanellas and Hartshorne have proved

\begin{prop} (see \cite[Proposition  2.10]{c-h-g-s}\label{casanellas-hartshorne} On a nonsingular projective variety $X$, any bounded family of simple bundles $\mathcal E$ with given rank and Chern classes satisfying $H^2(\mathcal E \otimes \mathcal E^{\vee}) = 0$ has a smooth modular family.
\end{prop}

The existence of a smooth modular family of simple vector bundles along with the fact that the property of being Ulrich in a family of vector bundles is an open condition will help us in showing the existence of stable Ulrich bundles on the varieties we are dealing with.

\vspace{3mm}
In the sequel, we will  focus on $n= \dim (X) =3$; in such a case, the following notation will be used throughout this work.

\label{notation}
  \begin{enumerate}
\item [ ] $X$ is a smooth, irreducible, projective variety of dimension $3$ (or simply a {\em $3$-fold});
\item[ ]$\chi(\mathcal F) =  \sum_{i=0}^3 (-1)^i  h^i(\mathcal F)$, the Euler characteristic of $\mathcal F$, where $\mathcal F$ is any vector bundle of rank $r \geqslant1$ on $X$;
\item[ ]  $K_X$ the canonical bundle of $X.$ When the context is clear, $X$ may be dropped, so $K_X = K$;
\item[ ] $c_i = c_i(X)$,  the $i^{th}$ Chern class of $X$;
\item[ ] $d = \deg{X} = L^3$, the degree of $X$ in the embedding given by a very-ample line bundle $L$;
\item[] $g = g(X),$ the sectional genus of $\xel$ defined by $2g-2=(K+2 L)L^2;$ 
\item[] if $S$ is a smooth surface, $\equiv$ will denote the numerical equivalence of divisors on $S$. 
\end{enumerate}

For non-reminded terminology and notation, we basically follow \cite{H}. 

\begin{dfntn}\label{specialvar} A pair $(X, L)$, where $X$ is a $3$-fold and $L$ is an ample line bundle on $X$, is
a {\em scroll} over a normal variety $Y$ if there exist an ample line
bundle $M$ on $Y$ and a surjective morphism  $\varphi: X \to Y$ with
connected fibers
such that $K_X + (4 - \dim Y) L = \varphi^*(M).$
\end{dfntn}

In particular, if $Y$ is a smooth surface and $\xel$ is a scroll over $Y$,  then (see \cite[Prop. 14.1.3]{BESO})
$X \cong \Projcal{E} $, where ${\mathcal E}= \varphi_{*}(L)$ is a vector bundle on $Y$ and
$L$ is the tautological  line bundle on $\Projcal{E}.$ Moreover, if $S \in |L|$ is a smooth divisor,
then (see e.g. \cite[Thm. 11.1.2]{BESO}) $S$ is the blow
up of $Y$ at $c_2(\mathcal{E})$ points; therefore $\chi({\mathcal O}_{Y}) = \chi({\mathcal O}_{S})$ and
\begin{equation}\label{eq:d}
d : = L^3 = c_1^2(\mathcal{E})-c_2(\mathcal{E}).
\end{equation}

For the reader convenience we recall  that the proof of  \cite[Theorem 2.4]{f-lc-pl} more precsisely shows the following result. 

\begin{theo}\label{pullback} Let $(Y,H)$ be a polarized surface with $H$ very ample and let $\mathcal E$ be a rank two vector bundle on $Y$ such that $\mathcal E$ is (very) ample and spanned. 
Let $\mathcal F$ be a rank $r \geqslant 1$ vector bundle on $Y$. Then on the $3$-fold scroll  $X \cong \mathbb{P}(\mathcal E)\xrightarrow{\pi} Y$, 
the vector bundle $\mathcal U:= \xi \otimes \pi^*\mathcal F$ is Ulrich with respect to $\xi$, where  $\xi$ denotes the tautological line bundle on $X$, with $(X,\xi) \cong (\Pp(\mathcal E), \Oc_{\Pp(\mathcal E)}(1))$, if and only if $\mathcal F$ is such that
\begin{equation}\label{needed to pullback}
\begin{array}{ccc}
  H^i(Y,\mathcal F)=0 & \text{and} & H^i(Y,\mathcal F(-c_1(\mathcal E)))=0, \; 0 \leqslant i \leqslant 2.
\end{array}
\end{equation} In particular, if $c_1(\mathcal E)$ is very ample on $Y$, then the rank $r$ vector bunde on $X$, $ \mathcal U=\xi \otimes \pi^*\mathcal F$, is Ulrich with respect to $\xi$ 
 if and only if the rank $r$ vector bundle on $Y$,  $\mathcal F (c_1(\mathcal E))$, is Ulrich with respect to $c_1(\mathcal E)$. 

\end{theo}

Throughout this work, the base $Y$ of the scroll  $X$ in Definition \ref{specialvar}  will be the Hirzebruch surface 
$\FF_e := \Pp(\Oc_{\Pp^1} \oplus\Oc_{\Pp^1}(-e))$, with $e \geqslant  0$ an integer. 

Let $\pi_e : \FF_e \to \Pp^1$ be the natural projection onto the base. Then 
${\rm Num}(\FF_e) = \mathbb{Z}[C_e] \oplus \mathbb{Z}[f],$ where:

\noindent
$\bullet$ $f := \pi_e^*(p)$, for any $p \in \Pp^1$, whereas

\noindent
$\bullet$ $C_e$ denotes either the unique section corresponding to the morphism of vector bundles on $\Pp^1$ \linebreak
$\Oc_{\Pp^1} \oplus\Oc_{\Pp^1}(-e) \to\!\!\!\to \Oc_{\Pp^1}(-e)$, when $e>0$, or the fiber of the other ruling different from that 
induced by $f$, when otherwise $e=0$.

\noindent
In particular$$C_e^2 = - e, \; f^2 = 0, \; C_ef = 1.$$

Let $\Ee$ be a rank-two vector bundle over $\FF_e$ and let $c_i(\mathcal{E}_e)$ be its  $i^{th}$-Chern class. Then $c_1( \mathcal{E}_e) \num a C_e + b f$, for some $ a, b \in \mathbb Z$, and $c_2(\mathcal{E}_e) \in \mathbb Z$. For the line bundle $\cL \num \alpha C_e + \beta f$ we will also use the notation $\Oc_{\FF_e}(\alpha,\beta)$.

From now on, we will consider the following:

\begin{ass}\label{ass:AB} Let $e \geqslant  0$, $b_e$, $k_e$ be integers such that
\begin{equation}\label{(iii)}
b_e-e< k_e< 2b_e-4e,
\end{equation} and let ${\mathcal E}_e$ be a rank-two vector bundle over $\FF_e$, with
\begin{eqnarray*}
c_1({\mathcal E}_e) \num 3 C_e + b_e f \;\; {\rm and} \;\; c_2({\mathcal E}_e) = k_e,
\end{eqnarray*} which fits in the exact sequence 
\begin{equation}\label{eq:al-be}
0 \to A_e \to {\mathcal E}_e \to B_e \to 0,
\end{equation} where $A_e$ and $B_e$ are line bundles on $\FF_e$ such that
\begin{equation}\label{eq:al-be3}
A_e \num 2 C_e + (2b_e-k_e-2e) f \;\; {\rm and} \;\; B_e \num C_e + (k_e - b_e + 2e) f
\end{equation}
\end{ass} From \eqref{eq:al-be}, in particular, one has $c_1({\mathcal E}_e) = A_e + B_e \;\; {\rm and} \;\; c_2({\mathcal E}_e) = A_eB_e$.
\vspace{3mm}

\begin{rem*}\label{rem:assAB} Here we explain the above assumptions and their consequences. First of all, notice that condition $b_e-e < k_e$ in \eqref{(iii)} ensures that the line bundle $B_e$ is very ample (cf. \cite[\S\,V,\;Prop.2.20]{H}) and {\em non--special}, i.e. $h^1(B_e)= 0$ (cf. \cite[Lemma\;3.2]{fa-fl}, where computations hold true also for cases $e=0,1$). Similarly, condition $k_e< 2b_e - 4e$ in \eqref{(iii)}  implies that also $A_e$ is very ample (cf. \cite[\S\,V,\;Prop.2.20]{H}) and non-special (cf. \cite[Lemma\;3.9]{fa-fl}, where computations hold true also for cases $e=0,1$). Notice further that \eqref{(iii)} gives in particular  
$$b_e \geqslant 3e + 1;$$on the other hand, if $b_e = 3e+1$, then \eqref{(iii)} would give 
$$2e+1 < k_e < 2e+2.$$Thus it is clear that, in order to have integers $b_e$ and $k_e$ satisfying \eqref{(iii)}, it is implicit from 
Assumptions \ref{ass:AB}  that one must have 
\begin{equation}\label{eq:rem:assAB}
b_e \geqslant 3e+2, 
\end{equation} which is in accordance with classification results by T. Fujita \cite[(1.3) Lemma]{Fu}.

Moreover, from \cite[Prop.\;2.6,\;4.2]{BDS} and from the non-speciality of $A_e$, it follows that any vector bundle ${\mathcal E}_e$ fitting in the exact sequence 
\eqref{eq:al-be} turns out to be very ample on $\FF_e$, namely the tautological line bundle $\Oc_{\Pp(\mathcal E_e)}(1)$ is very ample on $\Pp(\mathcal E_e)$. This is in accordance 
with the necessary numerical conditions for very--ampleness of ${\mathcal E}_e$ as in \cite[Prop.7.2]{al-be}.   

At last we stress that the existence of the exact sequence \eqref{eq:al-be}, with $A_e, B_e$ as in \eqref{eq:al-be3}, is a natural condition; indeed if one is concerned with very-ample rank-two vector bundles ${\mathcal E}_e$, with $c_1({\mathcal E}_e) \num 3 C_e + b_e f$, one must have 
 that the restriction of ${\mathcal E}_e$ at any fiber $f$ is isomorphic to $\Oc_{\Pp^1}(2) \oplus\Oc_{\Pp^1}(1)$, i.e. ${\mathcal E}_e$ is {\em uniform} in the sense of 
\cite{bro} and \cite{ApBr}, so ${\mathcal E}_e$  fits in an exact sequence as \eqref{eq:al-be}, with $A_e$ and $B_e$ as in \eqref{eq:al-be3} (cf. \cite[Prop.7.2]{al-be} and \cite{bro}). 
\end{rem*}

%
%
\section{Ulrich line bundles on  $3$-fold  scrolls over $\FF_e$} \label{Ulrich lb}
In this section, we consider $3$-dimensional
scrolls over $\FF_e$,  with $e\geqslant  0$,  in projective spaces satisfying conditions as in Assumptions \ref{ass:AB}.  

Let therefore  $\E_e$ be a very ample, rank-two vector bundle over
$\FF_e$ such that $$c_1(\Ee) \num 3 C_e + b_e f, \;\; c_2(\E_e) = k_e,$$ 
with $b_e$ and $k_e$ integers as in \eqref{(iii)}. Let $\scrollcal{E_e}$ be the associated $3$-fold  scroll over $\FF_e,$ 
and let $\pi_e: \FF_e \to \Pin{1}$ and $\varphi: \Pp(\E_e) \to  \FF_e$ be the usual projections. Then $\Oc_{\Pp(\Ee)}(1)$ gives rise to the embedding
\begin{equation}\label{eq:Xe}
\Phi_e:= \Phi_{|\Oc_{\Pp(\Ee)}(1)|}: \, \Pp(\E_e) \hookrightarrow  X_e \subset \Pin{n_e},
\end{equation} where $X_e = \Phi_e( \Pp(\Ee))$ is smooth, non-degenerate, of degree $d_e$  and sectional genus $g_e$, with
\begin{equation}\label{eq:nde}
n_e = 4b_e-k_e-6e+4, \;\; d_e = 6b_e-9e-k_e \;\; {\rm and} \;\; g_e = 2b_e - 3e -2. 
\end{equation}
We set $(X_e, \xi) \cong \scrollcal{E_e}$. Our aim in this section is to find out if there actually exist on $X_e$ line bundles which are Ulrich w.r.t. $\xi$ and, in the affirmative case, 
to classify them.
\bigskip

%

\begin{theo}\label{prop:LineB}  Let $e \geqslant 0$ be an integer and let $(X_e, \xi)$ be a $3$-fold scroll as above. Then $X_e$  does  not support any Ulrich line bundle  w.r.t.  $\xi$ unless $e = 0$, in which case the following are the unique Ulrich line bundles on $X_0$: 
\begin{itemize}
\item[(i)]  $L_1 :=\xi+\varphi^*\Oc_{\FF_0}(2,-1) $ and  its Ulrich dual $ L_2 :=\xi+\varphi^*\Oc_{\FF_0}(-1,b_0-1)$, with  $b_0 \geqslant 2$, see  \eqref{eq:rem:assAB}; 
\item[(ii)]  for any integer $t\geqslant  1$, $M_1 :=2\xi+\varphi^*\Oc_{\FF_0}(-1,-t-1)$ and  its Ulrich dual \\$M_2:=\varphi^*\Oc_{\FF_0}(2,3t-1)$,   which only occur for $b_0=2t, k_0=3t$.
\end{itemize}
\end{theo}

\begin{rem*}\label{Ulrichlinedbls 2 o 4} We like to point out that, for any integer $t\geqslant  1$, if  $(b_0,k_0)= (2t,3t)$  on $\FF_0$ we have four distinct  Ulrich line bundles on $X_0$, for any fixed integer $t \geqslant 1$, 
precisely $L_1$ and its Ulrich dual $L_2$,  $M_1$ and its Ulrich dual $M_2$. For other values of  $b_0 \geqslant 2$  not in the above  mentioned range, we have only two Ulrich line bundles, $L_1$ and its Ulrich dual $L_2$. Moreover, in view of Theorem \ref{pullback}, we notice that the Ulrich line bundles $L_1$ and $L_2$ on $X_0$ are associated to the line bundles 
$\Oc_{\FF_0}(5,b_0-1)$ and $\Oc_{\FF_0}(2,2b_0-1)$, respectively, which are the only Ulrich line bundles on $\FF_0$ with respect to the very ample polarization $c_1(\mathcal E_0) = \Oc_{\FF_0}(3,b_0)$ 
(cf.\;\cite[Example\;2.3]{cas}, \cite[Prop.\;4.4]{ant}). 
\end{rem*}
\begin{proof} (of Theorem \ref{prop:LineB}) Let ${\mathcal L} = a\xi + \varphi^*\Oc_{\FF_e}(\alpha,\beta)$ be an Ulrich line bundle on $X_e$. From \cite[Corollary 2.2]{f-lc-pl} we know that $a=0, 1, 2$. 
%
%

\medskip

\noindent
{\bf Case I}: If $a=1$ then, by Theorem \ref{pullback}, ${\mathcal L} = \xi + \varphi^*\Oc_{\FF_e}(\alpha,\beta)$ is Ulrich with respect to $\xi$ if and only if 
\begin{eqnarray*}
H^i(\FF_e, \Oc_{\FF_e}(\alpha,\beta) ) = H^i(\FF_e, \Oc_{\FF_e}(\alpha,\beta) -c_1({\mathcal E}_e)) =0
\quad  \text{for $i=0, 1, 2.$ }
 \end{eqnarray*} 
 Thus $\chi(\FF_e, \Oc_{\FF_e}(\alpha,\beta) )=\chi(\FF_e, \Oc_{\FF_e}(\alpha,\beta) - c_1({\mathcal E}_e))=0$. 
By Riemann-Roch we get, respectively,
\begin{eqnarray}
\label{eq1}
(\alpha +1)(e\alpha -2\beta-2)=0
\end{eqnarray} and 
\begin{eqnarray}
\label{eq2}
(\alpha-2)(e\alpha -2\beta+2b_e-3e-2)=0
\end{eqnarray} 
Thus either $\alpha=-1$ which, along with \eqref{eq2}, gives  $\beta=b_e-2e-1$ or  $\alpha=2$ which, along with \eqref{eq1}, gives  $\beta=e-1$.
We need to check that  $H^i(\FF_e, \Oc_{\FF_e}(\alpha,\beta) ) = H^i(\FF_e, \Oc_{\FF_e}(\alpha,\beta) -c_1(\Ee)) = 0$ \quad for $i\geqslant 0,$  with $(\alpha, \beta)=(-1, b_e-2e-1)$ or $(\alpha, \beta)=(2, e-1).$ 

If $e=0$  then the vanishings follow by the K{\"u}nneth formula, hence we get $\mathcal L=L_2$ in the first case whereas $\mathcal L=L_1$ in the latter case, where $L_1$ and $L_2$ are as in the statement. 

If  $e>0$, the cohomology groups are not all zero therefore there are no Ulrich line bundles with $a=1$ in these cases. 

%
%

\medskip

\noindent
{\bf Case II}:  If $a=2$ then, by \cite[Corollary 2.2]{f-lc-pl}, $\mathcal L = 2\xi + \varphi^*\Oc_{\FF_e}(\alpha,\beta)$ is Ulrich with respect to $\xi$ if and only if 
\begin{eqnarray*}
H^i(\FF_e, \Oc_{\FF_e}(\alpha,\beta) ) = H^i(\FF_e, \E_e(\Oc_{\FF_e}(\alpha,\beta))) =0
\quad  \text{for $i=0, 1, 2.$}
 \end{eqnarray*} 
 Thus $\chi(\FF_e, \Oc_{\FF_e}(\alpha,\beta) )=\chi(\FF_e,  \E_e(\Oc_{\FF_e}(\alpha,\beta))))=0$.  By Riemann-Roch we get \eqref{eq1} and 
\begin{eqnarray} 
\label{eq3}
-e\alpha^2+2\alpha \beta+\alpha b_e-4e\alpha+2\alpha+5 \beta+4b_e-6e-k_e+5=0,
\end{eqnarray} 
respectively. From \eqref{eq1} either  $\alpha=-1$ or  $\beta=\frac{\alpha e}{2}-1.$

\medskip

{\bf Case II-a}: $\alpha=-1$. Plugging such value  in \eqref{eq3} we get 
 $\beta=-b_e+e+\frac{k_e}{3}-1$, which forces $k_e=3t$ for some $t\in  \mathbb{Z}$, hence $\beta=-b_e+e+t-1$.    We compute $H^i(\FF_e, \Oc_{\FF_e}(\alpha,\beta))=H^i(\FF_e, \Oc_{\FF_e}(-1,-b_e+e+t-1) )$ and  $H^i(\FF_e, \E_e(\Oc_{\FF_e}(-1,-b_e+e+t-1)))$ for $i=0, 1, 2.$

Now $R^{i}{\pi_e}_*(\Oc_{\FF_e}(-1,-b_e+e+t-1))=0$, for $i\geqslant 0$, 
hence, from Leray's isomorphism we have  $H^i(\FF_e, \Oc_{\FF_e}(-1,-b_e+e+t-1) )\cong H^i(\Pp^1, 0 )=0$, for $i=0, 1, 2.$
\vspace{2mm}

To compute $H^i(\FF_e, \Ee(\Oc_{\FF_e}(-1,-b_e+e+t-1)))$
we recall that the vector bundle $\Ee$ sits in the exact sequence \eqref{eq:al-be}, where $A_e\in |\Oc_{\FF_e}(2,2b_e-k_e-2e)|$ and $B_e\in  |\Oc_{\FF_e}(1,k_e-b_e+2e)|$ and after twisting \eqref{eq:al-be} with $\Oc_{\FF_e}(-1,-b_e+e+t-1)$ we have
\begin{equation}\label{eqtwist:al-be}
0 \to  \Oc_{\FF_e}(1,b_e-2t-e-1)\to {\mathcal E}_e(\Oc_{\FF_e}(-1,-b_e+e+t-1)) \to  \Oc_{\FF_e}(0,4t-2b_e+3e-1) \to 0.
\end{equation}

Now 
\begin{eqnarray*}
\begin{aligned}
R^{i}{\pi_e}_*\Oc_{\FF_e}(1,b_e-2t-e-1)=0, \quad \mbox{for $i>0$,  and} \\
{\pi_e}_*\Oc_{\FF_e}(1,b_e-2t-e-1)\cong (\Oc_{\Pp^1} \oplus \Oc_{\Pp^1}(-e)) \otimes \Oc_{\Pp^1} (b_e-2t-e-1),
\end{aligned}
\end{eqnarray*}
hence, from Leray's isomorphism we have
\begin{eqnarray}
\label{coom a sinistra}
 \, \qquad\quad  H^i(\FF_e,\Oc_{\FF_e}(1,b_e-2t-e-1))\cong H^i(\Pp^1, (\Oc_{\Pp^1} \oplus \Oc_{\Pp^1}(-e))  (b_e-2t-e-1))\\ \nonumber
= H^i(\Pp^1, \Oc_{\Pp^1} (b_e-2t-e-1))\oplus H^i(\Pp^1, \Oc_{\Pp^1} (b_e-2t-2e-1)) 
\end{eqnarray} and similarly
\begin{eqnarray}
\label{coom a destra}
H^i(\FF_e,  \Oc_{\FF_e}(0,4t-2b_e+3e-1) )&\cong&H^i(\Pp^1, \Oc_{\Pp^1}(4t-2b_e+3e-1))
\end{eqnarray} 

We  consider first the case $e=0$ and then the case $e \geqslant 1$.

%
%

If  $e=0$, then \eqref{coom a sinistra} and \eqref{coom a destra} become
\begin{eqnarray*}
H^i(\FF_0, \Oc_{\FF_0}(1,b_0-2t-1))&=& H^i(\Pp^1, \Oc_{\Pp^1} (b_0-2t-1)^{\oplus 2}) 
\end{eqnarray*} 
\begin{eqnarray*}
H^i(\FF_0,  \Oc_{\FF_0}(0,4t-2b_0-1) )&=& H^i(\Pp^1, \Oc_{\Pp^1}(4t-2b_0-1))
\end{eqnarray*}

If $b_0-2t-1\geqslant 0$ then 
\begin{eqnarray*}
h^0(\Pp^1, \Oc_{\Pp^1} (b_0-2t-1)^{\oplus 2}) &= &2(b_0-2t) \\
h^1(\Pp^1, \Oc_{\Pp^1} (b_0-2t-1)^{\oplus 2}) &= &0, \quad \text{by Serre's duality on $\Pp^1$.}
\end{eqnarray*}
Note that if  $b_0-2t-1\geqslant 0$ then  
 $4t-2b_0-1 \leqslant -3$ hence $h^0(\Pp^1, \Oc_{\Pp^1}(4t-2b_0-1))=0$ and, by Serre duality, $h^1(\Pp^1, \Oc_{\Pp^1}(4t-2b_0-1))\cong h^0(\Pp^1, \Oc_{\Pp^1}(2b_0-4t-1))=2b_0-4t.$
 
These computations, along with the cohomology sequence associated to \eqref{eqtwist:al-be}, give 
\begin{eqnarray*}
h^i({\mathcal E}_0(\Oc_{\FF_0}(-1,-b_0+t-1)))&=&2b_0-4t\geqslant 2 \, \mbox{(by assumption),  for $i=0,1$,}\\
h^2({\mathcal E}_0(\Oc_{\FF_0}(-1,-b_0+t-1)))&=&0, \, \mbox{trivially.}
\end{eqnarray*}

If $b_0-2t-1< 0$ then 
\begin{eqnarray*}
h^0(\Pp^1, \Oc_{\Pp^1} (b_0-2t-1)^{\oplus 2}) &= &0 \\
h^1(\Pp^1, \Oc_{\Pp^1} (b_0-2t-1)^{\oplus 2})&\cong &2 h^0(\Pp^1, \Oc_{\Pp^1} (2t-b_0-1)), \quad \text{by Serre's duality on $\Pp^1$.}
\end{eqnarray*}
Note that 
\begin{eqnarray*}
h^0(\Pp^1, \Oc_{\Pp^1} (2t-b_0-1))= \left\{
    \begin{array}{ll}
      0 &   \mbox{ if $b_0=2t$}\\ 
       2t-b_0  & \mbox{ if $2t-b_0-1\geqslant 0$}.
    \end{array}
\right.
\end{eqnarray*}

Note also that  $b_0-2t-1< 0$ 
implies that   $4t-2b_0-1 \geqslant -2$, hence 
 
 if $4t-2b_0-1 \geqslant 0$,
  $h^0(\Pp^1, \Oc_{\Pp^1}(4t-2b_0-1))=4t-2b_0$ and $h^1(\Pp^1, \Oc_{\Pp^1}(4t-2b_0-1))=0;$

if $4t-2b_0-1 = -1$, $h^0(\Pp^1, \Oc_{\Pp^1}(4t-2b_0-1))=h^1(\Pp^1, \Oc_{\Pp^1}(4t-2b_0-1))=0;$

if $4t-2b_0-1 = -2$, $h^0(\Pp^1, \Oc_{\Pp^1}(4t-2b_0-1))=0$ and $h^1(\Pp^1, \Oc_{\Pp^1}(4t-2b_0-1))=1$.

These facts, along with \eqref{eqtwist:al-be}, give that 
\begin{eqnarray*}
 h^i({\mathcal E}_0(\Oc_{\FF_0}(-1,-b_0+t-1)))= 0,\, \mbox{ for $i=0,1,2$, \quad if $b_0=2t$, }\\
 h^0({\mathcal E}_0(\Oc_{\FF_0}(-1,-b_0+t-1)))\neq 0, \, \mbox{in the remaining two cases}, 
\end{eqnarray*}  
the latter case holds because otherwise from the cohomology sequence associated to  \eqref{eqtwist:al-be}  it would follow that  $h^1({\mathcal E}_0(\Oc_{\FF_0}(-1,-b_0+t-1)))< 0$, which is impossible. 
Thus we are left with the cases $b_0=2t, k_0=3t$, where $t \geqslant 1$, as it follows from   \eqref{(iii)} and \eqref{eq:rem:assAB}. 
Hence in this case we get  $\mathcal L=M_1$   is Ulrich and its Ulrich dual  is  $M_2$.

\smallskip

%
%
If $e \geqslant 1$ in order to compute the cohomology groups in \eqref{coom a sinistra} and \eqref{coom a destra} we consider first the case in which $b_e-2t-2e-1 \geqslant 0.$ Note that in this case also $b_e-2t-e-1 \geqslant 0$ and thus  in \eqref{coom a sinistra} we have
\begin{eqnarray*}
h^0(\FF_e,\Oc_{\FF_e}(1,b_e-2t-e-1)) = h^0(\Pp^1, (\Oc_{\Pp^1} \oplus \Oc_{\Pp^1}(-e))  (b_e-2t-e-1))&=&2b_e-4t-3e, \\ 
h^1(\FF_e,\Oc_{\FF_e}(1,b_e-2t-e-1))=h^1(\Pp^1, (\Oc_{\Pp^1} \oplus \Oc_{\Pp^1}(-e))  (b_e-2t-e-1))&=&0.
\end{eqnarray*}
As for \eqref{coom a destra} note that 
$4t-2b_e+3e-1 \leqslant -3-e$, by assumption, and thus  
\begin{eqnarray*}
H^0(\FF_e,  \Oc_{\FF_e}(0,4t-2b_e+3e-1) )&\cong&H^0(\Pp^1, \Oc_{\Pp^1}(4t-2b_e+3e-1))=0 \quad \mbox{and}\\
h^1(\Pp^1, \Oc_{\Pp^1}(4t-2b_e+3e-1)))&=&h^0(\Pp^1, \Oc_{\Pp^1}(2b_e-4t-3e-1)))=2b_e-4t-3e.
\end{eqnarray*}

From the cohomology sequence associated to \eqref{eqtwist:al-be} it follows that 
\begin{eqnarray*}
h^0({\mathcal E}_e(\Oc_{\FF_e}(-1,-b_e+e+t-1)))&=&h^1({\mathcal E}_e(\Oc_{\FF_e}(-1,-b_e+e+t-1)))\\
&=&2b_e-4t-3e \geqslant 2+e,
\end{eqnarray*}
because $4t-2b_e+3e-1\leqslant -3-e$.

If  $b_e-2t-e-1 \geqslant 0$ and $b_e-2t-2e-1 < 0$ (the case $b_e-2t-2e-1 \geqslant 0$ was just treated), then 
\begin{eqnarray*}
h^0(\FF_e,\Oc_{\FF_e}(1,b_e-2t-e-1))&=&h^0(\Pp^1, (\Oc_{\Pp^1} \oplus \Oc_{\Pp^1}(-e))  (b_e-2t-e-1))\\&=&b_e-2t-e\geqslant 1.  
\end{eqnarray*}
Thus
 $h^0({\mathcal E}_e(\Oc_{\FF_e}(-1,-b_e+e+t-1))\geqslant h^0(\FF_e,\Oc_{\FF_e}(1,b_e-2t-e-1))= b_e-2t-e\geqslant 1$. 

\medskip

%
%

{\bf Case II-b: } $\beta=\frac{\alpha e}{2}-1.$  Plugging such value in \eqref{eq3} we get 
\begin{eqnarray}\label{alfa}
\alpha=\frac{-8b_e+12e+2k_e}{2b_e-3e}
\end{eqnarray}
which implies that 
\begin{eqnarray*}
(\alpha+4) (2b_e-3e)=2k_e.
\end{eqnarray*}
 By  \eqref{(iii)} and \eqref{eq:rem:assAB},   we get $2b_e-3e\geqslant 3e+4\geqslant 2$ and  $k_e> b_e-e\geqslant 2e+2$. Thus $\alpha+4 >0$, that is $\alpha \geqslant -3$.
Notice that if $\alpha =-3$ then \eqref{alfa} gives $b_e=\frac{3}{2}e+k_e= \frac{1}{2}e+e+k_e> \frac{1}{2}e+b_e$ by  \eqref{(iii)}, which is a contradiction. Hence $\alpha \geqslant -2$ and therefore from \eqref{alfa}
it follows that $k_e \geqslant 2b_e-3e$ which contradicts the condition $k_e< 2b_e-4e$ in  \eqref{(iii)}.

\medskip

The proof is complete since the case $a=0$ is the Ulrich dual of the case $a=2$.
\end{proof}

%
%
\section{Rank-$2$ Ulrich vector  bundles on $3$-fold scrolls over $\FF_e$ } \label{Ulrich rk 2 vb}
%
%

As in the previous section, we consider  $3$-fold
scrolls $(X_e, \xi)$,  with $e\geqslant 0$,  satisfying conditions as in Assumptions \ref{ass:AB}.   
Our aim is to prove the existence of some moduli spaces of rank-$2$ Ulrich vector bundles on such $3$-fold scrolls and to study their basic properties. 
As a matter of notation, in the sequel $F$ will always denote the fiber of the natural scroll map $\varphi: X_e \cong \Pp(\Ee) \to \mathbb{F}_e$.

 \subsection{Rank-$2 $ Ulrich vector bundles on  $3$-fold scrolls over  $\FF_0$}\label{S:subs41}


 From Theorem  \ref{prop:LineB}   we know that on  $X_0$ 
there exist  Ulrich line bundles. Using these line bundles, we will construct rank two Ulrich vector bundles arising as  non-trivial extensions of them.

\bigskip

\noindent {\bf Case L}: Let $L_1$ and $L_2$ be line bundles on $X_0$ as in Theorem  \ref{prop:LineB}-(i).  Notice that 
\begin{eqnarray*}
{\rm Ext}^1(L_2,L_1)&\cong& H^1(X_0,L_1-L_2)=H^1(X_0,\varphi^*\Oc_{\FF_0}(3,-b_0))\cong H^1(\FF_0, \Oc_{\FF_0}(3,-b_0))\\
&\cong& H^1(\Pp^1,S^3(\cO_{\Pp^1}\oplus\cO_{\Pp^1})(-b_0))\cong H^1(\Pp^1, \Oc_{\Pp^1}^{\oplus 4}(-b_0)) \cong H^0(\Pp^1, \Oc_{\Pp^1}^{\oplus 4}(b_0-2)).
\end{eqnarray*} Hence  $\dim {\rm Ext}^1(L_2,L_1)=4b_0-4\geqslant 4,$  being  $b_0\geqslant 2$  (see \eqref{eq:rem:assAB}). Thus there are non-trivial  extensions $\cF_1$ 
 \begin{eqnarray}
\label{extension1}
0 \to L_1  \to \cF_1 \to L_2 \to 0
\end{eqnarray} 
of  $L_2$ by $L_1$. Similarly 
\begin{eqnarray*}
{\rm Ext}^1(L_1,L_2)&\cong&H^1(X_0, L_2-L_1)=H^1(X_0, \varphi^*\Oc_{\FF_0}(-3,b_0))\cong H^1(\FF_0, \Oc_{\FF_0}(-3,b_0))\\
&\cong&H^1(\FF_0, \Oc_{\FF_0}(1,-2-b_0))\cong H^1(\Pp^1, (\cO_{\Pp^1}\oplus\cO_{\Pp^1})(-2-b_0))\\ &\cong &
H^0(\Pp^1, \Oc_{\Pp^1}^{\oplus 2}(b_0)).
\end{eqnarray*}
Hence $\dim {\rm Ext}^1(L_1,L_2)=2b_0+2\geqslant 6$ and thus there are   non-trivial extensions $\cF_1'$ 
\begin{eqnarray}
\label{extension1'}
0 \to L_2  \to \cF_1' \to L_1 \to 0
\end{eqnarray} 
of $L_1$ by $L_2$. Notice that the vector bundles $\cF_1$ and $\cF_1'$ are both rank two vector bundles which are Ulrich w.r.t. $\xi$ with 
\begin{eqnarray*}
c_1(\cF_1) = c_1(\cF_1') =2\xi+\varphi^*\Oc_{\FF_0}(1,b_0-2) \;\; \mbox{and} \;\;\\ \nonumber
 c_2(\cF_1) = c_2(\cF_1') = \xi \cdot\varphi^*\Oc_{\FF_0}(4,2b_0-2) +(2b_0-k_0-1)F.
\end{eqnarray*} 
 Moreover, since $L_1$ and $L_2$ are non-isomorphic line bundles with the same slope $$\mu(L_1)=\mu(L_2)=8b_0-k_0-3$$ with respect to $\xi$ then, by \cite[Lemma 4.2]{c-h-g-s},
$\cF_1$ and $\cF_1'$ are  simple vector bundles, in particular indecomposable.  

The family of extensions \eqref{extension1} is of dimension $4b_0-4$ while the one as in \eqref{extension1'} is of dimension $2b_0+2$, which are different positive 
integers unless $b_0=3$.

\bigskip

  \noindent {\bf Case M}:  Let $M_1$ and $M_2$ be line bundles on $X_0$ as in Theorem  \ref{prop:LineB}-(ii).  As above one computes 
 \begin{eqnarray*}
{\rm Ext}^1(M_2,M_1)&\cong& H^1(X_0, M_1-M_2)=H^1(X_0, 2\xi+\varphi^*\Oc_{\FF_0}(-3,-4t))\\
&\cong& H^1(\FF_0, S^2(\Eo)(-3,-4t)), 
\end{eqnarray*}
hence we need to compute  $H^1(\FF_0, S^2(\Eo)(-3,-4t))$, where $S^2(\Eo)$ denotes the second symmetric power of $\Eo$. The vector bundle $\Eo$ fits in the exact sequence 
\eqref{eq:al-be}, with $A_0$ and $B_0$ as in \eqref{eq:al-be3} and with $b_0=2t$ and $k_0=3t$, $t \geqslant 1$. By \cite[5.16.(c), p. 127]{H},  there is a finite filtration of 
$S^2(\Eo)$, 
$$S^2(\Eo)=F^{0}\supseteq F^{1}\supseteq F^{2} \supseteq F^{3}=0$$
with quotients 
$$ F^{p}/F^{p+1} \cong S^{p}(A_0)\otimes S^{2-p}(B_0),$$for each $0 \leqslant p \leqslant 2$.
Hence $$ F^{0}/F^{1} \cong S^{0}(A_e) \otimes S^{2}(B_0)= 2B_0$$
$$ F^{1}/F^{2} \cong S^{1}(A_0)\otimes S^{1}(B_0)=A_0+ B_0$$
$$ F^{2}/F^{3} \cong S^{2}(A_0)\otimes S^{0}(B_0)=2A_0, \text{ that is } F^{2}=2A_0,$$ since $F^3=0$. 
Thus, we get the following exact sequences
\begin{equation}\label{filtr1}
0 \to F^{1} \to S^2(\Eo) \to 2B_0\to 0
\end{equation}
\begin{equation}\label{filtr2}
0 \to  F^{2} \to F^{1}  \to A_0+B_0\to 0
\end{equation}
\begin{equation}\label{filtr3}
F^{2} = 2 A_0
\end{equation}
Twisting \brref{filtr1},  \brref{filtr2} with $\Oc_{\FF_0}(-3,-4t)$ and using  \brref{filtr3} we get 
\begin{equation}
\label{filtr1tw}
0 \to F^{1} (-3,-4t) \to S^2(\Eo) \otimes \Oc_{\FF_0}(-3,-4t) \to \Oc_{\FF_0}(-1,-2t)\to 0
\end{equation}
\begin{equation}
\label{filtr2tw}
0 \to  {\mathcal O}_{\FF_{0}}(1,-2t)\to F^{1}  \otimes \Oc_{\FF_0}(-3,-4t)  \to {\mathcal O}_{\FF_{0}}(0,-2t)\to 0
\end{equation}

First we focus on \eqref{filtr2tw}; one has $h^{i}( {\mathcal O}_{\FF_{0}}(1,-2t)) = h^{i}({\mathbb P}^1, \oofp{1}{-2t}^{\oplus 2})$,  so, for dimension reasons,  $h^i( {\mathcal O}_{\FF_{0}}(1,-2t)) = 0$, for any $i \geqslant 2$. Since  $t \geqslant 1$,  $h^{0}({\mathcal O}_{\FF_{0}}(1,-2t))=0$ and  $h^{1}({\mathcal O}_{\FF_{0}}(1,-2t))=4t-2$. Similarly 
 $$h^{i}( {\mathcal O}_{\FF_{0}}(0,-2t)) = h^{i}({\mathbb P}^1, \oofp{1}{-2t})$$  so,  $h^i( {\mathcal O}_{\FF_{0}}(0,-2t)) = 0$, for any $i \geqslant 2$, $h^{0}({\mathcal O}_{\FF_{0}}(0,-2t))=0$ and  $h^{1}({\mathcal O}_{\FF_{0}}(0,-2t))=2t-1$  then \brref{filtr2tw} gives 
\begin{eqnarray*}\label{eq:5.aiutob}
 h^{1}( F^{1}(-3, -4t))=6t-3,   \quad  h^{i}( F^{1}(-3, -4t))=0,  \,\mbox{for $i=0,2$}.
\end{eqnarray*}

Passing to \eqref{filtr1tw} observe that,  $h^i( {\mathcal O}_{\FF_{0}}(-1,-2t)) = 0$, for any $i \geqslant 0$. 
This, along with \eqref{eq:5.aiutob} and  \brref{filtr1tw} gives
\begin{eqnarray}\label{h2sym2E}
h^1(2\xi+\varphi^*\Oc_{\FF_0}(-3,-4t))&=&
h^{1}(\FF_{0}, S^2(\Eo(-3,-4t))=6t-3=3b_0-3, \quad  \\
 h^i(2\xi+\varphi^*\Oc_{\FF_0}(-3,-4t))&=&h^{i}(\FF_{0}, S^2(\Eo) (-3,-4t))= 0,  \quad\mbox{for $i=0,2,3$}. \nonumber\end{eqnarray}
Hence $\dim({\rm Ext}^1(M_2,M_1))=3b_0-3 \geqslant 3$  because  $b_0\geqslant 2$  (see \eqref{eq:rem:assAB}). Thus there are non-trivial  extensions $\cF_2$ 
 \begin{eqnarray}
\label{extension2}
0 \to M_1  \to \cF_2 \to M_2 \to 0
\end{eqnarray} of  $M_2$ by $M_1$. Similarly,
 
 \begin{eqnarray*}
{\rm Ext}^1(M_1,M_2)&\cong& H^1(\Pp({\mathcal E}_0), M_2-M_1)=H^1(-2\xi+\varphi^*\Oc_{\FF_0}(3,4t))\\
&\cong& H^2(\varphi^*\Oc_{\FF_0}(-2,-b_0-2))
\cong H^2(\FF_0, \Oc_{\FF_0}(-2,-b_0-2))\\&\cong& H^0(\FF_0, \Oc_{\FF_0}(0,b_0))\cong H^0(\Pp^1, \Oc_{\Pp^1}(b_0)).
\end{eqnarray*}
Hence $\dim({\rm Ext}^1(M_1,M_2))=b_0+1 \geqslant 3$  and thus there are non-trivial extensions $\cF_2'$ 
\begin{eqnarray}
\label{extension2'}
0 \to M_2  \to \cF_2' \to M_1 \to 0
\end{eqnarray} of  $M_1$ by $M_2$. Notice that the vector bundles $\cF_2$ and $\cF_2'$ are both Ulrich rank two vector bundles with 
$$c_1(\cF_2)=c_1(\cF_2')=2\xi+\varphi^*\Oc_{\FF_0}(1,2t-2)\;\; {\rm and} \;\; c_2(\cF_2)=c_2(\cF_2')=\xi \varphi^*\Oc_{\FF_0}(4,6t-2)- (5t+1)F.$$Moreover, since $M_1$ and $M_2$ are non-isomorphic line bundles with the same slope w.r.t. $\xi$ $$\mu(M_1)=\mu(M_2)=13t-3,$$then, by \cite[Lemma 4.2]{c-h-g-s},
$\cF_2$ and $\cF_2'$ are simple vector bundles, in particular indecomposable.  The family of extensions  \eqref{extension1} is of dimension $3b_0-3$ while the one as in \eqref{extension1'} has dimension $b_0+1$, which are different positive integers unless $b_0=2$.

\vspace{2mm}
%
\noindent {\bf Case L-M}: 
If we consider extensions using both line bundles of type $L_i$ and $M_j$, with $i, j=1,2$, one can easily see that for some of them we get only trivial extensions, precisely:
\vspace{3mm}

\noindent 
${\rm Ext}^1(M_1,L_1)\cong H^1(X_0,L_1-M_1)=H^1(X_0,-\xi+\varphi^*\Oc_{\FF_0}(3,t))\cong H^1(\FF_0, 0)=0$,
\vspace{3mm}

\noindent 
${\rm Ext}^1(L_1,M_2)\cong H^1(X_0,M_2-L_1)=H^1(X_0,-\xi+\varphi^*\Oc_{\FF_0}(0,3t))\cong H^1(\FF_0, 0)=0$,
\vspace{3mm}

\noindent ${\rm Ext}^1(M_1,L_2)\cong H^1(X_0,L_2-M_1)=H^1(X_0,-\xi+\varphi^*\Oc_{\FF_0}(0,3t))\cong H\cong H^1(\FF_0, 0)=0$,

\vspace{3mm}
\noindent   ${\rm Ext}^1(L_2,M_2)\cong H^1(X_0,M_2-L_2)=H^1(X_0,-\xi+\varphi^*\Oc_{\FF_0}(3,3t-b_0)\cong  H^1(\FF_0, 0)=0.$

\vspace{3mm}

\noindent
On the contrary, in the remaining possibilities  we get   non-trivial extensions and precisely: 
\begin{eqnarray*}
{\rm Ext}^1(L_1,M_1)&\cong& H^1(X_0,M_1-L_1)=H^1(X_0,\xi+\varphi^*\Oc_{\FF_0}(-3,-t))\cong H^1(\FF_0, \Eo(-3,-t)); 
\end{eqnarray*} 
an easy computation gives that $\dim({\rm Ext}^1(L_1,M_1))=1$  and thus there are   non-trivial extensions $\cF_3$
such that $c_1(\cF_3)=3\xi+\varphi^*\Oc_{\FF_0}(1,-t-2)$ and $c_2(\cF_3)=\xi \varphi^*\Oc_{\FF_0}9,3t-3)- (8t+1)F$.
\begin{eqnarray*}
{\rm Ext}^1(M_2,L_1)&\cong& H^1(X_0,L_1-M_2)=H^1(X_0,\xi+\varphi^*\Oc_{\FF_0}(0,-3t))\cong H^1(\FF_0, \Eo(0,-t)); 
\end{eqnarray*} 
in this case $\dim({\rm Ext}^1(M_2,L_1))=5b_0-5 \geqslant 5$  and so there are    non-trivial  extensions $\cF_4$ (because $b_0 \geqslant 2$, by \eqref{eq:rem:assAB}) such that $c_1(\cF_4)=\xi+\varphi^*\Oc_{\FF_0}(4,3t-2)$ and $c_2(\cF_4)=\xi \varphi^*\Oc_{\FF_0}(2,3t-1)+ (6t-4)F$.
\begin{eqnarray*}
{\rm Ext}^1(L_2,M_1)&\cong& H^1(X_0,M_1-L_2)=H^1(X_0,\xi+\varphi^*\Oc_{\FF_0}(0,-3t))\cong H^1(\FF_0, \Eo(0,-t)),
\end{eqnarray*} 
thus $\dim({\rm Ext}^1(L_2,M_1))=5b_0-5 \geqslant 5$  and thus there are    non-trivial  extensions $\cF_5$ with   $c_1(\cF_5)=3\xi+\varphi^*\Oc_{\FF_0}(-2,t-2)$ and $c_2(\cF_5)=\xi  \varphi^*\Oc_{\FF_0}(3,7t-3)+ (2-7t)F$.
\begin{eqnarray*}
{\rm Ext}^1(M_2,L_2)&\cong& H^1(X_0,L_2-M_2)=H^1(X_0,\xi+\varphi^*\Oc_{\FF_0}(-3,-t))\cong H^1(\FF_0, \Eo(-3,-t))\cong \mathbb C;
\end{eqnarray*} 
thus there are   non-trivial extensions $\cF_6$ with   $c_1(\cF_6)=\xi+\varphi^*\Oc_{\FF_0}(1,5t-2)$ and  $c_2(\cF_6)=\xi \varphi^*\Oc_{\FF_0}(2,3t-1)+ (t-1)F$.

\bigskip 

Previous computations show that there are Ulrich rank-$2$ vector bundles belonging to different moduli spaces, since their Chern classes are different.

As explained in Introduction, the aim of this paper is to give effective proofs for the Ulrich wildness of $3$--fold scrolls $(X_e, \xi)$ as above, for any $e \geqslant 0$, explicitely exhibiting irreducible components of moduli spaces of indecomposable Ulrich vector bundles of infinitely many ranks, together with 
all details as in {\bf Main Theorem} and {\bf Main Corollary}, namely Ulrich complexity of the $X_e$'s, generic smoothness and dimension of the modular components, etcetera.  
For these reasons,  in the sequel we will focus only on extensions of type \eqref{extension1}.  In particular, here  we prove the following theorem.

%

\begin{theo}\label{prop:rk 2 simple Ulrich vctB e=0;I}  Let  $(X_0, \xi) \cong \scrollcal{E_0}$ be a $3$-fold  scroll over $\FF_0$, with  $\mathcal E_0$ as in Assumptions \ref{ass:AB}.  Let $\varphi: X_0 \to \FF_0$ be the scroll map and $F$ be the $\varphi$-fibre. Then the moduli space of rank-$2$ vector bundles $\cU$ on $X_0$ which are Ulrich w.r.t. $\xi$ and  with Chern classes 
\begin{equation}\label{eq:chern20}
c_1(\cU)=2\xi+\varphi^*\Oc_{\FF_0}(1,b_0-2) \;\; {\rm and} \;\; c_2(\cU)=\xi \cdot \varphi^*\Oc_{\FF_0}(4,2b_0-2) +(2b_0-k_0-1)F, 
\end{equation} is not empty and it contains a generically smooth component $\mathcal M$ of dimension 
$$ \dim (\mathcal M) = 6b_0-3, $$whose general point  $[\cU]$ 
corresponds to a  special and slope-stable vector bundle, of slope 
\begin{equation}\label{eq:slope20}
\mu(\cU) = 8b_0 - k_0 -3,
\end{equation}  w.r.t. $\xi$  and where $b_0 \geqslant 2$ by \eqref{eq:rem:assAB}. 
\end{theo}

\begin{proof}  We consider  non--trivial extensions \eqref{extension1} as in {\bf Case L}; recall that $\dim {\rm Ext}^1(L_2,L_1)=4b_0-4\geqslant 4$,  being  $b_0 \geqslant 2$, 
and  moreover that a general vector bundle $\cF_1$ arising from a general extension as in \eqref{extension1} is Ulrich w.r.t. $\xi$. 

Moreover, since $\mu(L_1) = \mu(L_2) = 8b_0 - k_0 -3$, as computed in \S\,\ref{S:subs41}-{\bf Case L}, one has $\mu(\cF_1) = 8b_0 - k_0 -3$ and furthermore, since $L_1$ and $L_2$ are slope--stable, of the same slope and non--isomorphic line bundles, by \cite[Lemma\,4.2]{c-h-g-s}, $\mathcal F_1$ is simple, that is $h^0(\cF_1\otimes\cF_1^{\vee}) = 1$, in particular it is indecomposable.

We now want to show that  $h^2(\cF_1\otimes\cF_1^{\vee}) = 0 = h^3(\cF_1\otimes\cF_1^{\vee})$ and that 
$\chi(\cF_1\otimes\cF_1^{\vee})=-6b_0+3$. Tensoring \eqref{extension1} with $\cF_1^{\vee}$ we get
\begin{eqnarray}
\label{extension1tensorFdual}
0 \to L_1 \otimes \cF_1^{\vee}   \to \cF_1\otimes \cF_1^{\vee} \to  L_2 \otimes \cF_1^{\vee} \to 0.
\end{eqnarray} Dualizing \eqref{extension1} gives the following exact sequence
\begin{eqnarray}
\label{extension1dual}
0 \to L_2^{\vee}  \to \cF_1^{\vee} \to L_1^{\vee}\ \to 0
\end{eqnarray} Tensoring \eqref{extension1dual} with $L_1$ and $L_2$, respectively, gives
\begin{eqnarray}
\label{extension1dualL1}
0 \to L_2^{\vee}\otimes L_1(= \varphi^*\Oc_{\FF_0}(3,-b_0)) \to L_1 \otimes \cF_1^{\vee} \to {\cO}_{X_0} \to 0
\end{eqnarray} 
\begin{eqnarray}
\label{extension1dualL2}
0 \to  {\cO}_{X_0} \to L_2 \otimes \cF_1^{\vee} \to L_2 \otimes L_1^{\vee}(=\varphi^*\Oc_{\FF_0}(-3,b_0)) \to 0
\end{eqnarray} 

Because  $\cF_1$ is simple, then $h^0(X,  \cF_1\otimes \cF_1^{\vee} )=1$. The remaining cohomology $H^i(X,  \cF_1\otimes \cF_1^{\vee} )$  can be easily computed from the cohomology sequence associated to \eqref{extension1dualL1} and \eqref{extension1dualL2}.  Clearly $h^i(\cO_{X_0})=0$ if $i\geqslant 1$ and  $h^0(\cO_{X_0})=1$. It remains to compute 
$H^i( \varphi^*\Oc_{\FF_0}(3,-b_0))$ and $H^i( \varphi^*\Oc_{\FF_0}(-3,b_0)).$ 
 \begin{eqnarray}\label{eq:calcoliutili}
 H^i(X_0, \varphi^*\Oc_{\FF_0}(3,-b_0))&\cong &H^i(\FF_0, \Oc_{\FF_0}(3,-b_0))\cong H^i(\Pp^1, S^3(\Oc_{\Pp^1}\oplus \Oc_{\Pp^1})(-b_0)) \\
 &&  H^i(\Pp^1, \Oc^{\oplus 4}(-b_0))=\left\{
    \begin{array}{ll}
0 &\mbox{ if $i=0,2,3$}\\
     4b_0-4&   \mbox{ if  $i=1$}
    \end{array}
\right.\nonumber
 \end{eqnarray}
 Similarly
 \begin{eqnarray}\label{eq:calcoliutili2}
 H^i(X, \varphi^*\Oc_{\FF_0}(-3,b_0))&\cong &H^i(\FF_0, \Oc_{\FF_0}(-3,b_0))\cong H^{2-i}((\FF_0, \Oc_{\FF_0}(1,-2-b_0))\\
 &&H^{2-i}(\Pp^1, \Oc_{\Pp^1}^{\oplus 2}(-2-b_0)) =\left\{
    \begin{array}{ll}
0 &\mbox{ if $i=0,2,3$}\\
     2b_0+2&   \mbox{ if  $i=1$}
    \end{array}
\right.\nonumber
 \end{eqnarray} It thus follows that $h^2(\cF_1\otimes\cF_1^{\vee}) = 0 = h^3(\cF_1\otimes\cF_1^{\vee})$. From \eqref{extension1tensorFdual} we have that 
\begin{eqnarray*}\chi(\cF_1\otimes \cF_1^{\vee} )=\chi( L_1 \otimes \cF_1^{\vee})+\chi(L_2 \otimes \cF_1^{\vee})=-6b_0+4.
\end{eqnarray*}

Since $\cF_1$ is simple with $h^2(\cF_1\otimes \cF_1^{\vee}) = 0$, by \cite[Proposition 2.10]{c-h-g-s} there exists a smooth modular family 
containing the point $[\cF_1]$. Furthermore, since $\cF_1$ is Ulrich, with Chern classes 
$$c_1(\cF_1)=2\xi+\varphi^*\Oc_{\FF_0}(1,b_0-2) \;\; {\rm and} \;\; c_2(\cF_1)=\xi \cdot \varphi^*\Oc_{\FF_0}(4,2b_0-2) +(2b_0-k_0-1)F,$$as computed in \S\,\ref{S:subs41}--{\bf Case L}, 
the general point $[\cU]$ of the smooth modular family to which $[\cF_1]$ belongs corresponds to an Ulrich rank two vector bundles with same Chern classes as above, i.e. 
as in \eqref{eq:chern20}, as it follows from the facts that Ulrichness is an open condition (by semi-continuity) and that Chern classes are invariants 
on irreducible families.

We want to show that $\cU$ is also slope--stable w.r.t. $\xi$. By Theorem \ref{thm:stab}--(b) (cf. also \cite[Sect 3, (3.2)]{b}), if $ \cU $ were not a stable bundle, it would be presented as an extension of Ulrich line bundles on $X_0$. In such a case, by the classification of Ulrich line bundles given in Theorem \ref{prop:LineB} and all the possible extensions computed in \S\,\ref{S:subs41}-{\bf Case M} or {\bf Case L-M}, we see that the only possibilities for $\cU$ to arise as an extension of Ulrich line bundles should be extensions either \eqref{extension1} or  \eqref{extension1'}, by Chern classes reasons. 
In both cases the dimension of (the projectivization) of the corresponding families of extensions are either $4b_0 -5$ or $2b_0+1$. On the other hand, by semi-continuity on the smooth modular family, 
one has 
$$h^j(\cU\otimes \cU^{\vee}) = h^j(\cF_1\otimes \cF_1^{\vee}) = 0, \; 2 \leqslant j \leqslant 3, \; {\rm and} \; h^0(\cU\otimes \cU^{\vee}) = h^0(\cF_1\otimes \cF_1^{\vee}) = 1,$$thus
$$ h^1(\cU\otimes \cU^{\vee}) = 1 - \chi (\cU\otimes \cU^{\vee} ) = 1 - \chi(\cF_1\otimes \cF_1^{\vee}) = h^1(\cF_1\otimes \cF_1^{\vee}) = 6b_0-3,$$as computed above.  In other words, 
the smooth modular family whose general point is $[\cU]$ is of dimension $6b_0 -3$, which is bigger than $4b_0-5$ and $2 b_0 +1$,  for any $b_0 \geqslant 2$. This shows 
that $[\cU]$ general corresponds to a stable, and so also slope-stable bundle  (cf. Theorem \ref{thm:stab}-(c) above).

By slope-stability of $\cU$, we deduce that the moduli space of rank two Ulrich bundles with Chern classes as in \eqref{eq:chern20} is not empty and it contains a generically smooth 
component $\mathcal M$, of dimension $6b_0-3$ whose general point $[\cU] \in \mathcal M$ is also slope-stable, whose slope w.r.t. $\xi$ is 
$\mu(\cU) = \frac{c_1(\cU) \cdot \xi^2}{2} = 8b_0 - k_0 -3$,  as $c_1( \cU ) = c_1(\cF_1)$.

Finally, note that  
\begin{eqnarray*}K_{X_0}+4\xi&=&-2\xi+\varphi^*\Oc_{\FF_0}(-2,-2)+\varphi^*\Oc_{\FF_0}(3,b_0)+4\xi=2\xi+\varphi^*\Oc_{\FF_0}(1,b_0-2)\\&=&c_1(\cF_1) = c_1(\cU).
\end{eqnarray*} This, together with the fact that $\cU$ is of rank two, gives 
$$\cU^{\vee} \cong \cU(- c_1(\cU)) = \cU (-K_{X_0} - 4 \xi),$$i.e. that $\cU^{\vee}(K_{X_0} + 4 \xi) \cong \cU$ in other words 
$\cU$ is isomorphic to its Ulrich dual bundle, i.e. $\cU $ is special, as stated. 
\end{proof}

If in particular we set $b_0=3$ then, from  Theorem \ref{prop:rk 2 simple Ulrich vctB e=0;I}, one gets 
$$c_1(\cU) = 2 \xi + \varphi^*\Oc_{\FF_0}(1, 1), \;\dim (\mathcal M ) = 15 \; {\rm and} \; 
\mu(\cU) = 21-k_0,$$which is what was obtained in \cite[Proposition 5.7]{f-lc-pl} for $k_0=7,8,9,10$.

\begin{rem}\label{rem:rapagnetta} {\normalfont Besides the reasons stated before Theorem \ref{prop:rk 2 simple Ulrich vctB e=0;I}, other motivations which  explain why the previous result focuses on rank two Ulrich bundles arising 
from (deformations of) extensions as in \eqref{extension1}, are the following. 

\noindent
(i) First of all, Theorem \ref{prop:LineB} shows that Ulrich line bundles 
$M_1$ and $M_2$ are {\em sporadic}, i.e. they exist only when $(b_0, k_0) = (2t, 3t)$, for some integer $t \geqslant 1$; the same sporadic behaviour occurs therefore for 
extensions in {\bf Case M} and in {\bf Case L-M} as in \S\,\ref{S:subs41}. Furthermore, such extensions give rise to components of different moduli spaces, 
since in any case Chern classes are different from those in \eqref{eq:chern20}. 

\noindent
(ii) Concerning extensions \eqref{extension1} and \eqref{extension1'} in {\bf Case L}, we have the following:

\begin{claim}\label{cl:merd} Deformations of bundles $\mathcal F_1$, arising as non trivial extensions in \eqref{extension1}, 
and of bundles $\mathcal F_1'$, arising as non trivial extensions in \eqref{extension1'}, give rise to the same modular component $\mathcal M$ as in Theorem \ref{prop:rk 2 simple Ulrich vctB e=0;I}.
\end{claim}
 
To prove the Claim, we have benefitted of useful discussions and reference advices given to us by A. Rapagnetta and we thank him for this. 
  
As a general fact, recall that the moduli space $\mathcal M$ of semistable bundles with given rank and Chern classes on a smooth projective variety $X$ is constructed as a GIT quotient 
of an open subscheme $\mathcal R$ of a suitable Quot-scheme $\mathcal Q$, modulo the action of a group $G = {\rm PGL}(k, \mathbb C)$ for a suitable integer $k >>0$ (cf. e.g. 
\cite[Section\;4.3]{HL}). Here we focus on the rank two case. 

If we denote by $\pi: \mathcal R \to \mathcal M$ the quotient map and if $\rho \in \mathcal R$ is a closed point, then $\pi(\rho) \in \mathcal M$ is a closed point iff the $G$-orbit 
$G \cdot \rho$ (with notation of left action of $G$ as in \cite{Drezet}) is closed in $\mathcal R$ iff the corresponding quotient bundle $\mathcal F_{\rho}$ on $X$ is {\em polystable}, namely it is either stable (in such a case $Aut(\mathcal F_{\rho}) \cong \mathbb C^*$) or it is (strictly) polystable of the form $\mathcal F_{\rho} = \mathcal L_1 \oplus \mathcal L_2$, i.e. $\mathcal F_{\rho}$ decomposable 
and the two line bundles $\mathcal L_1$ and $\mathcal L_2$ are not isomorphic with the same Hilbert polynomial (in such a case $Aut(\mathcal F_{\rho}) \cong \mathbb C^* \times \mathbb C^*$) or 
it is (strictly) polystable of the form $\mathcal F_{\rho} = \mathcal L^{\oplus 2}$ (in which case $Aut(\mathcal F_{\rho}) \cong {\rm GL}(2 , \mathbb C)$). 

If instead $\mathcal F_{\rho}$ is semistable but not polystable (e.g. as any non trivial extension as in \eqref{extension1}) then $\mathcal F_{\rho}$ arises as 
a non trivial extension of two line bundles $\mathcal L_1$ and $\mathcal L_2$ with same Hilbert polynomial and the Jordan-H\"older graded object of $\mathcal F_{\rho}$ is $gr^{JH} (\mathcal F_{\rho}) = 
\mathcal L_1 \oplus \mathcal L_2$ which is therefore (strictly) polystable and the corresponding point $\pi(\rho) \in \mathcal M$ is represented by the class $[\mathcal L_1 \oplus \mathcal L_2]$, namely closed points in $\mathcal M$ are in bijection with {\em S-equivalence} classes of semistable bundles (cf. \cite[Def.\;1.5.3 and Lemma\;4.1.2]{HL}). 

If $\rho \in \mathcal R$ is a point for which the corresponding bundle $\mathcal F_{\rho}$ on $X$ is polystable, namely the orbit $G \cdot \rho$ is closed in $\mathcal R$, and if $\mathcal R$ 
is locally (in the analytic topology) irreducible around $\rho \in \mathcal R$, then also $\mathcal M$ is locally (in the analytic topology) irreducible around $\pi(\rho) \in \mathcal M$. Indeed,  
since $G$ is a connected group, the action of $G$ preserves the irreducible components of $\mathcal R$ and, moreover, by the GIT action if a point $\rho' \in \mathcal R$ belongs 
to the closure $\overline{G \cdot \rho^{''}}$ in $\mathcal R$ of the $G$-orbit $G \cdot \rho^{''}$ 
of a point $\rho^{''} \in \mathcal R$ then $G \cdot \rho' \subseteq \overline{G \cdot \rho^{''}}$. 

\begin{proof}[Proof of Claim \ref{cl:merd}] With preliminaries as above, take $X = X_0$ and denote by $\rho_{1,2} \in \mathcal R$ the point corresponding to the bundle 
$L_1 \oplus L_2$ as in \S\;\ref{S:subs41}-{\bf Case L}, which is the same  Jordan-H\"older graded object of the bundles $\mathcal F_1$ and $\mathcal F_1'$ arising as non trivial extensions in \eqref{extension1} and \eqref{extension1'}, respectively. Since 
$\rho_{1,2}$ corresponds to a (strictly) polystable bundle on $X_0$, then $G \cdot \rho_{1,2}$ is a closed $G$-orbit and we can apply the same reasoning as above. Therefore, if we show that $\rho_{1,2}$ is a smooth point of $\mathcal R$, then $\mathcal R$ is locally (in the analytic topology) irreducible around $\rho_{1,2}$ and therefore the same occurs for $\mathcal M$ 
around the point $\pi(\rho_{1,2})$ which implies that the modular component  arising from (deformations of) extensions $\mathcal F_1$ as in \eqref{extension1} 
is the same component arising from (deformations of) extensions $\mathcal F'_1$ as in \eqref{extension1'}. 

Therefore, we need to show that $\mathcal R$ is smooth at the point $\rho_{1,2} $. To do this, notice first that certainly it exists an irreducible component, say $R_{1,2}$, of $\mathcal R$ passing through the point $\rho_{1,2} $ whose dimension is 
$\dim(G) + h^1 (\mathcal F_1 \otimes \mathcal F_1^{\vee})$, where $\mathcal F_1$ is a non-trivial extension as in \eqref{extension1}. Indeed, as in the proof of Theorem \ref{prop:rk 2 simple Ulrich vctB e=0;I}, $\mathcal F_1$ admits an irreducible smooth modular family of dimension $h^1 (\mathcal F_1 \otimes \mathcal F_1^{\vee})$ which gives 
rise to an irreducible parameter space $\mathcal P \subset \mathcal R$, parametrizing all quotients which are isomorphic to the bundles of such an irreducible smooth modular family, which is therefore of dimension $\dim(\mathcal P) = \dim(G) + h^1 (\mathcal F_1 \otimes \mathcal F_1^{\vee})$ and whose closure $\overline{\mathcal P}$ in $\mathcal R$ 
contains the point $\rho_{1,2}$; thus $\rho_{1,2}  \in \overline{\mathcal P} \subseteq R_{1,2}$ so one has 
\begin{equation}\label{eq:merd}
\dim(R_{1,2}) \geqslant \dim (\overline{\mathcal P}) = \dim(G) + h^1 (\mathcal F_1 \otimes \mathcal F_1^{\vee}).
\end{equation} Since $\rho_{1,2} \in \mathcal R$ corresponds to a (strictly) polystable bundle, from above, its $G$-orbit $G\cdot \rho_{1,2}$ is closed in $\mathcal R$; 
thus by Luna's \'etale slice theorem (cf. e.g. \cite[Thm.\;4.2.12]{HL}) there exists a locally closed subscheme $S \subset \mathcal R$ passing through $\rho_{1,2}$ 
which is $G_{\rho_{1,2}}$-invariant, where $G_{\rho_{1,2}}$ denotes the stabilizer in $G$ of the point $\rho_{1,2}$, and such that the multiplication map 
$G \times S \to \mathcal R$ induces a $G$-equivariant \'etale morphism $\psi : G \times_{G_{\rho_{1,2}}} S \to \mathcal R$, where 
$$G \times_{G_{\rho_{1,2}}} S := (G \times S)//G_{\rho_{1,2}}$$as in \cite[Sect.\;1.1]{Drezet}. Because $G = {\rm PGL}(k, \mathbb C)$, then one has  
$$G_{\rho_{1,2}} \cong \frac{Aut_{\mathcal R} (\rho_{1,2})}{\mathbb C^*} \cong \frac{Aut (L_1 \oplus L_2)}{\mathbb C^*} = \frac{H^0((L_1 \oplus L_2) \otimes (L_1 \oplus L_2)^{\vee})}{\mathbb C^*} \cong \frac{\mathbb C^* \times \mathbb C^*}{\mathbb C^*} \cong \mathbb C^*,$$
since $h^0(L_1 - L_2) = h^0(L_2 - L_1) = 0$ as it follows from \eqref{extension1dualL1}, \eqref{extension1dualL2}, \eqref{eq:calcoliutili}, \eqref{eq:calcoliutili2}. Therefore, we have an isomorphism 
of tangent spaces 
$$T_{\rho_{1,2}} (\mathcal R) \cong T_{(Id_G,\rho_{1,2})} (G \times_{G_{\rho_{1,2}}} S),$$the latter of dimension
$$\dim\left(T_{(Id_G,\rho_{1,2})} (G \times_{G_{\rho_{1,2}}} S) \right) = \dim\left(T_{\rho_{1,2}} (S)\right) + \dim(G) - \dim(\mathbb C^*)$$as it follows from 
\cite[Proposition 4.9-(6)]{Drezet} and the facts $\dim\left(T_{Id_G}(G) \right) = \dim(G)$ and \linebreak $\dim\left(T_{Id_G}(G_{\rho_{1,2}}) \right)= \dim(\mathbb C^*)$. In particular, 
\begin{equation}\label{eq:merd2}
\dim\left( T_{\rho_{1,2}} (\mathcal R)\right) = \dim(\left( T_{\rho_{1,2}} (S)\right) + \dim(G) -1.
\end{equation}On the other hand, from the proof of \cite[Proposition 1.2.3]{OGrady}, one has 
\begin{equation}\label{eq:merd3}
\dim\left( T_{\rho_{1,2}} (S)\right) = {\rm Ext^1}( L_1 \oplus L_2, L_1 \oplus L_2) = H^1((L_1 \oplus L_2) \otimes (L_1 \oplus L_2)^{\vee}).
\end{equation}Using \eqref{extension1dualL1}, \eqref{extension1dualL2} and the cohomological computations as in \eqref{eq:calcoliutili}, \eqref{eq:calcoliutili2} 
and the fact that $L_1 -L_1 = L_2 - L_2 = \mathcal O_{X_0}$, from 
\eqref{eq:merd3} we have $\dim\left( T_{\rho_{1,2}} (S)\right) = 6b_0 -2$ thus, from \eqref{eq:merd2}, we get 
$\dim\left( T_{\rho_{1,2}} (\mathcal R)\right)  = 6b_0 - 3 + \dim(G)$. On the other hand, from the proof of Theorem \ref{prop:rk 2 simple Ulrich vctB e=0;I}, one has 
$h^1(\mathcal F_1 \otimes \mathcal F_1^{\vee} ) = 6b_0 -3$, where $\mathcal F_1$ a non trivial extension as in \eqref{extension1}. Therefore, from above
$\dim\left( T_{\rho_{1,2}} (\mathcal R)\right) = h^1(\mathcal F_1 \otimes \mathcal F_1^{\vee} ) + \dim(G)$ which, together with \eqref{eq:merd}, implies that 
$$\dim(R_{1,2}) = \dim\left( T_{\rho_{1,2}} (\mathcal R)\right),$$i.e. that $\mathcal R$ is smooth at $\rho_{1,2}$ and so that $R_{1,2}$ is the unique irreducible component 
of $\mathcal R$ passing through $\rho_{1,2}$, which completes the proof of the claim. 
\end{proof}
}
\end{rem}

\bigskip

%
%
 \subsection{Rank-$2$ Ulrich vector bundles  on $3$-fold scrolls over $\FF_e$, $e>0$}\label{subs42} 
  In this section, we will focus on the case $e >0$.

\begin{theo}\label{thm:rk 2 vctB e>0}    Let  $(X_e, \xi) \cong \scrollcal{E_e}$ be a $3$-fold  scroll over $\FF_e$, with $e> 0$,  and $\mathcal E_e$ be as in Assumptions \ref{ass:AB}. Let $\varphi: X_e \to \FF_e$ be the scroll map and $F$ be the $\varphi$- fibre. Then the moduli space of  rank two vector bundles $ \cU $ on $X_e$,  which are Ulrich w.r.t. $\xi$, with Chern classes
\begin{equation}\label{eq:chern2e}
c_1( \cU )=2\xi+\varphi^*\Oc_{\FF_e}(1,b_e-e-2) \;\; {\rm and} \;\;  c_2( \cU ) = \xi \varphi^*\Oc_{F_e}(4,2b_e-e-2) + (2b_e-3e-k_e-1) F,
\end{equation}
is not empty and it contains a generically smooth component $\mathcal M$ of dimension 

$$ \dim(\mathcal M) = 6b_e-9e-3,$$
whose general point $[ \cU ]$ 
corresponds to a  special and slope-stable vector bundle of slope w.r.t. $\xi$ \linebreak
\begin{eqnarray}\label{eq:slope2e}
\mu( \cU ) = 8 b_e - k_e - 12 e - 3.
\end{eqnarray}
\end{theo}
\begin{proof}  By \cite[Theorem 3.4]{a-c-mr}, we know that there exist  rank two vector bundles $\mathcal H_1$ on $\FF_e$, which are  Ulrich with respect to  $c_1({\mathcal E}_e) =\Oc_{\FF_e}(3, b_e)$,  given by extensions
\begin{eqnarray}
\label{extensionsuFe}
0 \to \Oc_{\FF_e}(3 , b_e-1)  \to \cH_1 \to I_Z\otimes \Oc_{\FF_e}(4, 2b_e-1-e)  \to 0, 
\end{eqnarray} where $Z$ is a general zero-dimensional subscheme of $\FF_e$ of length $\ell(Z)=2b_e-3e$.
Such a bundle $\cH_1$ is stable,  cf. \cite[Remark 3.7]{a-c-mr},  hence simple, that is $h^0(\cH_1 \otimes \cH_1^{\vee})=1$, and indecomposable.

Set $\cH:=  \cH_1(-c_1({\mathcal E}_e))$, which is stable being a twist of a stable vector bundle,  so it is also simple, i.e. 
$h^0(\cH \otimes \cH^{\vee})=1$. By Theorem \ref{pullback} the vector bundle 
$ \cV : = \xi \otimes \varphi^*(\cH)$ is a rank two vector bundle on  $X_e$ which is Ulrich with respect to $\xi$.  

From \eqref{extensionsuFe} we see that $c_1(\cH)= \Oc_{\FF_e}(1 , b_e-e-2)$ and $c_2(\cH)=2b_e-3e-1$. Easy Chern classes computations give that  
\begin{eqnarray*}
c_1( \cV ) = 2\xi+\varphi^*\Oc_{\FF_e}(1,b_e-e-2) \;\; {\rm and} \;\; c_2( \cV ) = \xi \varphi^*\Oc_{F_e}(4,2b_e-e-2) + (2b_e-3e-k_e-1) F.
\end{eqnarray*} 
Moreover by Theorem \ref{thm:stab}--(b) (cf.\;also \cite[Sect 3, (3.2)]{b}) such a bundle $ \cV $  is stable,  so also slope-stable by Theorem \ref{thm:stab}-(c),  since there are no Ulrich line bundles on $(X_e, \xi) $ as it follows from Theorem \ref{prop:LineB}. 

Our next step is to compute the cohomology groups $H^i(X_e,  \cV \otimes\cV^{\vee} )$ \,  for $i=0,1,2,3$. Because $H^i(X_e,   \cV \otimes\cV^{\vee} ) = H^i(X_e,  \varphi^*(\cH\otimes \cH^{\vee}))\cong H^i(\FF_e, \cH\otimes \cH^{\vee})$ we will focus on computations of $H^i(\FF_e, \cH\otimes \cH^{\vee})$, $i=0,1,2,3$. 

First of all $ h^3(\FF_e, \cH\otimes \cH^{\vee}) = 0$, as $\FF_e$ is a surface, and $ h^0(\FF_e, \cH\otimes \cH^{\vee}) = 1$, as $\cH$ is simple. 
For the other cohomology groups, we  tensor \eqref{extensionsuFe} with $-c_1({\mathcal E}_e)=\Oc_{{\mathbb F}_e}(-3,-b_e)$ and we get
\begin{eqnarray}
\label{extensionsuFe twisted}
0 \to \Oc_{\FF_e}(0 , -1)  \to \cH \to I_Z\otimes \Oc_{\FF_e}(1, b_e-1-e)  \to 0.
\end{eqnarray} 
Because $\cH$ is of rank $2$ and $c_1(\cH)=\Oc_{\FF_e}(1 , b_e-2-e)$, we have that 
$ \cH^{\vee}\cong \cH\otimes \Oc_{\FF_e}(-1 , -b_e+2+e)$ and thus after tensoring \eqref{extensionsuFe twisted} with $\cH^{\vee} $ we get
\begin{eqnarray}
\label{extensionsuFe twisted Hdual}
0 \to\cH \otimes  \Oc_{\FF_e}(-1 , 1+e-b_e)  \to \cH \otimes \cH^{\vee} \to I_Z\otimes \cH\otimes \Oc_{\FF_e}(0, 1)  \to 0.
\end{eqnarray} 
In order to compute the cohomology groups of $\cH \otimes  \Oc_{\FF_e}(-1 , 1+e-b_e)$ we will use the short exact sequence \eqref{extensionsuFe twisted} twisted with $\Oc_{\FF_e}(-1 , 1+e-b_e)$
which gives 
\begin{eqnarray}
\label{extensionsuFe twisted O(-1,1+e-b)}
0 \to  \Oc_{\FF_e}(-1 , e-b_e)  \to \cH \otimes  \Oc_{\FF_e}(-1 , 1+e-b_e)\to I_Z
\to 0.
\end{eqnarray} 
From this we can easily see that $H^0(\cF\otimes  \Oc_{\FF_e}(-1 , 1+e-b_e))=0=H^2(\cF\otimes  \Oc_{\FF_e}(-1 , 1+e-b_e))$ and $h^1(\cF\otimes  \Oc_{\FF_e}(-1 , 1+e-b_e))=2b_e-3e-1$.

Our next task is to compute the  cohomology groups of $I_Z\otimes \cH \otimes \Oc_{\FF_e}(0, 1)$. We tensor the sequence 
\begin{eqnarray*}
0 \to  I_Z \to \Oc_{\FF_e} \to \Oc_Z  \to 0 
\end{eqnarray*}
with $\cH \otimes \Oc_{\FF_e}(0, 1)$  and $\Oc_{\FF_e}(1 , b_e-e)$, respectively, and we get \begin{eqnarray}
\label{exact4}
0 \to  I_Z\otimes \cH \otimes \Oc_{\FF_e}(0, 1)  \to \cH \otimes  \Oc_{\FF_e}(0 , 1)\to (\cH \otimes  \Oc_{\FF_e}(0 , 1))_{|Z}
\to 0
\end{eqnarray}
\begin{eqnarray}
\label{exact5}
0 \to  I_Z\otimes \Oc_{\FF_e}(1,  b_e-e)  \to \Oc_{\FF_e}(1,  b_e-e)\to \Oc_{Z}
\to 0.
\end{eqnarray}
 We  tensor \eqref{extensionsuFe twisted} with $\Oc_{\FF_e}(0 , 1
 )$ and we get
 \begin{eqnarray}
\label{exact6}
0 \to  \Oc_{\FF_e}  \to \cH \otimes  \Oc_{\FF_e}(0 , 1)\to   I_Z\otimes  \Oc_{\FF_e}(1, b_e-e) 
\to 0.
\end{eqnarray}
Now use the cohomology sequence associated to the short exact sequences \eqref{exact4}, \eqref{exact5} and \eqref{exact6}.
Note that 
\begin{align*}
h^0(\FF_e,\Oc_{\FF_e}(1, b_e-e))\cong  h^0(\Pp^1,\Oc_{\Pp^1}(b_e-e)\oplus \Oc_{\Pp^1}(b_e-2e))= 2b_e-3e+2, 
\end{align*}
 and  
\begin{align*}
h^i(\FF_e,\Oc_{\FF_e}(1, b_e-e))\cong  h^i(\Pp^1,\Oc_{\Pp^1}(b_e-e)\oplus \Oc_{\Pp^1}(b_e-2e))=0 \quad \text{for $i=1,2$ }
\end{align*} 
because  Assumptions \ref{ass:AB} forces $b_e \geqslant 3e +2$). Notice that $\dim (|\Oc_{\FF_e}(1, b_e-e)|)=2b_e-3e+1$,  so $h^0( I_Z\otimes  \Oc_{\FF_e}(1, b_e-e))=2$ , being $Z$  general of length $\ell(Z)=2b_e-3e$. Therefore from 
\eqref{exact5} it follows that $h^i( I_Z\otimes  \Oc_{\FF_e}(1, b_e-e))=0$ for $i=1,2$. Now using \eqref{exact6} it follows that $h^0(\FF_e,\cH \otimes \Oc_{\FF_e}(0, 1))=3$ and  $h^i(\FF_e,\cH \otimes \Oc_{\FF_e}(0, 1))=0$ for $i=1,2$. Thus $h^2(I_Z\otimes \cH\otimes \Oc_{\FF_e}(0, 1))=0$ and this, combined with  the cohomology sequence associated to \eqref{extensionsuFe twisted Hdual}, gives that $h^2( \cH\otimes   \cH^{\vee})=0$.

Thus from \eqref{exact4}, since $h^0(\FF_e, (\cH\otimes  \Oc_{\FF_e}(0 ,1))_{|Z}
)=2(2b_2-3e)$ and $h^i(\FF_e, \cH\otimes  \Oc_{\FF_e}(0 , 1)_{|Z}
)=0$ for $i=1,2$, it follows that $\chi(I_Z\otimes \cH\otimes \Oc_{\FF_e}(0, 1) )=\chi(\cH\otimes \Oc_{\FF_e}(0 , 1))-\chi( \cH\otimes  \Oc_{\FF_e}(0 , 1)_{|Z})=3-4b_e+6e.$
From the cohomology sequence associated to \eqref{extensionsuFe twisted Hdual} it follows that $\chi( \cH\otimes   \cH^{\vee} )=\chi(I_Z\otimes \cH\otimes \Oc_{\FF_e}(0, 1) )+\chi(\cH \otimes  \Oc_{\FF_e}(-1 , 1+e-b_e) )=3-4b_e+6e-2b_2+3e+1=4-6b_e+9e$ and thus $h^1( \cH\otimes   \cH^{\vee})=1-\chi( \cH\otimes   \cH^{\vee})=6b_e-9e-3$.

As already observed,  $H^i(X_e, \cV \otimes\cV^{\vee} )\cong H^i(\FF_e, \mathcal H\otimes \mathcal H^{\vee})$, $i=0,1,2,3$, 
hence the above computations  give us the dimensions of all the cohomology groups $H^i(X_e, \cV \otimes\cV^{\vee})$, $i=0,1,2,3$.

By \cite[Proposition 2.10]{c-h-g-s}, since $h^2( \cV  \otimes \cV^{\vee}) = 0$ and $\cV$ is simple, it follows that $\cV$ admits 
a smooth modular family giving rise to a component $\mathcal M$ of the moduli space of rank two Ulrich vector bundles with Chern classes  
$$c_1( \cV)=2\xi+\varphi^*\Oc_{\FF_e}(1,b_e-2-e)\;{\rm and}\; c_2( \cV )=\xi \varphi^*\Oc_{\FF_e}(4,b_e-2-e) +(2b_e-3e-k_e-1)F.$$Moreover, since $\cV$ is 
slope-stable, it corresponds to a smooth point $[\cV]$ of such a component $\mathcal M$ such that, 
$\dim (\mathcal M) = h^1(\cV  \otimes \cV^{\vee} )  = 6 b_e-9e-3$, as stated. 

Since Ulrichness, slope-stability, simplicity are open  conditions as well as Chern classes are constant for vector bundles varying in $\mathcal M$, it follows that  all the properties satisfied by $\mathcal V$ hold true  for the general member $[\cU ] \in \mathcal M$, 
in particular $\mathcal M$ is also generically smooth.

Note further that  $K_{X_e}+4\xi=-2\xi+\varphi^*\Oc_{\FF_e}(-2,-2-e)+\varphi^*\Oc_{\FF_e}(3,b_e)+4\xi=2\xi+\varphi^*\Oc_{\FF_0}(1,b_e-e-2)=c_1(\cV)$, 
thus $ \cU $ is  a special Ulrich bundle.  Finally, the slope of $ \cU $ with respect to $\xi$ is 
\begin{eqnarray*}
\mu( \cU ) = \frac{c_1( \cU ) \cdot \xi^2}{2} = \frac{(2\xi+\varphi^*\Oc_{\FF_e}(1,b_e-e-2) ) \cdot \xi^2}{2} = 8b_e - k_e - 12 e - 3.
\end{eqnarray*}
\end{proof}

\begin{rem*} (i) In view of Theorem \ref{pullback}, together with the fact that $c_1(\mathcal E_e) = 3 C_e + b_e f$ is very ample on $\FF_e$ (cf. \eqref{eq:rem:assAB}), the bundle 
$\cH_1$ is the rank two bundle on $\FF_e$ which is Ulrich w.r.t. $c_1(\mathcal E_e)$ and which gives rise to the rank two vector bundle $\mathcal V = \xi \otimes \varphi^*(\mathcal H_1(-c_1(\mathcal E_e)))$ as in the proof of Theorem \ref{thm:rk 2 vctB e>0}, which is Ulrich  w.r.t. $\xi$ on $X_e$. 

\smallskip

\noindent 
(ii) In \cite[Theorems 5.8, 5.9]{f-lc-pl} it was shown the existence of stable rank two Ulrich vector bundle w.r.t. $\xi$ on $X_0$ and $X_1$ of low degree. In  \cite[Corollary 5.17]{hoc} it was shown the existence of rank two Ulrich vector bundle on $\mathbb P(\mathcal E) \to \FF_e$ w.r.t. different very ample polarizations $D= \pi^*(A)+\xi$ with $A$ such that $rk(\mathcal E)A+c_1(\mathcal E)$ is also very ample, $rk(\mathcal E) \geqslant 2$ and $\pi: \mathbb P(\mathcal E) \to \FF_e$ the natural projection. But nothing was said about their moduli spaces.

\smallskip

\noindent
(iii) If we set $e=1$, $b_e=5$ and $k_1=10, 11$ we get $3$-fold scrolls $X_1$ of degree either 
$11$ or $10$, as $\deg(X_1)=21-k_1$. Applying Theorem \ref{thm:rk 2 vctB e>0} to such $3$-folds, one gets 
\begin{eqnarray*}
c_1(\cU) = 2\xi+\varphi^*\Oc_{\FF_1}(1,2), \; \dim (\mathcal M) =18 \; {\rm and} \; \mu(\cU) = 40-k_1 -12 -3 = 25 -k_1.
\end{eqnarray*} These scrolls have  been considered in \cite[Theorem 5.9]{f-lc-pl} where it was only shown the existence of rank two Ulrich bundles on them, but nothing was said about their moduli space.
\end{rem*}

%
%

\section{Higher rank Ulrich vector  bundles on $3$-fold scrolls over $\FF_e$} \label{Ulrich higher rk  vb} In this section we will construct higher rank  slope-stable  Ulrich vector bundles on $X_e$, where $e \geqslant0$. We will moreover compute the dimensions of the modular components arising from the constructed bundles,  completely proving the {\bf Main Theorem} and  the {\bf Main Corollary} stated in  the Introduction.

To do so, we will use  Theorems \ref{prop:LineB}, \ref{prop:rk 2 simple Ulrich vctB e=0;I}, 
\ref{thm:rk 2 vctB e>0}, as well as   inductive procedures and deformation arguments.

%
%

 \subsection{Higher rank  Ulrich vector bundles on $3$-fold scrolls over $\FF_0$}\label{S:subs51} We will first concentrate on the case $e=0$.
From Theorem \ref{prop:LineB} we know that, under Assumptions \ref{ass:AB}, the case $e=0$ is the only case where Ulrich line bundles on 
$(X_0, \xi)$ actually exist. We will focus on line bundles
\begin{eqnarray}\label{eq:Li}
L_1 =\xi+\varphi^*\Oc_{\FF_0}(2,-1) \;\; \mbox{and its Ulrich dual} \;\;  L_2 =\xi+\varphi^*\Oc_{\FF_0}(-1,b_0-1),
\end{eqnarray}as in Theorem \ref{prop:LineB}-(i), which are Ulrich w.r.t. $\xi$ on $X_0$. 

Recalling computations in \S\;\ref{S:subs41}-{\bf Case L} and the fact that  $b_0 \geqslant 2$ by \eqref{eq:rem:assAB},  we have that: 
\begin{eqnarray}\label{eq:dimextLi}
\dim ({\rm Ext}^1 (L_2, L_1)) &=&  h^1(L_1 - L_2) = 4 (b_0-1) \geqslant4, \;\mbox{whereas} \\
 \dim ({\rm Ext}^1 (L_1, L_2)) &= & h^1(L_2 - L_1) = 2(b_0+1) \geqslant6. \nonumber
 \end{eqnarray} 

In Theorem \ref{prop:rk 2 simple Ulrich vctB e=0;I} we used such extensions to construct rank-$2$ Ulrich vector bundles; to construct higher rank Ulrich bundles on $X_0$ 
we proceed with an iterative strategy as follows.

Set $\cG_1 := L_1$; from \eqref{eq:dimextLi} the general $[\cG_2] \in {\rm Ext}^1 (L_2, \cG_1) = {\rm Ext}^1 (L_2, L_1)$ is associated to 
a non-splitting extension 
\begin{equation}\label{eq:1r1}
0 \to \cG_1=L_1 \to \cG_2 \to L_2 \to 0,
\end{equation}where $\cG_2$ is a rank-$2$ Ulrich and simple vector bundle on $X_0$ with 
$$c_1(\cG_2) = 2 \xi + \varphi^*\Oc_{\FF_0}(1,b_0-2)$$ (cf. \eqref{extension1}, where $\cG_2 := \cF_1$ therein, 
and see the proof of Theorem \ref{prop:rk 2 simple Ulrich vctB e=0;I}). 
If, in the next step,  we considered further extensions ${\rm Ext}^1 (L_2, \cG_2)$, it is easy to see that the dimension of such an extension space drops by one with respect to 
that of ${\rm Ext}^1 (L_2, \cG_1)$. Therefore, proceeding in this way,  after finitely many steps we would have only splitting bundles in ${\rm Ext}^1 (L_2, \cG_r)$ for any 
 $r \geqslant  r_0$, for some positive integer $r_0$. 

To avoid this, similarly as in \cite[\S\;4]{cfk1}, we proceed by taking extensions
 \[0 \to \cG_2 \to \cG_3 \to L_1 \to 0,\;\; 0 \to \cG_3 \to \cG_4 \to L_2 \to 0,\; \ldots , 
  \]
  and so on, that is, defining
  \begin{equation} \label{eq:jr}
    \epsilon_r: =
    \begin{cases}
      1, & \mbox{if $r$ is odd}, \\
      2, & \mbox{if $r$ is even},
    \end{cases}
  \end{equation}
  we take successive extensions $[\cG_{r}] \in \Ext^1(L_{\epsilon_{r}},\cG_{r-1})$ for all $r \geqslant2$:
  \begin{equation}\label{eq:1}
0 \to \cG_{r-1} \to \cG_{r} \to L_{\epsilon_{r}} \to 0.
 \end{equation}

The fact that we can always take   {\em non--trivial} such extensions will be proved in a moment in Corollary \ref{cor:Corollario al Lemma 4.2} below. 
In any case all vector bundles $\cG_{r}$, recursively defined as in \eqref{eq:1}, are of rank $r$ and 
Ulrich w.r.t. $\xi$, since extensions of Ulrich bundles w.r.t. $\xi$ are again Ulrich w.r.t. $\xi$. The first Chern class of $\cG_r$  is given by 
\begin{equation} \label{eq:c1rcaso0}
    c_1(\cG_r): =
    \begin{cases} 
      r \xi + \varphi^*\Oc_{\FF_0}(3, b_0 -3) + \varphi^*\Oc_{\FF_0}\left(\frac{r-3}{2}, \frac{(r-3)}{2}(b_0-2) \right), & \mbox{if $r$ is odd}, \\
      r \xi + \varphi^*\Oc_{\FF_0}\left(\frac{r}{2}, \frac{r}{2}(b_0-2)\right), & \mbox{if $r$ is even}. 
    \end{cases}
  \end{equation}Thus, for any $r \geqslant1$, its slope w.r.t. $\xi$ is  
	\begin{equation}\label{eq:slopercaso0}
	\mu(\cG_r) = 8 b_0 - k_0 - 3, 
	\end{equation}as in \S\;\ref{S:subs41}-{\bf Case L} and in \eqref{eq:slope20}. Moreover, from Theorem \ref{thm:stab}-(a), 
	any such $\cG_r$ is strictly semistable and slope-semistable, being an extension of Ulrich bundles of the same slope  
	$\mu(\cG_{r-1}) = \mu (L_{\epsilon_{r}}) = 8 b_0 - k_0 - 3$.

\begin{lemma} \label{lemma:1} Let  $L$ denote any of the two line bundles 
$L_1$ and $L_2$ as in \eqref{eq:Li}. Then, for all integers $r \geqslant1$, we have
    \begin{itemize}
    \item[(i)] $h^2(\cG_r \otimes L^{\vee})= h^3(\cG_r \otimes L^{\vee}) = 0$,
      \item[(ii)] $h^2(\cG_r^{\vee} \otimes L)=h^3(\cG_r^{\vee} \otimes L) = 0 $,
      \item[(iii)] $h^1(\cG_r \otimes L_{\epsilon_{r+1}}^{\vee})\geqslant{\min}\{4b_0-4,\;2b_0+2\} \geqslant4$.
      \end{itemize}
  \end{lemma}
  
	\begin{proof} For $r=1$ we have $\cG_1 = L_1$; therefore $\cG_1 \otimes L^{\vee}$ and 
	$\cG_1^{\vee} \otimes L$  are either equal to $\mathcal O_{X_0}$, if $L=L_1$, or equal to $L_1 - L_2$ and $L_2 - L_1$, respectively, if $L=L_2$. Therefore (i) and (ii) hold true by computations as in \S\;\ref{S:subs41}-{\bf Case L}. 
	As for (iii), by \eqref{eq:jr} we have that $L_{\epsilon_{2}} = L_2$ thus 
	$h^1(\cG_1 \otimes L_{2}^{\vee}) = h^1(L_1 - L_2) = 4 b_0 - 4$, as is \S\;\ref{S:subs41}-{\bf Case L}, 
	the latter being always greater than or equal to ${\min}\{4b_0-4,\;2b_0+2\} \geqslant 4$ since $b_0 \geqslant2$ by  \eqref{eq:rem:assAB}.
	
	Therefore, we will assume $r \geqslant2$ and proceed by induction. Regarding (i), since it holds for $r=1$, assuming it holds for $r-1$ then 
	by tensoring \eqref{eq:1} with $L^{\vee}$ we get that
	\begin{eqnarray*}
     h^j(\cG_{r} \otimes L^{\vee}) =0, \;\; j=2,3,
    \end{eqnarray*}
	because $h^j(\cG_{r-1} \otimes L^{\vee}) = 0$, for $j=2,3$, by inductive hypothesis whereas 
		$h^j(L_{\epsilon_{r}} \otimes L^{\vee}) = 0 $, for $j=2,3$, since $L_{\epsilon_{r}} \otimes L^{\vee}$ 
		 is either $\mathcal O_{X_0}$, or $L_2-L_1$, or $L_1-L_2$.

    A similar reasoning, tensoring the dual of \eqref{eq:1} by $L$, proves (ii).

    To prove (iii), tensor \eqref{eq:1} by $L_{\epsilon_{r+1}}^{\vee}$ and use that $h^2(\cG_{r-1}\otimes L_{\epsilon_{r+1}}^{\vee})=0$ by (i). Thus 
		we have the surjection
$$H^1(\cG_r \otimes L_{\epsilon_{r+1}}^{\vee}) \twoheadrightarrow H^1(L_{\epsilon_{r}} \otimes L_{\epsilon_{r+1}}^{\vee}),$$which implies that 
$h^1(\cG_r \otimes L_{\epsilon_{r+1}}^{\vee}) \geqslant h^1(L_{\epsilon_{r}} \otimes L_{\epsilon_{r+1}}^{\vee})$. According to the 
parity of $r$, we have that $L_{\epsilon_{r}} \otimes L_{\epsilon_{r+1}}^{\vee}$ equals 
either $L_1 - L_2$ or $L_2 - L_1$. From computations as in \S\;\ref{S:subs41}-{\bf Case L}, 
$h^1(L_1-L_2) = 4b_0-4$ whereas $h^1(L_2-L_1) = 2b_0+2$. Notice that 
\[
{\rm min} \{4b_0-4,\;2b_0+2\} := \begin{cases}
      4b_0-4 = 4, & \mbox{if $b_0=2$}, \\
      4b_0-4 = 2b_0+2 = 8, & \mbox{if $b_0=3$}, \\
			2b_0+2 \geqslant10, & \mbox{if $b_0\geqslant4$}.
    \end{cases}
		\]Therefore one concludes.
		\end{proof}

  \begin{cor}\label{cor:Corollario al Lemma 4.2}  For any integer $r \geqslant 1$ there exist on 
	$X_0$ rank-$r$ vector bundles $\cG_r$, which are Ulrich w.r.t. $\xi$, with first Chern class 
	$c_1(\cG_r)$ as in \eqref{eq:c1rcaso0}, of slope $\mu(\cG_r) = 8 b_0 - k_0 - 3$ w.r.t. $\xi$ and which arise as non-trivial extensions 
	as in \eqref{eq:1} if $r \geqslant 2$. 
	\end{cor}
	\begin{proof} For $r=1$, we have $\cG_1 = L_1$ and the statement holds true from Theorem \ref{prop:LineB} and 
	computations in \S\;\ref{S:subs41}-{\bf Case L}. 
	
	 For any $r \geqslant2$, notice that 
	$$\Ext^1(L_{\epsilon_{r}}, \cG_{r-1}) \cong H^1(\cG_{r-1} \otimes L_{\epsilon_{r}}^{\vee}).$$Therefore, 
	from Lemma \ref{lemma:1}-(iii) there exist non--trivial  extensions as in \eqref{eq:1}, which are therefore 
	Ulrich with respect to $\xi$ and whose Chern class $c_1(\cG_r)$ is exactly as in \eqref{eq:c1rcaso0}. 
	
	By induction $\mu(\cG_{r-1}) = \mu(L_{\epsilon_{r}}) = 8 b_0 - k_0 - 3$; then 
	$\cG_r$ has the same slope w.r.t. $\xi$.
	\end{proof}

 From Corollary \ref{cor:Corollario al Lemma 4.2}, at any step we can always pick {\em non--trivial} extensions of the form \eqref{eq:1} and 
we will henceforth do so.

\begin{lemma} \label{lemma:2}  Let $r \geqslant1$  be an integer. Then we have 
    \begin{itemize}
    \item[(i)] $h^1(\cG_{r+1} \otimes L_{\epsilon_{r+1}}^{\vee})=h^1(\cG_r \otimes L_{\epsilon_{r+1}}^{\vee})-1$,
    \item[(ii)] $h^1(\cG_r \otimes L_{\epsilon_{r+1}}^{\vee})= 
		\begin{cases}
      \frac{(r+1)}{2} h^1(L_1-L_2) - \frac{(r-1)}{2} = 2(r+1)(b_0-1) - \frac{(r-1)}{2}, & \mbox{if $r$ is odd}, \\
			\frac{r}{2} h^1(L_2-L_1) - \frac{(r-2)}{2} = r(b_0+1) - \frac{(r-2)}{2}, & \mbox{if $r$ is even}.
    \end{cases}$
		\item[(iii)] $h^2(\cG_r \otimes \cG_r^{\vee}) = h^3(\cG_r \otimes \cG_r^{\vee})=0$,
		
\item[(iv)] $\chi(\cG_r \otimes L_{\epsilon_{r+1}}^{\vee})= 
\begin{cases}
      \frac{(r+1)}{2} (1- h^1(L_1-L_2)) -1  =  \frac{(r+1)}{2} (5 - 4 b_0) -1, & \mbox{if $r$ is odd}, \\
			\frac{r}{2} (1- h^1(L_2-L_1)) = \frac{r}{2} (-1- 2b_0) , & \mbox{if $r$ is even}.
    \end{cases}$

\item[(v)] $\chi(L_{\epsilon_{r}} \otimes \cG_r^{\vee})= 
\begin{cases}
      \frac{(r-1)}{2} (1- h^1(L_1-L_2)) + 1  = \frac{(r-1)}{2} (5 - 4 b_0) + 1, & \mbox{if $r$ is odd}, \\
			\frac{r}{2} (1- h^1(L_2-L_1)) = \frac{r}{2} (-1- 2b_0), & \mbox{if $r$ is even}.
    \end{cases}$
\item[(vi)]  $\chi(\cG_r \otimes \cG_r^{\vee})= 
\begin{cases}
  \scriptstyle    \frac{(r^2 -1)}{4} (2-h^1(L_1-L_2)-h^1(L_2-L_1)) + 1  = \frac{(r^2 -1)}{4}(4 - 6 b_0) + 1, & \mbox{if $r$ is odd}, \\
			\frac{r^2}{4} (2- h^1(L_1-L_2)-h^1(L_2-L_1)) = \frac{r^2}{4} (4- 6b_0) , & \mbox{if $r$ is even}.
    \end{cases}$
\end{itemize}
\end{lemma}

\begin{proof}  (i) Consider the exact sequence \eqref{eq:1}, where $r$ is replaced by $r+1$. From \linebreak 
$\Ext^1(L_{\epsilon_{r+1}},\cG_r)\cong H^1(\cG_r \otimes L_{\epsilon_{r+1}}^{\vee})$ and the fact that the exact 
sequence defining $\cG_{r+1}$ is constructed by taking a non--zero vector $[\cG_{r+1}]$ 
in $\Ext^1(L_{\epsilon_{r+1}},\cG_r)$, it follows that the coboundary map 
$$H^0(\mathcal O_{X_0}) \stackrel{\partial}{\longrightarrow } H^1(\cG_r \otimes L_{\epsilon_{r+1}}^{\vee})$$of the exact sequence 
\begin{eqnarray}\label{eq:dag}
0 \to \cG_r \otimes L_{\epsilon_{r+1}}^{\vee} \to \cG_{r+1}\otimes L_{\epsilon_{r+1}}^{\vee} \to  \mathcal O_{X_0} \to 0,
 \end{eqnarray} is non--zero so it is injective. Thus, (i) follows from the cohomology of \eqref{eq:dag}.

\medskip

\noindent
(ii) We use induction on $r$. For $r=1$, the right hand side of the formula
yields $4(b_0-1)$ which is exactly $h^1(\cG_1 \otimes L_2^{\vee})=h^1(L_1-L_2)$ as 
in \S\;\ref{S:subs41}-{\bf Case L}. 

When  $r=2$, the right hand side of the formula  is $2(b_0+1)$ which is $ h^1(\cG_2 \otimes L_1^{\vee}) = h^1(L_2 - L_1) = 2b_0+2$, as it follows from 
computations in \S\;\ref{S:subs41}-{\bf Case L}, from the exact sequence 
\begin{eqnarray*}
0 \to \mathcal O_{X_0} \to \cG_2 \otimes L_1^{\vee} \to L_2 - L_1 \to 0,
\end{eqnarray*}
obtained 
by \eqref{eq:1} with $r=2$ and tensored with $L_1^{\vee}$, and the fact that 
$h^j(\mathcal O_{X_0}) = 0$, for $j=1,2$.

Assume now that the formula holds true up to some integer $r \geqslant 2$; we have to show that it holds also for $r+1$. Consider the exact sequence \eqref{eq:1}, with $r$ replaced by $r+1$, and tensor it by $L_{\epsilon_{r+2}}^{\vee}$. We thus 
obtain
\begin{eqnarray} \label{eq:dagdag}
0 \to \cG_r\otimes L_{\epsilon_{r+2}}^{\vee} \to \cG_{r+1}\otimes L_{\epsilon_{r+2}}^{\vee} \to  L_{\epsilon_{r+1}}\otimes L_{\epsilon_{r+2}}^{\vee} \to 0
 \end{eqnarray}

If $r$ is even, then $L_{\epsilon_{r+2}} = L_2$ whereas $L_{\epsilon_{r+1}} =L_1$. Thus $h^0(L_{\epsilon_{r+1}}\otimes L_{\epsilon_{r+2}}^{\vee})= h^0(L_1 - L_2) = 0$ and  
$h^1(L_{\epsilon_{r+1}}\otimes L_{\epsilon_{r+2}}^{\vee})=h^1(L_1 - L_2) = 4b_0 -4$.  On the other hand, by Lemma \ref{lemma:1}-(i), 
$h^2(\cG_r\otimes L_{\epsilon_{r+2}}^{\vee})=0$. Thus, from \eqref{eq:dagdag}, we get: 
\begin{eqnarray*}
 h^1(\cG_{r+1} \otimes L_{\epsilon_{r+2}}^{\vee})=(4b_0-4) +h^1(\cG_r\otimes L_{\epsilon_{r+2}}^{\vee})=(4b_0-4)+
h^1(\cG_r\otimes L_{\epsilon_{r}}^{\vee}), 
 \end{eqnarray*}
 as $r$ and $r+2$ have the same parity. Using (i),  we have $h^1(\cG_r\otimes L_{\epsilon_{r}}^{\vee}) = h^1(\cG_{r-1}\otimes L_{\epsilon_{r}}^{\vee}) -1$ therefore, 
by inductive hypothesis with $r-1$ odd, we have $h^1(\cG_{r-1}\otimes L_{\epsilon_{r}}^{\vee}) = \frac{r}{2} (4b_0-4) - \frac{(r-2)}{2}$. 
 Summing up, we have 
\begin{eqnarray*}
h^1(\cG_{r+1} \otimes L_{\epsilon_{r+2}}^{\vee}) = (4b_0-4) + \frac{r}{2} (4b_0-4) - \frac{(r-2)}{2} -1,
\end{eqnarray*}
which is easily seen to be equal to the right hand side expression in (ii), when $r$ is replaced by $r+1$. 

If  $r$ is odd, the same holds for $r+2$ whereas $r+1$ is even. In this case  $L_{\epsilon_{r+2}}= L_1$, 
$L_{\epsilon_{r+1}} =L_2$ so $h^1(L_{\epsilon_{r+1}}\otimes L_{\epsilon_{r+2}}^{\vee})=h^1(L_2 - L_1) = 2b_0+2$ and one applies the same procedure 
as in the previous case. 

\medskip

\noindent
(iii)  We again use induction on $r$. For $r=1$, formula (iii) states that $h^j(L_1 -L_1)=h^j(\mathcal O_{X_0})=0$, for $j=2,3$,  which is  certainly true. 

Assume now that (iii) holds  up to some integer $r \geqslant 1$; we have to prove that it holds also for $r+1$. 
Consider the exact sequence \eqref{eq:1}, where $r$ is replaced by $r+1$, and tensor it by $\cG_{r+1}^{\vee}$. From this we get that, for $j=2,3$, 
\begin{eqnarray} \label{eq:zaniolo}
 h^j( \cG_{r+1} \otimes \cG_{r+1}^{\vee}) \leqslant h^j( \cG_{r} \otimes \cG_{r+1}^{\vee}) + h^j( L_{\epsilon_{r+1}} \otimes \cG_{r+1}^{\vee}) = h^j( \cG_{r} \otimes \cG_{r+1}^{\vee}),
\end{eqnarray} the latter equality follows from $h^j( L_{\epsilon_{r+1}} \otimes \cG_{r+1}^{\vee}) =0$, $j =2,3$, as in 
Lemma \ref{lemma:1}-(ii).  

Consider the dual exact sequence of \eqref{eq:1}, where $r$ is replaced by $r+1$, and tensor it by $\cG_{r}$. Thus, 
Lemma \ref{lemma:1}-(i) yields that, for $j=2,3$, one has 
\begin{eqnarray} \label{eq:pogba}
h^j( \cG_{r} \otimes \cG_{r+1}^{\vee}) \leqslant h^j(\cG_r \otimes L_{\epsilon_{r+1}}^{\vee})+h^j(\cG_{r} \otimes \cG_{r}^{\vee})= h^j(\cG_{r} \otimes \cG_{r}^{\vee}).
\end{eqnarray} Now \eqref{eq:zaniolo}--\eqref{eq:pogba} and the inductive hypothesis yield
$h^j( \cG_{r+1} \otimes \cG_{r+1}^{\vee})=0$, for $j=2,3$, as desired.

\medskip

\noindent
(iv)  For $r=1$,  (iv) reads $\chi(L_1-L_2)= - h^1(L_1-L_2) = 4 - 4b_0$, which is true since 
$h^j(L_1-L_2)=0$ for $j=0,2,3$.

For $r=2$, (iv) reads $\chi(\cG_2 \otimes L_1^{\vee}) = 1 - h^1(L_2-L_1) = - 1 - 2 b_0$ and this holds true because 
 if we take the exact sequence \eqref{eq:1}, with $r=2$, tensored by $L_1^{\vee}$ then 
\begin{eqnarray*} 
\chi(\cG_2 \otimes L_1^{\vee}) = \chi (\mathcal O_{X_0}) + \chi (L_2-L_1) = 1 - h^1(L_2-L_1) = 1 - (2b_0+2),
\end{eqnarray*} 
as 
$h^j(L_2-L_1)= 0$ for $j=0,2,3$. 

Assume now that the formula holds up to a certain integer $r \geqslant 2$, we have to prove that it also holds for $r+1$.  
From  \eqref{eq:dagdag} we get 
\begin{eqnarray*} 
\chi(\cG_{r+1}\otimes L_{\epsilon_{r+2}}^{\vee}) = \chi(\cG_r\otimes L_{\epsilon_{r+2}}^{\vee}) + \chi(L_{\epsilon_{r+1}}\otimes L_{\epsilon_{r+2}}^{\vee}).
\end{eqnarray*}

If $r$ is even, the same is   true for $r+2$ whereas $r+1$ is odd. Therefore, 
\begin{equation} \label{eq:maz1}
  \chi(\cG_{r+1}\otimes L_{\epsilon_{r+2}}^{\vee}) = \chi(\cG_r\otimes L_{2}^{\vee})+ \chi(L_{1} - L_{2}) = 
	\chi(\cG_r\otimes L_{2}^{\vee})-h^1(L_1-L_2).
  \end{equation} Then  \eqref{eq:dag}, with $r$ replaced by $r-1$, yields
  \begin{equation}\label{eq:maz2}
    \chi(\cG_r\otimes L_{2}^{\vee})=\chi(\cG_{r-1}\otimes L_{2}^{\vee})+\chi(\mathcal O_{X_0})= \chi(\cG_{r-1}\otimes L_{2}^{\vee})+1.
  \end{equation}
Substituting \eqref{eq:maz2} into \eqref{eq:maz1} and using the inductive hypothesis with $r-1$ odd, we get
  \begin{eqnarray*}
    \chi(\cG_{r+1}\otimes L_{2}^{\vee}) &=& \chi(\cG_{r-1}\otimes L_{2}^{\vee})+1-h^1(L_1-L_2) \\
    & = & \frac{(r)}{2} (1- h^1(L_1-L_2))-h^1(L_1-L_2) \\
		 & = & \frac{(r+2)}{2} (1- h^1(L_2-L_1)) -1,   
\end{eqnarray*} proving that the formula holds also for $r+1$ odd.

Similar procedure can be used to treat the case when $r$ is odd. In this case, $L_{\epsilon_{r+1}} = L_2$ whereas $L_{\epsilon_{r+2}} = L_1$. Thus, 
from the above computations, 
\begin{eqnarray*}
\chi(\cG_{r+1}\otimes L_1^{\vee}) = \chi(\cG_r\otimes L_{1}^{\vee})+ \chi(L_2 - L_1) = \chi(\cG_r\otimes L_{1}^{\vee})- h^1(L_2 - L_1).
\end{eqnarray*}
As in the previous case, 
$\chi(\cG_r\otimes L_{1}^{\vee}) = 1 + \chi (\cG_{r-1} \otimes L_{1}^{\vee})$ so, applying inductive hypothesis with $r-1$ even, we get 
$\chi(\cG_r\otimes L_{1}^{\vee}) = 1 + \frac{(r-1)}{2} (1 - h^1(L_2-L_1))$. Adding up all these quantities, we get 
\begin{eqnarray*}\chi(\cG_{r+1}\otimes L_{\epsilon_{r+2}}^{\vee})= \chi(\cG_{r+1}\otimes L_1^{\vee}) = \frac{r+1}{2} (1 - h^1(L_2 -L_1)),
\end{eqnarray*}
so formula (iv) 
holds true also for $r+1$ even.

\medskip

\noindent
(v) For $r=1$, (v) reads $\chi(L_1 -L_1) = \chi(\mathcal O_{X_0}) =1$, which is correct. 
For $r=2$, (v) reads  $\chi(L_2\otimes \cG_2^{\vee})= 1 - h^1(L_2-L_1)$, which is once again correct 
as it follows from the dual of sequence \eqref{eq:1} tensored by $L_2$.

Assume now that the formula holds up to a certain integer $r \geqslant 2 $ and we need to proving it  for $r+1$.  
Dualizing \eqref{eq:1}, replacing  $r$  by  $r+1$ and tensoring it  by $L_{\epsilon_{r+1}}$
 we find that
\begin{eqnarray} \label{eq:maz3}
  \chi(L_{\epsilon_{r+1}} \otimes \cG_{r+1}^{\vee} ) & = &   \chi(L_{\epsilon_{r+1}} \otimes L_{\epsilon_{r+1}}^{\vee})+ \chi(L_{\epsilon_{r+1}}\otimes \cG_{r}^{\vee}) \\
  \nonumber & = &  \chi(\mathcal O_{X_n})+\chi(L_{\epsilon_{r+1}}\otimes \cG_{r}^{\vee}) 
 =  1+\chi(L_{\epsilon_{r+1}}\otimes \cG_{r}^{\vee}).
  \end{eqnarray}
   The  dual of sequence \eqref{eq:1}, with $r$ replaced by $r-1$, tensored by $L_{\epsilon_{r+1}}$ yields
  \begin{equation}
    \label{eq:maz4}
\chi(L_{\epsilon_{r+1}}\otimes \cG_{r}^{\vee})=\chi(L_{\epsilon_{r+1}} \otimes L_{\epsilon_{r}}^{\vee})+ \chi(L_{\epsilon_{r+1}}\otimes \cG_{r-1}^{\vee}).
\end{equation}
   Substituting \eqref{eq:maz4} into \eqref{eq:maz3} and using the fact that $r+1$ and $r-1$ have the same parity, we get 
\begin{eqnarray*}	\chi(L_{\epsilon_{r+1}} \otimes \cG_{r+1}^{\vee} ) = 1 + \chi(L_{\epsilon_{r+1}} \otimes L_{\epsilon_{r}}^{\vee}) + 
\chi(L_{\epsilon_{r-1}}\otimes \cG_{r-1}^{\vee}).\end{eqnarray*}

If $r$ is even,  then $\chi(L_{\epsilon_{r+1}} \otimes L_{\epsilon_{r}}^{\vee}) = \chi(L_1-L_2) = - h^1(L_1-L_2)$ 
whereas, from the inductive hypothesis with $r-1$ odd, $\chi(L_{\epsilon_{r-1}}\otimes \cG_{r-1}^{\vee}) = 1 + \frac{(r-2)}{2} (1 - h^1(L_1-L_2))$. Thus
\begin{eqnarray*}\chi(L_{\epsilon_{r+1}} \otimes \cG_{r+1}^{\vee} ) = 1 -  h^1(L_1-L_2) + 1 + \frac{(r-2)}{2} (1 - h^1(L_1-L_2)),
\end{eqnarray*}
the latter 
equals $1 +  \frac{r}{2} (1 - h^1(L_1-L_2))$, proving that the formula holds also for $r+1$ odd.

If $r$ is odd, the strategy is similar; in this case one has 
$\chi(L_{\epsilon_{r+1}} \otimes L_{\epsilon_{r}}^{\vee}) = \chi(L_2-L_1) = - h^1(L_2-L_1)$ and, by the inductive hypothesis with $r-1$ even,
$\chi(L_{\epsilon_{r-1}}\otimes \cG_{r-1}^{\vee}) = \frac{(r-1)}{2} (1 - h^1(L_2-L_1))$ so one can conclude. 

\medskip

\noindent
(vi) We first check the given formula for $r=1,2$. We have $\chi(\cG_1 \otimes \cG_1^{\vee})=\chi(L_1-L_1)=\chi(\mathcal O_{X_0})=1$, which fits with the given formula for $r=1$. 
From \eqref{eq:1}, with $r=2$, tensored by $\cG_2^{\vee}$ we get
\begin{equation}\label{eq:cr1}
  \chi(\cG_2 \otimes \cG_2^{\vee})=\chi(L_1 \otimes \cG_2^{\vee})+\chi(L_2 \otimes \cG_2^{\vee})\stackrel{(v)}{=} \chi(L_1 \otimes \cG_2^{\vee}) + 1 - h^1(L_2-L_1).
\end{equation} From the dual of  \eqref{eq:1}, with $r=2$,  tensored by $L_1$ we get
\begin{equation} \label{eq:cr2}
  \chi(L_1 \otimes \cG_2^{\vee})=\chi(L_1 -L_1)+\chi(L_1 -L_2)=\chi(\mathcal O_{X_0})- h^1(L_1-L_2) =1-h^1(L_1-L_2).
\end{equation} Combining \eqref{eq:cr1} and \eqref{eq:cr2}, we get 
\begin{eqnarray*}
 \chi(\cG_2 \otimes \cG_2^{\vee})= 2 -h^1(L_1-L_2) - h^1(L_2-L_1),
 \end{eqnarray*} which again fits with the given formula for $r=2$.

Assume now that the given formula is valid up to a certain integer $r \geqslant2$; we need to prove it holds for $r+1$. 
From \eqref{eq:1}, in which $r$ is replaced by $r+1$, tensored by $\cG_{r+1}^{\vee}$ and successively the dual of  \eqref{eq:1}, with $r$ replaced by $r+1$,  
tensored by $\cG_r$ we get
\begin{eqnarray*}
\chi(\cG_{r+1} \otimes \cG_{r+1}^{\vee}) & = & \chi(\cG_{r} \otimes \cG_{r}^{\vee}) +\chi(\cG_{r} \otimes L_{\epsilon_{r+1}}^{\vee})+ \chi(L_{\epsilon_{r+1}} \otimes \cG_{r+1}^{\vee}).
\end{eqnarray*}

If $r$ is even, then $r+1$ is odd and $L_{\epsilon_{r+1}} = L_1$. From (v) with $(r+1)$ odd, we get 
$\chi(L_{\epsilon_{r+1}} \otimes \cG_{r+1}^{\vee}) = 1 + \frac{r}{2} (1 - h^1(L_1-L_2))$, whereas from (iv) with $r$ even
$\chi(\cG_{r} \otimes L_{\epsilon_{r+1}}^{\vee}) = \frac{r}{2} (1 - h^1(L_2-L_1))$. Finally, by the inductive hypothesis with $r$ even, 
$\chi(\cG_{r} \otimes \cG_{r}^{\vee}) = \frac{r^2}{4} (2 - h^1(L_1-L_2) - h^1(L_2-L_1))$. Summing--up the three quantities, one gets 
\begin{eqnarray*}
\chi(\cG_{r+1} \otimes \cG_{r+1}^{\vee}) = 1 + \frac{(r+1)^2-1}{4} (2 - h^1(L_1-L_2) - h^1(L_2-L_1)),
\end{eqnarray*}
proving that the formula holds for $r+1$ odd.
 
If  $r$ is odd, then 
$\chi(L_{\epsilon_{r+1}} \otimes \cG_{r+1}^{\vee}) = \frac{r+1}{2} (1 - h^1(L_2-L_1))$, as it follows from (v) with $(r+1)$ even, whereas
$\chi(\cG_{r} \otimes L_{\epsilon_{r+1}}^{\vee}) = \frac{(r+1)}{2} (1 - h^1(L_1-L_2)) - 1$, as predicted by (iv) with $r$ odd. Finally, form the inductive hypothesis with 
$r$ odd, we have 
$\chi(\cG_{r} \otimes \cG_{r}^{\vee}) = 1 + \frac{(r^2-1)}{4} (2 - h^1(L_1-L_2)- h^1(L_2-L_1))$. If we add up the three quantities, we get 
\begin{eqnarray*}
\chi(\cG_{r+1} \otimes \cG_{r+1}^{\vee}) = \frac{(r+1)^2}{4} (2 - h^1(L_1-L_2) - h^1(L_2-L_1)),
\end{eqnarray*}
finishing the proof. 
\end{proof}

Notice some fundamental properties arising from the first step of the previous iterative contruction in \eqref{eq:1}. 
We set $\cG_1 = L_1$, which is a Ulrich line bundle of slope $\mu = 8b_0 - k_0 -3$; by considering 
non--trivial extensions \eqref{eq:1r1}, $\cG_2$ turned out to be a simple bundle, as it follows from \cite[Lemma\,4.2]{c-h-g-s} and from 
the fact that $\cG_1 = L_1$ and $L_{\epsilon_2} = L_2$ are slope--stable, of the same slope $\mu = 8b_0 - k_0 -3$, non-isomorphic line bundles. 
Moreover, by construction, $\cG_2$ turned out to be Ulrich, strictly semistable and of slope $\mu = 8b_0 - k_0 -3$; on the other hand, 
in the proof of Theorem \ref{prop:rk 2 simple Ulrich vctB e=0;I} we showed that $\cG_2$ deforms in a smooth, irreducible modular family to a slope-stable 
Ulrich bundle $\cU_2 := \cU$, of same slope and Chern classes as $\cG_2$, which gives rise to a general point $[\cU_2]$ 
of the generically smooth component $\mathcal M(2) := \mathcal M$ of the moduli space of Ulrich bundles described in Theorem \ref{prop:rk 2 simple Ulrich vctB e=0;I}. 

In this way, by induction, we can assume that up to a given integer $r \geqslant 3$ we have constructed a generically smooth irreducible component 
$\mathcal M(r-1)$ of the moduli space of Ulrich bundles of rank $(r-1)$, with Chern class $c_1 = c_1(\cG_{r-1})$ and slope $\mu = 8b_0 - k_0 -3$, whose general 
point $[\cU_{r-1}] \in \mathcal M(r-1)$ is slope-stable.

Consider now extensions 
\begin{equation} \label{eq:estensioneL}
  0 \to \cU_{r-1} \to \cF_r \to L_{\epsilon_r} \to 0, 
\end{equation}
with $[\cU_{r-1}] \in \mathcal M(r-1)$ general and with $L_{\epsilon_r}$ defined as in \eqref{eq:jr}, \eqref{eq:1}, according to the parity of $r$. 
Notice that 
$${\rm Ext}^1(L_{\epsilon_r}, \cU_{r-1}) \cong H^1(\cU_{r-1} \otimes L_{\epsilon_r}^{\vee}).$$

\begin{lemma}\label{lem:new} In the above set-up, one has 
$$h^1(\cU_{r-1} \otimes L_{\epsilon_r}^{\vee}) \geqslant {\rm min} \{4b_0-5,\; 2b_0+1\} \geqslant 3.$$In particular, 
${\rm Ext}^1(L_{\epsilon_r}, \cU_{r-1})$ contains non-trivial extensions as in \eqref{eq:estensioneL}.
\end{lemma}

\begin{proof} By inductive assumption, $\cU_{r-1}$ specializes to $\cG_{r-1}$ in the smooth modular family thus, by semi-continuity and by Lemma \ref{lemma:1}-(i), 
one has 
\begin{equation}\label{eq:porca1}
h^j(\cU_{r-1} \otimes L_{\epsilon_r}^{\vee}) = h^j(\cG_{r-1} \otimes L_{\epsilon_r}^{\vee}) = 0, \; j = 2,3.
\end{equation} For the same reason
\begin{equation}\label{eq:(*)}
\chi(\cU_{r-1} \otimes L_{\epsilon_r}^{\vee}) = \chi(\cG_{r-1} \otimes L_{\epsilon_r}^{\vee}),
\end{equation} where the latter is as in Lemma \ref{lemma:2}-(iv).

Thus, equality in \eqref{eq:(*)}, together with \eqref{eq:porca1}, reads
$$h^0(\cU_{r-1} \otimes L_{\epsilon_r}^{\vee}) - h^1(\cU_{r-1} \otimes L_{\epsilon_r}^{\vee}) = h^0(\cG_{r-1} \otimes L_{\epsilon_r}^{\vee} ) - h^1( \cG_{r-1} \otimes L_{\epsilon_r}^{\vee}),$$namely
\begin{equation}\label{eq:porca3}
h^1(\cU_{r-1} \otimes L_{\epsilon_r}^{\vee}) = h^1( \cG_{r-1} \otimes L_{\epsilon_r}^{\vee}) - \left(h^0(\cG_{r-1} \otimes L_{\epsilon_r}^{\vee} ) - h^0(\cU_{r-1} \otimes L_{\epsilon_r}^{\vee} ) \right),
\end{equation}where $h^1(\cG_{r-1} \otimes L_{\epsilon_r}^{\vee}) \geqslant {\rm min} \{4b_0-4,\;2b_0 +2\} \geqslant 4$, as from Lemma \ref{lemma:1}-(iii) where $r$ is replaced by $r-1$. 
We claim that the following equality

\begin{equation}\label{eq:porcavacca}
h^0(\cG_{r-1} \otimes L_{\epsilon_r}^{\vee}) = 
\left\{
\begin{array}{ccl}
0 & {\rm if} &  r\; {\rm \; even} \\
1 & {\rm if} & r \; {\rm odd} 
\end{array}
\right.
\end{equation}holds true. 

Assume for a moment that \eqref{eq:porcavacca} has been proved;  since $\cU_{r-1}$ is slope-stable, 
of the same slope as $L_{\epsilon_r}$, and $\cU_{r-1}$ is not isomorphic to $L_{\epsilon_r}$, then 
$h^0(\cU_{r-1} \otimes L_{\epsilon_r}^{\vee} ) = 0$ as any non-zero homomorphism $L_{\epsilon_r} \to \cU_{r-1}$ 
should be an isomorphism. Thus, using \eqref{eq:porca3}, for any $r \geqslant 2$ 
one gets therefore 
$$h^1(\cU_{r-1} \otimes L_{\epsilon_r}^{\vee}) \geqslant h^1(\cG_{r-1} \otimes L_{\epsilon_r}^{\vee} ) -1$$which, together with Lemma \ref{lemma:1}-(iii), proves the statement. 

Thus, we are left with the proof of \eqref{eq:porcavacca}. To prove it, we will use induction on $r$. 

If $r=2$, then $\cG_1 = L_1$, $L_{\epsilon_2} = L_2$, thus 
$h^0(\cG_1 \otimes L_2^{\vee}) = h^0(L_1-L_2) = 0$, as it follows from \eqref{extension1dualL1} and from \eqref{eq:calcoliutili}. If otherwise $r=3$, then $\cG_{r-1} = \cG_2$ 
as in \eqref{eq:1r1} whereas $L_{\epsilon_3} = L_1$, as in \eqref{eq:jr}. Thus, tensoring \eqref{eq:1r1} by $L_1^{\vee}$, one gets
$$0 \to \Oc_{X_0} \to \cG_2 \otimes L_1^{\vee} \to L_2 - L_1 \to 0;$$since 
$h^0(L_2 - L_1) = 0$, from \eqref{extension1dualL2} and from \eqref{eq:calcoliutili2}, then $h^0( \cG_2 \otimes L_1^{\vee}) = h^0(\Oc_{X_0}) = 1$.

Assume therefore that, up to some integer $r-2 \geqslant 2$,  \eqref{eq:porcavacca} holds true and take $\cG_{r-1}$ a non-trivial extension as in \eqref{eq:1}, with $r$ 
replaced by $r-1$, namely 
\begin{equation}\label{eq:porca*}
0 \to \cG_{r-2} \to \cG_{r-1} \to L_{\epsilon_{r-1}} \to 0. 
\end{equation}

If $r$ is even, then $r-2$ is even and $r-1$ is odd, in particular $L_{\epsilon_{r-1}} = L_1$ and $L_{\epsilon_{r}} = L_2$. Thus, tensoring \eqref{eq:porca*} with $L_{\epsilon_{r}}^{\vee} = 
L_2^{\vee}$ gives 
$$0 \to \cG_{r-2} \otimes L_2^{\vee} \to \cG_{r-1} \otimes L_2^{\vee} \to L_1 - L_2 \to 0.$$Since $h^0(L_1 - L_2) = 0$ then 
$$h^0(\cG_{r-1} \otimes L_{\epsilon_{r}}^{\vee})  = h^0(\cG_{r-1} \otimes L_2^{\vee}) = h^0(\cG_{r-2} \otimes L_2^{\vee}).$$On the other hand, 
by \eqref{eq:1}, with $r$ replaced by $r-2$, namely 
\begin{equation}\label{eq:porca**}
0 \to \cG_{r-3} \to \cG_{r-2} \to L_{\epsilon_{r-2}} \to 0,
\end{equation}we have $L_{\epsilon_{r-2}} = L_2$, since $r-2$ is even as $r$ is. Thus, tensoring \eqref{eq:porca**} with $L_{\epsilon_r}^{\vee}$ and taking into account that $r$ is even, one gets 
$$0 \to \cG_{r-3} \otimes L_2^{\vee} \to \cG_{r-2}  \otimes L_2^{\vee}\to \Oc_{X_0} \to 0.$$Notice that 
$\cG_{r-3} \otimes L_2^{\vee} = \cG_{r-3} \otimes L_{\epsilon_{r-2}}^{\vee}$ thus, since $r-3$ is odd, $h^0(\cG_{r-3} \otimes L_2^{\vee}) = 0$ by induction and by \eqref{eq:porcavacca}. On the other hand, 
the coboundary map   
$$H^0(\Oc_{X_0}) \cong {\mathbb C} \stackrel{\partial}{\longrightarrow} H^1(\cG_{r-3} \otimes  L_2^{\vee})= H^1(\cG_{r-3} \otimes  L_{\epsilon_{r-2}}^{\vee}) \cong {\rm Ext}^1(L_{\epsilon_{r-2}}, \cG_{r-3})$$is non-zero since, by iterative construction, $\cG_{r-2}$ is taken to be a non-trivial extension; therefore $\partial$ is injective which implies 
$h^0(\cG_{r-2} \otimes  L_2^{\vee}) = 0$ and so $h^0(\cG_{r-1} \otimes  L_{\epsilon_r}^{\vee}) = 0$, as desired.

Assume now $r$ to be odd thus, $L_{\epsilon_{r}} = L_1$ whereas $L_{\epsilon_{r-1}} = L_2$. Tensoring \eqref{eq:porca*} with $L_1^{\vee}$ gives 
$$0 \to \cG_{r-2} \otimes L_1^{\vee} \to \cG_{r-1} \otimes L_1^{\vee} \to L_2 - L_1 \to 0.$$As $h^0(L_2 - L_1) = 0$, then 
$$h^0(\cG_{r-1} \otimes L_{\epsilon_{r}}^{\vee})  = h^0(\cG_{r-1} \otimes L_1^{\vee}) = h^0(\cG_{r-2} \otimes L_1^{\vee}).$$Since $r$ is odd, then also $r-2$ is odd and one gets
$$0 \to \cG_{r-3} \otimes L_1^{\vee} \to \cG_{r-2}  \otimes L_1^{\vee}\to \Oc_{X_0} \to 0.$$Notice that 
$h^0(\cG_{r-3} \otimes L_1^{\vee} ) = h^0(\cG_{r-3} \otimes L_{\epsilon_{r-2}}^{\vee}) = 1$, as it follows from \eqref{eq:porcavacca} with $r$ replaced by $r-2$ which is odd since $r$ is. 
On the other hand, the fact that $\cG_{r-2}$ arises from a non--trivial extension implies as before that the coboundary map 
$$H^0(\Oc_{X_0}) \cong {\mathbb C} \stackrel{\partial}{\longrightarrow} H^1(\cG_{r-3} \otimes  L_1^{\vee})= H^1(\cG_{r-3} \otimes  L_{\epsilon_{r-2}}^{\vee}) \cong {\rm Ext}^1(L_{\epsilon_{r-2}}, \cG_{r-3})$$is once again injective. This gives  $h^0(\cG_{r-2} \otimes  L_1^{\vee}) = h^0(\cG_{r-3} \otimes  L_1^{\vee}) = 1$, which implies 
 $h^0(\cG_{r-1} \otimes  L_{\epsilon_r}^{\vee}) = 1$. This concludes the proof of the Lemma.  \end{proof}

Lemma \ref{lem:new} ensures that there exist non-trivial extensions arising from \eqref{eq:estensioneL}. Since \linebreak $[\cU_{r-1}] \in \mathcal M(r-1)$ general 
is slope-stable, with $\cU_{r-1}$ not isomorphic to $L_{\epsilon_r}$ (if $r>2$, ${\rm rk}(\cU_{r-1}) > 1 = {\rm rk}(L_{\epsilon_r})$, if otherwise 
$r=2$, $\cU_{1} = L_1$ and $L_{\epsilon_2} = L_2$ are not isomorphic), moreover $\cU_{r-1}$ and $L_{\epsilon_r}$ have the same slope
$\mu = 8b_0 - k_0 -3$ then, by \cite[Lemma\;4.2]{c-h-g-s}, the general bundle $\mathcal F_r$ as in \eqref{eq:estensioneL} is simple, of rank $r$, Ulrich w.r.t. $\xi$ and with $c_1(\mathcal F_r) = c_1(\cU_{r-1}) + L_{\epsilon_r}$ as in \eqref{eq:c1rcaso0}. 

Moreover, as $\cU_{r-1}$ specializes to $\cG_{r-1}$ in the smooth modular family, $\mathcal F_r$ specializes to $\cG_r$ thus, by semi-continuity
$h^2(\mathcal F_r \otimes \mathcal F_r^{\vee}) = 0$ as $h^2(\cG_r \otimes \cG_r^{\vee}) = 0$ (cf. Lemma \ref{lemma:2}-(iii)). Therefore, by \cite[Proposition\;10.2]{c-h-g-s}, 
$\mathcal F_r$ admits a smooth modular family, which we denote by $\mathcal M(r)$. 

For $r \geqslant 2$, the scheme $\mathcal M(r)$ contains a subscheme, denoted by $\mathcal M(r)^{\rm ext}$, which parametrizes bundles $\cF_r$ that are 
non--trivial extensions as in \eqref{eq:estensioneL}.

\begin{lemma} \label{lemma:genUr}  Let $r \geqslant2$  be an integer and let 
$[\cU_r] \in \mathcal M(r)$ be a general point. Then $\cU_r$ is a vector bundle of rank $r$, which is 
Ulrich with respect to $\xi$, with slope w.r.t. $\xi$ given by 
$\mu:= 8b_0 - k_0 - 3$, and with  first Chern class 
\[c_1(\cU_r): =
   \begin{cases} 
      r \xi + \varphi^*\Oc_{\FF_0}(3, b_0 -3) + \varphi^*\Oc_{\FF_0}\left(\frac{r-3}{2}, \frac{(r-3)}{2}(b_0-2) \right), & \mbox{if $r$ is odd}, \\
      r \xi + \varphi^*\Oc_{\FF_0}\left(\frac{r}{2}, \frac{r}{2}(b_0-2)\right), & \mbox{if $r$ is even}.
    \end{cases}\] Moreover $\cU_r$ is simple, in particular indecomposable, with 
  \begin{itemize}
\item[(i)]  $\chi(\cU_r \otimes \cU_r^{\vee})=   \begin{cases}
\scriptstyle    \frac{(r^2 -1)}{4} (2-h^1(L_1-L_2)-h^1(L_2-L_1)) + 1  = \frac{(r^2 -1)}{4}(4 - 6 b_0) + 1, &  \mbox{if $r$ is odd,}
      \\
		\scriptstyle 	\frac{r^2}{4} (2- h^1(L_1-L_2)-h^1(L_2-L_1)) = \frac{r^2}{4} (4- 6b_0) , & \mbox{if $r$ is even}.
    \end{cases}$
 \item[(ii)] $h^j(\cU_r \otimes \cU_r^{\vee})=0$, for $j =2,3$.  
  \end{itemize}
\end{lemma}
  
\begin{proof} Since $\mathcal F_r$ is of rank $r$ and Ulrich w.r.t. $\xi$, the same holds for the general member $[\cU_r] \in \mathcal M(r)$, since 
Ulrichness is an open property in irreducible families as $\mathcal M(r)$. For the same reasons 
$c_1(\cU_r) = c_1(\mathcal F_r) = c_1(\cG_r)$, as in \eqref{eq:c1rcaso0}, and $\mu(\cU_r) = \mu(\mathcal F_r) = \mu(\cU_{r-1})$. 

Since $\mathcal F_r$ is simple, as observed above, by semi-continuity $h^0( \cU_r \otimes \cU_r^{\vee}) = 1$,  i.e. 
$\cU_r$ is simple, in particular it is indecomposable.

Property (ii) follows by specializing $\cF_r$ to a vector bundle $\cG_r$ constructed above, and using semi-continuity and Lemma \ref{lemma:2}-(iii) and (ii), respectively. Property (i) follows by Lemma \ref{lemma:2}-(vi), since the given $\chi$ depends only on the Chern classes of the two factors and  on $X_0$, which are constant in the irreducible 
	family $\mathcal M(r)$. 
\end{proof}

We want to prove that the general member $[\cU_r] \in \mathcal M(r)$ corresponds also to a slope--stable bundle $\cU_r$. To this aim we will first 
need the following auxiliary results.

\begin{lemma} \label{lemma:uniquedest}
  Let $r \geqslant2$ be an integer and assume that $[\cF_r] \in \mathcal M(r)^{\rm ext}$ sits in a non--splitting sequence like  \eqref{eq:estensioneL} 
with $[\cU_{r-1}] \in \mathcal M(r-1)$ being slope--stable w.r.t. $\xi$. Then, 
if $\cD$ is a destabilizing subsheaf of $\cF_r$, then $\cD^{\vee} \cong  \cU_{r-1}^{\vee}$ and $(\cF_r/\cD)^{\vee} \cong L_{\epsilon_r}^{\vee}$; if furthermore $\cF_r/\cD$ is torsion--free, then  $\cD \cong  \cU_{r-1}$ and $\cF_r/\cD \cong L_{\epsilon_r}$. 
\end{lemma}

\begin{proof} The reasoning is similar to \cite[Lemma\;4.5]{cfk1}, we will describe it for reader's convenience.  Assume that $\cD$ is a destabilizing subsheaf of $\cF_r$, that is  $0<{\rm rk}(\cD) < {\rm rk}(\cF_r )=r$ and $\mu(\cD) \geqslant\mu=\mu(\cF_r)$. Define the sheaves 
\[ \cQ:={\rm Im}\{\cD \subset \cU_r \to L_{\epsilon_r}\} \; \; \mbox{and} \; \; \cK:={\rm Ker}\{\cD \to \cQ\}\] so that 
 \eqref{eq:estensioneL}  may be put into the following commutative diagram with exact rows and columns:
  \[
	\begin{array}{ccccccc}
        & 0          &     & 0             &     &  0         & \\
        & \downarrow &     & \downarrow    &     & \downarrow &  \\
	0 \to & \cK        & \to & \cD           & \to &  \cQ        & \to 0 \\    
        & \downarrow &     & \downarrow    &     & \downarrow &  \\
	0 \to & \cU_{r-1} &\to  &\cF_r &\to  & L_{\epsilon_r}  & \to  0 \\
	      & \downarrow &     & \downarrow    &     & \downarrow &  \\
  0 \to   & \cK' & \to &\cF_r/\cD & \to  & \cQ'  & \to 0 \\
             & \downarrow &     & \downarrow    &     & \downarrow &  \\
						  & 0          &     & 0             &     &  0         &
	\end{array}
\]
defining the sheaves $\cK'$ and $\cQ'$. We have ${\rm rk}(\cQ) \leqslant 1$.

Assume that ${\rm rk}(\cQ)=0$. Then $\cQ =0$, whence $\cK \cong \cD$ and $\cQ' \cong L_{\epsilon_r}$. Since $\mu(\cK) =\mu(\cD) \geqslant\mu=\mu(\cU_{r-1})$ and $\cU_{r-1}$ is slope--stable,
we must have ${\rm rk}(\cK)={\rm rk}(\cU_{r-1})=r-1$. It follows that ${\rm rk}(\cK')=0$. As 
\[c_1(\cK) =  c_1(\cU_{r-1})-c_1(\cK')=c_1(\cU_{r-1})-D',\]
where $D'$ is an effective divisor supported on the codimension one locus of the support of $\cK'$, we have 
\[ \mu \leqslant \mu(\cK)=\frac{\left(c_1(\cU_{r-1})-D'\right) \cdot \xi^2}{r-1}=
  \frac{c_1(\cU_{r-1})\cdot \xi^2}{r-1}-\frac{D' \cdot \xi^2}{r-1}
= \mu-\frac{D' \cdot \xi^2}{r-1}.\]
  Hence $D'=0$, which means that $\cK'$ is supported in codimension at least two.
  Thus, the sheaves ${\mathfrak ext}^i(\cK',\mathcal O_{X_0})$ are zero, for $i \leqslant 1$, and it follows that $$ \cD^{\vee} \cong \cK^{\vee} \cong \cU_{r-1}^{\vee}\;\; {\rm and} \;\; (\cF_r/\cD)^{\vee} \cong {\cQ'}^{\vee} \cong 
	L_{\epsilon_r}^{\vee},$$as desired. If furthermore $\cF_r/\cD$ is torsion--free, then we must have $\cK'=0$, whence $\cD \cong  \cU_{r-1}$ and
${\cF_r}/\cD \cong L_{\epsilon_r}$.

 Next we prove that ${\rm rk}(\cQ)=1$ cannot happen. Indeed, if ${\rm rk}(\cQ)=1$, then  ${\rm rk}(\cK)={\rm rk}(\cD)-1 \leqslant r-2<r-1={\rm rk}(\cU_{r-1})$
and ${\rm rk}(\cQ')=0$; in particular $\cQ'$ is a torsion sheaf. Since
\[c_1(\cK) =  c_1(\cD)-c_1(\cQ)=c_1(\cD)-c_1(L_{\epsilon_r})+c_1(\cQ')=c_1(\cD)-c_1(L_{\epsilon_r})+D,\]
where $D$ is  either  an effective divisor supported on the codimension-one locus of the support of $\cQ'$ 
 or it is $D=0$ if ${\rm codim}({\rm Supp} (\cQ')) \geqslant 2$. Then, we have
\begin{eqnarray*}
  \mu(\cK) & = & \frac{\Big(c_1(\cD)-c_1(L_{\epsilon_r})+D\Big) \cdot \xi^2}{{\rm rk}(\cK)} \geqslant\frac{\Big(c_1(\cD)-c_1(L_{\epsilon_r})\Big) \cdot \xi^2}{{\rm rk}(\cK)} 
\\
          & = & \frac{\mu(\cD){\rm rk}(\cD)-c_1(L_{\epsilon_r})\cdot \xi^2}{{\rm rk}(\cK)}
= \frac{\mu(\cD){\rm rk}(\cD)-\mu}{{\rm rk}(\cD)-1}
\geq
                \frac{\mu{\rm rk}(\cD)-\mu}{{\rm rk}(\cD)-1} =\mu
                \end{eqnarray*}
                This contradicts the slope--stability of $\cU_{r-1}$.             
\end{proof}

\begin{lemma} \label{lemma:dimU} Let $r \geqslant2$  be an integer. Assume that the general member  $[\cU_{r-1}] \in \mathcal M(r-1)$ corresponds to a 
slope--stable bundle $\cU_{r-1}$. Then  the scheme $\mathcal M(r)$ is  generically  smooth of dimension 
	\begin{eqnarray*}
	\dim (\mathcal M(r) ) = \begin{cases} \frac{(r^2 -1)}{4}(6 b_0 -4), & \mbox{if $r$ is odd}, \\
			 \frac{r^2}{4} (6b_0-4) +1 , & \mbox{if $r$ is even}.
    \end{cases}
    \end{eqnarray*}
    Furthermore $\mathcal M(r)$ properly contains the locally closed subscheme $\mathcal M(r)^{\rm ext}$, namely \linebreak 
		$\dim(\mathcal M(r)^{\rm ext}) < \dim(\mathcal M(r))$.  
\end{lemma}

\begin{proof} Consider the general member $[\cU_r] \in \mathcal M(r)$; then it satisfies $h^0(\cU_r \otimes \cU_r^{\vee})=1$ and 
$h^j(\cU_r \otimes \cU_r^{\vee})=0$ for $j=2,3$, by Lemma \ref{lemma:genUr}.

From the fact that $h^2(\cU_r \otimes \cU_r^{\vee})=0$, it follows that  $\mathcal M(r)$ is  generically  smooth of dimension $\dim (\mathcal M(r)) = h^1(\cU_r \otimes \cU_r^{\vee})$ (cf. 
e.g. \cite[Prop.\;2.10]{c-h-g-s}). On the other hand, since $h^3(\cU_r \otimes \cU_r^{\vee})=0$ and $h^0(\cU_r \otimes \cU_r^{\vee})=1$, we have  
$h^1(\cU_r \otimes \cU_r^{\vee})  =  -\chi(\cU_r \otimes \cU_r^{\vee})+1$. Therefore, the formula concerning $\dim (\mathcal M(r))$ directly follows from Lemma \ref{lemma:genUr}-(i).

Similarly, being slope-stable by assumptions, also the general member $\cU_{r-1}$ of $\mathcal M(r-1)$ satisfies 
	$h^0(\cU_{r-1} \otimes \cU_{r-1}^{\vee})=1$. Thus, using Lemma \ref{lemma:genUr}-(ii), the same reasoning as above 
	shows that
  \begin{equation}\label{eq:dimUr-1}
\dim (\mathcal M(r-1))=  h^1(\cU_{r-1} \otimes \cU_{r-1}^{\vee}) = - \chi(\cU_{r-1} \otimes \cU_{r-1}^{\vee}) +1,   
\end{equation}where $\chi (\cU_{r-1} \otimes \cU_{r-1}^{\vee})$ as in Lemma \ref{lemma:genUr}-(i) (with $r$ replaced by $r-1$).   Morover, by specialization 
of $\cU_{r-1}$ to  $\cG_{r-1}$ and semi-continuity, we have 
  \begin{equation}\label{eq:dimextv}
    \dim (\Ext^1(L_{\epsilon_r},\cU_{r-1}))= h^1(\cU_{r-1} \otimes L_{\epsilon_r}^{\vee})\leqslant h^1(\cG_{r-1} \otimes L_{\epsilon_r}^{\vee}),
  \end{equation} where the latter is as in Lemma \ref{lemma:2}-(ii) (with $r$ replaced by $r-1$). 
	Therefore, by the very definition of $\mathcal M(r)^{\rm ext}$ and by \eqref{eq:dimUr-1}-\eqref{eq:dimextv}, we have 
	\begin{eqnarray*}
  \dim (\mathcal M(r)^{\rm ext}) & \leqslant &  \dim (\mathcal M(r-1)) + \dim (\mathbb P (\Ext^1(L_{\epsilon_r},\cU_{r-1})) \\
          & = & - \chi(\cU_{r-1} \otimes \cU_{r-1}^{\vee}) +1 + h^1(\cU_{r-1} \otimes L_{\epsilon_r}^{\vee}) -1 \\
& \leqslant & - \chi(\cU_{r-1} \otimes \cU_{r-1}^{\vee}) + h^1(\cG_{r-1} \otimes L_{\epsilon_r}^{\vee}).
\end{eqnarray*} On the other hand, from the above discussion, 
\begin{eqnarray*}
\dim (\mathcal M(r)) = - \chi(\cU_{r} \otimes \cU_{r}^{\vee}) +1.
\end{eqnarray*}
Therefore to prove that $\dim (\mathcal M(r)^{\rm ext}) < \dim (\mathcal M(r))$ 
it is enough to  show that for any integer $r \geqslant2$ the following inequality 
\begin{eqnarray*} - \chi(\cU_{r-1} \otimes \cU_{r-1}^{\vee}) + h^1(\cG_{r-1} \otimes L_{\epsilon_r}^{\vee}) < - \chi(\cU_{r} \otimes \cU_{r}^{\vee}) +1
\end{eqnarray*}
holds true. 
Notice that the previous inequality reads also  
\begin{equation}\label{eq:ineqpal}
- \chi(\cU_{r} \otimes \cU_{r}^{\vee}) +1 + \chi(\cU_{r-1} \otimes \cU_{r-1}^{\vee}) - h^1(\cG_{r-1} \otimes L_{\epsilon_r}^{\vee}) >0,
\end{equation} which is satisfied for any $r \geqslant2$, as we can easily see. 

Indeed use Lemmas \ref{lemma:genUr}-(i) and \ref{lemma:2}-(ii): if $r$ is even, the left hand side of \eqref{eq:ineqpal} reads $rb_0 + 2 + \frac{(r-2)}{2}$ which obviously is positive since 
$r, b_0 \geqslant2$; if $r$ is odd, then $r \geqslant3$ and the left hand side of \eqref{eq:ineqpal} reads $(r-1) (2b_0-3) + \frac{(r-3)}{2}$ which obviously is positive 
under the assumptions $r \geqslant3$, $b_0 \geqslant2$. 
\end{proof}

We can now prove slope--stability of the general member of $\mathcal M(r)$.

\begin{prop}\label{prop:slopstab} Let $r \geqslant 1$ be an integer. The general member $[\cU_r] \in \mathcal M(r)$ corresponds to a 
slope--stable bundle. 
\end{prop}
\begin{proof} We use induction on $r$, the result being obviously true for $r=1$, where $\cU_1 = L_1$, $\mathcal M(1) = \{L_1\}$ is a singleton, 
and $\mathcal M(1)^{\rm ext} = \emptyset$.

Assume therefore $r \geqslant2$ and that the general member of $\mathcal M(r)$ is not slope--stable, whereas the general member of $\mathcal M(r-1)$ is. Then, similarly as in 
\cite[Prop.\;4.7]{cfk1}, we may find a one-parameter family of bundles $\{\cU_r^{(t)}\}$ over the unit disc $\Delta$ such that $\cU_r^{(t)}$ is a general member of $\mathcal M(r)$ for $t \neq 0$ and $\cU_r^{(0)}$ lies in $\mathcal M(r)^{\rm ext}$, and such that we have a destabilizing sequence
\begin{equation} \label{eq:destat1} 
    0 \to \cD^{(t)} \to \cU_r^{(t)} \to \cQ^{(t)} \to 0 
  \end{equation}
  for $t \neq 0$, which we can take to be saturated, that is, such that $\cQ^{(t)}$ is torsion free, whence so that $\cD^{(t)}$ and $\cQ^{(t)}$ are (Ulrich) vector bundles  (see \cite[Thm. 2.9]{c-h-g-s} or \cite[(3.2)]{b}).

  The limit of $\mathbb P(\cQ^{(t)}) \subset \mathbb P(\cU_r^{(t)})$ defines a subvariety of $\mathbb P(\cU_r^{(0)})$ of the same dimension as  $\mathbb P(\cQ^{(t)})$,   whence a coherent sheaf 
	$\cQ^{(0)}$ of rank ${\rm rk}(\cQ^{(t)})$ with a surjection $\cU_r^{(0)} \to \cQ^{(0)}$. Denoting by $\cD^{(0)}$ its kernel, we have
${\rm rk}(\cD^{(0)})={\rm rk}(\cD^{(t)})$ and $c_1(\cD^{(0)})=c_1(\cD^{(t)})$. Hence, \eqref{eq:destat1} specializes to a destabilizing sequence for $t=0$. 
Lemma \ref{lemma:uniquedest} yields  that ${\cD^{(0)}}^{\vee}$ (respectively,
${\cQ^{(0)}}^{\vee}$) is the dual of a member of $\mathcal{M}(r-1)$ (resp., the dual of $L_{\epsilon_r}$).
It follows that ${\cD^{(t)}}^{\vee}$ (resp., ${\cQ^{(t)}}^{\vee}$)
is a deformation of the dual of a member of  $\mathfrak{U}(r-1)$  (resp., a deformation of $L_{\epsilon_r}^{\vee}$), whence that $\cD^{(t)}$ is a deformation of a member of  $\mathcal{M}(r-1)$, as both are locally free, and $\cQ^{(t)} \cong L_{\epsilon_r}$, for the same reason. 

In other words, the general member of $\mathcal M(r)$ 
is an extension of $L_{\epsilon_r}$ by a member of
 $\mathcal{M}(r-1)$. Hence $\mathcal M(r)=\mathcal M(r)^{\rm ext}$, contradicting Lemma \ref{lemma:dimU}. 
\end{proof}

The collection of the previous results gives the following

\begin{theo}\label{thm:general0}  Let  $(X_0, \xi) \cong \scrollcal{E_0}$ be a $3$-fold  scroll over $\FF_0$, with $\mathcal E_0$ as in Assumptions \ref{ass:AB}  Let $\varphi: X_0 \to \FF_0$ be the scroll map and $F$ be the $\varphi$-fibre. Let $r \geqslant2$ be any integer. 
Then the moduli space of rank-$r$ vector bundles $\cU_r$ on $X_0$ which are Ulrich w.r.t. $\xi$ and with first Chern class
\begin{eqnarray*}
c_1(\cU_r) =
    \begin{cases} 
      r \xi + \varphi^*\Oc_{\FF_0}(3, b_0 -3) + \varphi^*\Oc_{\FF_0}\left(\frac{r-3}{2}, \frac{(r-3)}{2}(b_0-2) \right), & \mbox{if $r$ is odd}, \\
      r \xi + \varphi^*\Oc_{\FF_0}\left(\frac{r}{2}, \frac{r}{2}(b_0-2)\right), & \mbox{if $r$ is even}. 
    \end{cases}
    \end{eqnarray*}
    is not empty and it contains a generically smooth component $\mathcal M(r)$ of dimension 
		\begin{eqnarray*}
		\dim (\mathcal M(r) ) = \begin{cases} \frac{(r^2 -1)}{4}(6 b_0 -4), & \mbox{if $r$ is odd}, \\
			 \frac{r^2}{4} (6b_0-4) +1 , & \mbox{if $r$ is even},
    \end{cases}
    \end{eqnarray*} with $b_0 \geqslant 2$ by \eqref{eq:rem:assAB}. Moreover the general point $[\cU_r] \in \mathcal M(r)$  
corresponds to a  slope-stable vector bundle, of slope w.r.t. $\xi$ given by 
$\mu(\cU_r) = 8b_0 - k_0 -3$. 
\end{theo}
\begin{proof} It directly follows from Theorem  \ref{prop:rk 2 simple Ulrich vctB e=0;I}, 
\eqref{eq:c1rcaso0}, \eqref{eq:slopercaso0} and from 
Lemmas \ref{lemma:genUr}, \ref{lemma:dimU} and Proposition \ref{prop:slopstab}, where by abuse of notation 
we have used the same symbol $\mathcal M(r)$  for the smooth modular family which gives rise to the generically smooth irreducible component 
of the corresponding moduli space. 
\end{proof}

%
%

\subsection{Higher rank Ulrich vector bundles on $3$-fold scrolls over $\FF_e$, $e>0$}\label{S:subs52}  Here we focus on the case 
$e>0$. The strategy we will use is slightly different from that used in \S\;\ref{S:subs51}.More precisely, as before  we will inductively define irreducible families of vector bundles 
on $(X_e, \xi)$ whose general members will  be slope--stable Ulrich bundles obtained, by induction, as deformations of extensions of lower ranks Ulrich  bundles. The main difference with respect to the  case $e=0$ is that there are no Ulrich line bundles w.r.t. $\xi$ on $(X_e, \xi)$ when $e >0$, as it follows 
from Theorem \ref{prop:LineB} (cf. {\bf Main Theorem}-(a)). 

Therefore, in the even rank case,  our starting point for 
the inductive process will be given by the use of  rank-$2$ Ulrich vector bundles as in Theorem \ref{thm:rk 2 vctB e>0}. Extensions, recursive procedures, deformations and moduli theory, will then allow us to construct slope-stable Ulrich vector bundles on $(X_e, \xi)$ of even ranks $r = 2h$, for any $h \geqslant  2$, and to study their modular  components. 

 For odd ranks, instead, we  will use some results in \cite{ant} concerning rank-$3$ vector bundles over $\FF_e$ which turn out to be Ulrich w.r.t. 
$c_1(\mathcal E_e) = 3 C_e + b_e f = \Oc_{\FF_e}(3,b_e)$, for any $e>0$, and then we will apply Theorem \ref{pullback} to obtain rank-$3$ vector 
bundles on $(X_e, \xi)$, $e >0$, which are Ulrich w.r.t. $\xi$ (cf. Theorem \ref{thm:antonelli3}).  Then such rank-$3$ bundles along with  $[\mathcal U_2] \in \mathcal M$  general as in Theorem \ref{thm:rk 2 vctB e>0}, will allow us to apply an inductive process on $X_e$ also for odd ranks $r = 2h +1$, for any $h \geqslant 1$. 
In order to avoid confusion, we will separately treat even ranks $r = 2h$ and odd ranks $r = 2h +1$, for any $h \geqslant 1$.

\medskip

\noindent
{\underline {\bf Even ranks}}: in the even rank cases, set $r = 2 h$, for $h \geqslant 1$ an integer. We start by defining the irreducible scheme $\mathcal M(2)$ to be the component $\mathcal M$ as in Theorem \ref{thm:rk 2 vctB e>0}. Recall that 
$\mathcal M(2) = \mathcal M$ is generically smooth, of dimension $\dim(\mathcal M(2)) = 6b_e-9e-3$ and that the general member $[\cU_2] \in \mathcal M(2)$ 
is a rank-$2$ vector bundle on $X_e$ which is Ulrich and slope-stable w.r.t. $\xi$, of slope $\mu(\cU_2) = 8 b_e - k_e - 12 e - 3$,  whose first Chern class is $c_1(\cU_2)=2\xi+\varphi^*\Oc_{\FF_e}(1,b_e-e-2)$.

Assume by induction we have constructed an irreducible scheme $\mathcal M(2h-2)$, for some $h \geqslant2$; similarly as in \cite{cfk3} we define $\mathcal M(2h)$ 
to be the (possibly empty a priori) component of the moduli space of Ulrich bundles on $(X_e,\xi)$ containing bundles $\cF_{2h}$ 
that are  non--trivial extensions of the form
\begin{equation} \label{eq:estensione}
0 \to \cU'_2 \to \cF_{2h} \to \cU_{2h-2} \to 0, 
\end{equation}
with $[\cU'_2] \in \mathcal M(2)$, $[\cU_{2h-2}] \in \mathcal M(2h-2)$ and such that $\cU'_2 \not \cong \cU_{2h-2}$ when $h=2$.  
Similarly as in the the case $e=0$ in \S\;\ref{S:subs51}, we let $\mathcal M(2h)^{\rm ext}$ denote the locus in $\mathcal M(2h)$ of bundles that are  
non-trivial  extensions of the form \eqref{eq:estensione}. 

In the next results we will prove that non-trivial extensions as in \eqref{eq:estensione} always exist and that 
$\mathcal M(2h)^{\rm ext} \neq \emptyset$, so in particular $\mathcal M(2h) \neq \emptyset$, for any $h \geqslant2$. 
In statements and proofs below we will use the following notation: $\cU'_2$ will correspond to a general member of $\mathcal M(2)$ and $\cU_{2h-2}$ to a general member of 
$\mathcal M(2h-2)$, with $\cU'_2 \not \cong \cU_{2h-2}$ when $h=2$.  We will denote by $\cF_{2h}$ a general member of  $\mathcal M(2h)^{\rm ext}$ and,  
in bounding cohomologies, we will use the fact that $\cU_{2h}$ {\em specializes} to $\cF_{2h}$  in an irreducible flat family.

All vector bundles $\cF_{2h}$, $h \geqslant2$, recursively defined as in \eqref{eq:estensione} are of rank $2h$ and 
Ulrich w.r.t. $\xi$, since extensions of bundles which are Ulrich w.r.t. $\xi$ are again Ulrich w.r.t. $\xi$. Their first Chern class is given by 
\begin{equation} \label{eq:c1rcasoe}
    c_1(\cF_{2h}): = 2h\xi+\varphi^*\Oc_{\FF_e} (h,h(b_e-e- 2)) = h\;c_1(\cU_2),
 \end{equation} where $h = \frac{r}{2}$,  whose slope w.r.t. $\xi$ is  
	\begin{equation}\label{eq:slopercasoe}
	\mu(\cF_{2h}) = 8 b_e - k_e - 12 e - 3. 
	\end{equation}From Theorem \ref{thm:stab}-(a), 
	any such bundle $\cF_{2h}$ is strictly semistable and slope-semistable, being extensions of Ulrich bundles of the same slope.

\begin{lemma} \label{lemma:indu}
Let $h \geqslant1$ be an integer and  assume  $\mathcal M (2k) \neq \emptyset$ for all $1 \leqslant k \leqslant h$. Then
 \begin{itemize}
 \item[(i)] $h^j(\cU_{2h} \otimes {\cU'_2}^{\vee})=h^j(\cU'_2 \otimes \cU_{2h}^{\vee})=0$ for $j=2,3$,
  \item[(ii)] $\chi(\cU_{2h} \otimes {\cU'_2}^{\vee})=\chi(\cU'_2 \otimes \cU_{2h}^{\vee})=h (4 + 9e - 6 b_e)$,
    \item[(iii)] $h^j(\cU_{2h} \otimes \cU_{2h}^{\vee})=0$ for $j=2,3$,
 \item[(iv)] $\chi(\cU_{2h} \otimes \cU_{2h}^{\vee})=h^2 (4 + 9e - 6 b_e)$.
\end{itemize}
\end{lemma}

\begin{proof} For $h=1$, (iii) and (iv)  follow from the proof of Theorem \ref{thm:rk 2 vctB e>0}. As for 
(i), the vanishings hold  when $\cU_2' = \cU_2$ once again by the proof of Theorem \ref{thm:rk 2 vctB e>0}, and thus, by semi-continuity, they also hold for a general pair 
  $([\cU_2'], [\cU_2]) \in \mathcal M (2) \times \mathcal M (2)$. Similarly, (ii)  follows from the proof of Theorem \ref{thm:rk 2 vctB e>0}, 
	since the given $\chi$ is constant as $\cU_2$ and $\cU'_2$ vary in $ \mathcal M (2)$.

We now prove the statements for any integer $h \geqslant2$ by induction. Assume therefore that they are  satisfied for all positive integers $k$ less than $h$. 

  (i) Let $j \in \{2,3\}$. By specialization and \eqref{eq:estensione} we have
  \[ h^j(\cU_{2h} \otimes {\cU'_2}^{\vee}) \leqslant h^j(\cF_{2h}\otimes {\cU'_2}^{\vee}) \leqslant h^j(\cU'_2\otimes {\cU'_2}^{\vee}) + h^j(\cU_{2h-2}\otimes {\cU'_2}^{\vee}),\]
  and the latter are $0$ by induction. Similarly, by specialization and the dual of \eqref{eq:estensione} we have
  \[ h^j(\cU'_2 \otimes {\cU_{2h}}^{\vee}) \leqslant h^j(\cU'_2 \otimes {\cF_{2h}}^{\vee}) \leqslant h^j(\cU'_2\otimes {\cU'_2}^{\vee}) + h^j(\cU'_2\otimes {\cU_{2h-2}}^{\vee}),\]
  which are again $0$ by induction.

(ii) By specialization, \eqref{eq:estensione} and induction we have
\begin{eqnarray*}
  \chi(\cU_{2h} \otimes {\cU'_2}^{\vee}) & = & \chi(\cF_{2h}\otimes {\cU'_2}^{\vee}) = \chi(\cU'_2\otimes {\cU'_2}^{\vee}) + \chi(\cU_{2h-2}\otimes {\cU'_2}^{\vee}) \\
  & =  & (4 + 9e - 6b_e) + (h-1)(4 + 9e - 6b_e)= h(4 + 9e - 6b_e).
  \end{eqnarray*} 
	Likewise, by specialization, the dual of \eqref{eq:estensione} and induction, the same holds for $\chi(\cU'_2 \otimes {\cU_{2h}}^{\vee})$. 
  
  (iii) Let $j = 2,3$; by specialization, \eqref{eq:estensione} and its dual we have
  \begin{eqnarray*}
    h^j(\cU_{2h} \otimes \cU_{2h}^{\vee}) & \leqslant & h^j(\cF_{2h} \otimes \cF_{2h}^{\vee}) \leqslant
                                 h^j(\cU'_2 \otimes \cF_{2h}^{\vee})+h^j(\cU_{2h-2} \otimes \cF_{2h}^{\vee}) \\
                        & \leqslant & h^j(\cU'_2 \otimes {\cU'_2}^{\vee})+h^j(\cU'_2 \otimes \cU_{2h-2}^{\vee})+
                                 h^j(\cU_{2h-2} \otimes {\cU'_2}^{\vee}) + h^j(\cU_{2h-2} \otimes \cU_{2h-2}^{\vee}),
    \end{eqnarray*}
which are all $0$ by induction. 

(iv) By specialization, \eqref{eq:estensione} and its dual we have
  \begin{eqnarray*}
    \chi(\cU_{2h} \otimes \cU_{2h}^{\vee}) & = & \chi(\cF_{2h} \otimes \cF_{2h}^{\vee}) =
                                 \chi(\cU'_2 \otimes \cF_{2h}^{\vee})+\chi(\cU_{2h-2} \otimes \cF_{2h}^{\vee}) \\
                   & = & \chi(\cU'_2 \otimes {\cU'_2}^{\vee})+\chi(\cU'_2 \otimes \cU_{2h-2}^{\vee})+ \chi(\cU_{2h-2} \otimes {\cU'_2}^{\vee}) + \chi(\cU_{2h-2} \otimes \cU_{2h-2}^{\vee}).
    \end{eqnarray*}
    By induction, this equals $(4 + 9e - 6b_e) + 2 (h-1)(4 + 9e - 6b_e)+ (h-1)^2(4 + 9e - 6b_e)=h^2(4 + 9e - 6b_e)$.
   \end{proof}

\begin{prop} \label{prop:h+1}
  For all integers $h \geqslant1$ the scheme $\mathcal M (2h)$ is not empty
and its general member $[\cU_{2h}] \in \mathcal M (2h)$ corresponds to a rank-$2h$ vector bundle $\cU_{2h}$ which is Ulrich w.r.t. $\xi$and which satisfies $$c_1(\cU_{2h})= 2h\xi+\varphi^*\Oc_{\FF_e}(h,h(b_e-e-2)) \; {\rm and} \; 
h^j(\cU_{2h} \otimes \cU_{2h}^{\vee})=0, \; j=2,3.$$ 
\end{prop}

\begin{proof} We prove this by induction on $h$, the case $h=1$ being satisfied by the choice of $\mathcal M (2)$.  
Therefore, let $h \geqslant2$; for general $[\cU_{2h-2}] \in
  \mathcal M (2h-2)$ and $[\cU'_2] \in
  \mathcal M (2)$, one has 
	\begin{eqnarray*}
	\dim ({\rm Ext}^1(\cU_{2h-2},\cU'_2)) = h^1(\cU'_2 \otimes \cU_{2h-2}^{\vee}).
	\end{eqnarray*}
	By Lemma \ref{lemma:indu}-(i) we have that 
	$h^j(\cU'_2 \otimes \cU_{2h-2}^{\vee}) = 0$, for $j=2,3$. Therefore 
	\begin{eqnarray*}
	\chi(\cU'_2 \otimes \cU_{2h-2}^{\vee}) = h^0(\cU'_2 \otimes \cU_{2h-2}^{\vee}) - h^1(\cU'_2 \otimes \cU_{2h-2}^{\vee})
	\end{eqnarray*}
	so, by specialization and invariance of $\chi$ in irreducible families, we have
  \begin{eqnarray} \label{eq:dimext}
    \dim ({\rm Ext}^1(\cU_{2h-2},\cU'_2)) & = & h^1(\cU'_2 \otimes \cU_{2h-2}^{\vee})  \\ \nonumber
		                                 & = &  -\chi(\cU'_2 \otimes \cU_{2h-2}^{\vee})+h^0(\cU'_2 \otimes \cU_{2h-2}^{\vee}) \\ \nonumber
	                                   &\geqslant&  -\chi(\cU'_2 \otimes \cU_{2h-2}^{\vee}) = - \chi(\cU'_2 \otimes \cF_{2h-2}^{\vee}) = \\ \nonumber
																		 & = &  (h-1) (6b_e - 9e -4) >0  \nonumber
    \end{eqnarray} the latter equality following from Lemma \ref{lemma:indu}-(ii) (with $h$ replaced by $h-1$) whereas the last strict inequality 
		following from $h \geqslant2$. Hence, by its very definition, one has that $\mathcal M (2h)^{\rm ext}$, and so also $\mathcal M (2h)$, is not empty.
		
		The members of $\mathcal M (2h)$ have rank $2h$ and first Chern class $2h\xi+\varphi^*\Oc_{\FF_e}(h,h(b_e-e-2))$ as in \eqref{eq:c1rcasoe}, since 
		$c_1 (\cU_{2h}) = c_1(\cF_{2h})$ being constant in $\mathcal M (2h)$. It is  immediate  that extensions of Ulrich bundles are still Ulrich, so the general member $\cU_{2h}$ of 
$\mathcal M(2h)$ is an Ulrich bundle. It also satisfies 
$h^j(\cU_{2h} \otimes \cU_{2h}^{\vee})=0$ for $j=2,3$ by Lemma \ref{lemma:indu}-(iii).
\end{proof}

We need to prove that the general member of $\mathcal M(2h)$ corresponds to a slope--stable vector bundle,  
that $\mathcal M(2h)$ is generically smooth and  we need to compute the dimension at its general point $[\cU_{2h}]$. We will again prove all these facts by induction on $h$. Similarly as in the case $e=0$, we  need the following auxiliary result.

\begin{lemma} \label{lemma:uniquedest1}
Let $\cF_{2h}$ correspond to a general member of  $\mathcal M(2h)^{\rm ext}$, sitting in an extension like \eqref{eq:estensione}.
Assume furthermore that $\cU'_2$ and $\cU_{2h-2}$ are slope--stable.  Let $\cD$ be a destabilizing subsheaf of 
$\cF_{2h}$. Then $\cD^{\vee}  \cong  {\cU_2'}^{\vee}$ and $\left(\cF_{2h}/\cD\right)^{\vee} \cong \cU_{2h-2}^{\vee}$. 
\end{lemma}

\begin{proof} The proof is almost identical to that of Lemma \ref{lemma:uniquedest}, so the reader is referred therein. 
\end{proof}

\begin{prop} \label{prop:h+1-II} For all integers $h \geqslant1$ the scheme $\mathcal M (2h)$ is not empty, generically smooth 
of dimension 
\begin{eqnarray*} \dim (\mathcal M(2h)) = h^2(6b_e - 9e \color{black} -4) +1.
\end{eqnarray*}
Its general member corresponds to a slope-stable bundle $\cU_{2h}$ 
whose slope w.r.t. $\xi$ is $\mu(\cU_{2h}) = 8b_e - k_e - 12 e -3$. Furthermore, $\mathcal M (2h)$ properly contains the locally closed 
subscheme $\mathcal M (2h)^{\rm ext}$, namely $\dim(\mathcal M (2h)^{\rm ext}) < \dim(\mathcal M (2h))$. 
\end{prop}

\begin{proof} We prove this by induction on $h$, the case $h=1$ being satisfied by $\mathcal M(2)$ as in Theorem \ref{thm:rk 2 vctB e>0}. 

Let therefore $h \geqslant2$ and assume that we have proved the lemma for all positive integers $k \leqslant h-1$;  we will prove it for $h$.

The slope of the members of  $\mathcal M (2)$ and $\mathcal M (2h-2)$ are both equal to $ 8b_e - k_e - 12 e -3$ as in \eqref{eq:slopercasoe}. 
Thus, by \cite[Lemma\;4.2]{c-h-g-s}, the general member $[\cF_{2h}] \in \mathcal M (2h)^{\rm ext}$ corresponds to a simple bundle. Hence, by semi-continuity, 
also the general member $[\cU_{2h}]\in \mathcal M (2h)$ corresponds to a simple bundle which also satisfies $h^j(\cU_{2h} \otimes \cU_{2h}^{\vee})=0$, $j=2,3$, by Lemma \ref{lemma:indu}-(iii).
 
Therefore  $\mathcal M (2h)$ is smooth at $[\cU_{2h}]$ (see, e.g., \cite[Prop. 2.10]{c-h-g-s}) with
\begin{eqnarray} \label{eq:dimUh+1}
  \dim(\mathcal M (2h))&=&h^1(\cU_{2h} \otimes \cU_{2h}^{\vee})=-\chi(\cU_{2h} \otimes \cU_{2h}^{\vee})+h^0(\cU_{2h} \otimes \cU_{2h}^{\vee})\\
\nonumber  &=&h^2(6b_e - 9e -4)+1,
\end{eqnarray} using the facts that $h^0(\cU_{2h} \otimes \cU_{2h}^{\vee})=1$ as $\cU_{2h}$ is simple, and that
$\chi(\cU_{2h} \otimes \cU_{2h}^{\vee})=h^2(4 + 9e - 6b_e)$ by Lemma \ref{lemma:indu}-(iv). This proves that $\mathcal M(2h)$ is generically smooth 
of the stated dimension. 

Finally, we prove that $\cU_{2h}$ general is slope--stable and that $\dim(\mathcal M (2h)^{\rm ext}) < \dim(\mathcal M (2h))$.
 If $\cU_{2h}$ general were not slope-stable then, as in the proof of Proposition \ref{prop:slopstab}, 
we could find a one-parameter family of bundles 
$\{\cU_{2h}^{(t)}\}$ over the disc $\Delta$ such that $\cU_{2h}^{(t)}$ is a general member of $\mathcal M (2h)$ for $t \neq 0$ and $\cU_{2h}^{(0)}$ lies in $\mathcal M (2h)^{\rm ext}$, and such that we have a destabilizing sequence
\begin{equation} \label{eq:destat1bis} 
    0 \to \cD^{(t)} \to \cU_{2h}^{(t)} \to \cG^{(t)} \to 0
  \end{equation} for $t \neq 0$, which we can take to be saturated, that is, such that $\cG^{(t)}$ is torsion free, whence so that $\cD^{(t)}$ and $\cG^{(t)}$ are (Ulrich) vector bundles  
	(see \cite[Thm. 2.9]{c-h-g-s} or \cite[(3.2)]{b}). The limit of $\Pp(\cG^{(t)}) \subset \Pp(\cU_{2h}^{(t)})$ defines a subvariety of $\Pp(\cU_{2h}^{(0)})$ of the same dimension as 
	$\Pp(\cG^{(t)})$, whence a coherent sheaf $\cG^{(0)}$ of rank ${\rm rk} (\cG^{(t)})$ with a surjection $\cU_{2h}^{(0)} \twoheadrightarrow \cG^{(0)}$. 
	Denoting by $\cD^{(0)}$ its kernel, we have ${\rm rk} (\cD^{(0)})= {\rm rk} (\cD^{(t)})$ and 
	$c_1(\cD^{(0)})=c_1(\cD^{(t)})$. Hence, \eqref{eq:destat1bis} specializes to a destabilizing sequence for $t=0$. 
	
Lemma \ref{lemma:uniquedest1} yields  that ${\cD^{(0)}}^{\vee}$ (resp., ${\cG^{(0)}}^{\vee}$) 
is the dual of a member of $\mathcal M (2)$ (resp., of $\mathcal M (2h)$). It follows that
${\cD^{(t)}}^{\vee}$ (resp., ${\cG^{(t)}}^{\vee}$)
is a deformation of the dual of a member of $\mathcal M (2))$ (resp., of $\mathcal M (2h)$), whence that ${\cD^{(t)}}$ (resp., ${\cG^{(t)}}$) is a deformation of a member of $\mathcal M (2)$ (resp., $\mathcal M (2h)$), as both are locally free. It follows that 
$[\cU_{2h}^{(t)}] \in \mathcal M (2h)^{\rm ext}$ for $t \neq 0$. Thus,
\begin{equation} \label{eq:sonuguali}
  \mathcal M (2h)^{\rm ext}=\mathcal M (2h).
\end{equation}

On the other hand we have
\begin{equation} \label{eq:dimext2}
  \dim (\mathcal M (2h)^{\rm ext}) \leqslant \dim (\Pp(\Ext^1(\cU_{2h-2},\cU'_2))) +\dim (\mathcal M (2h-2)) +\dim (\mathcal M (2)),
\end{equation}
for general 
  $[\cU_{2h-2}] \in \mathcal M (2h-2)$
and $[\cU'_2]\in \mathcal M (2)$. As $\cU_{2h-2}$ and  $\cU'_2$ are slope--stable by induction, of the same slope, we have 
$h^0(\cU'_2 \otimes \cU_{2h-2}^{\vee})=0$. Lemma \ref{lemma:indu}-(i), (ii) and (iii) thus yield 
\[  h^1(\cU'_2 \otimes \cU_{2h-2}^{\vee})  =  -\chi(\cU'_2 \otimes \cU_{2h-2}^{\vee}) = (h-1)(6b_e - 9e -4).
\]Hence, by \eqref{eq:dimext2} and \eqref{eq:dimUh+1} we have
\begin{eqnarray*}
  \dim (\mathcal M (2h)^{\rm ext}) & \leqslant & (h-1)(6b_e - 9e -4) -1 +\left[(h-1)^2(6b_e - 9e -4) +1 \right]+ (6b_e - 9e -3) \\
  & = & (h^2-h+1)(6b_e - 9e -4)+1  <  h^2(6b_e - 9e -4)+1=\dim (\mathcal M (2h)),
\end{eqnarray*}
as it easily follows from the fact that $h \geqslant2$. The previous inequality shows that \linebreak 
$\dim (\mathcal M (2h)^{\rm ext}) < \dim (\mathcal M (2h))$, as stated; in particular \eqref{eq:sonuguali} is a contradiction, 
which forces also $\cU_{2h}$ general to be slope-stable.
\end{proof}


\medskip

\noindent
{\underline {\bf Odd ranks}}:  in odd ranks, set $r = 2 h +1$, for $h \geqslant 1$ an integer. The first step is given by the following result.

\begin{theo}\label{thm:antonelli3}  Let  $(X_e, \xi) \cong \scrollcal{E_e}$ be a $3$-fold  scroll over $\FF_e$, with $e> 0$,  and $\mathcal E_e$ be as in Assumptions \ref{ass:AB}. Let $\varphi: X_e \to \FF_e$ be the scroll map and $F$ be the $\varphi$- fibre. Then the moduli space of  rank-$3$ vector bundles on $X_e$,  which are Ulrich w.r.t. $\xi$, with first Chern class
\begin{equation}\label{eq:chern3e}
c_1 = 3\xi+\varphi^*\Oc_{\FF_e}(3,b_e-3)
\end{equation}
is not empty and it contains a generically smooth component $\mathcal M(3)$ of dimension 
$$\dim (\mathcal M(3)) = 2(6b_e-9e-4),$$whose general point $[ \cU_3 ]$ 
corresponds to a slope-stable vector bundle $\cU_3$ of slope w.r.t. $\xi$
\begin{eqnarray}\label{eq:slope3e}
\mu( \cU_3) = 8 b_e - k_e - 12 e - 3.
\end{eqnarray}
\end{theo}

Recalling the expression of the first Chern class in {\bf Main Theorem}-(c) for $r=3$, this coincides with that 
in \eqref{eq:chern3e} since $\frac{r-3}{2} = 0$ in this case.

\begin{proof}[Proof of Theorem \ref{thm:antonelli3}] Similarly as in the proof of Theorem \ref{thm:rk 2 vctB e>0}, we consider a rank-3 vector bundle $\cA_3$ on $\FF_e$ which is Ulrich w.r.t. 
$c_1(\mathcal E_e) = 3C_e+b_ef = \Oc_{\FF_e}(3,b_e)$, for any $e >0$. 

Such a bundle certainly exists, as it follows from \cite[Prop.\;5.3,\;Thm.\;6.2]{ant}; indeed, using same notation therein, we let $r=3$, $D \in |\Oc_{\FF_e}(\alpha,\beta)|$ be any divisor on $\FF_e$,  polarization $c_1(\mathcal E_e) = 3C_e + b_ef = \Oc_{\FF_e}(3,b_e)$  (in \cite{ant} the polarization is denoted by $h$), it follows  that  any pair $(3,D)$, satisfying 
\[T=24e+\frac{b_e}{3}(\alpha+3) \in \mathbb Z\;\;  \mbox{and} \;\; 6+\frac{9e}{b_e} \leqslant \alpha \leqslant 15-\frac{9e}{b_e}\] 
 and such that 
  \begin{eqnarray*} 
 D (3C_e+b_ef)= \frac{3}{2}\big(2(3C_e+b_ef)^2
+ (3C_e+b_ef) K_{\FF_e}\big),
 \end{eqnarray*} is an {\em admissable Ulrich pair} on $\FF_e$ w.r.t. the polarization $\Oc_{\FF_e}(3,b_e)$ (cf. \cite[Def.\;5.1,\;Prop.\;5.3]{ant}).

Because $b_e \geqslant 3e+2$ from \eqref{eq:rem:assAB}, then one has $7 \leqslant \alpha\leqslant 14$.  
If we therefore take e.g. $\alpha=12$ and $\beta=4b_e-3$, an easy check shows that  the previous relations  hold true. Hence the pair $(3, \Oc_{\FF_e}(12, 4b_e -3)) = (3,\;12C_e+(4b_e-3)f)$ is an admissable Ulrich pair, thus $\FF_e$ certainly supports rank-$3$  vector bundles, say $\cA_3$, which are Ulrich w.r.t. $c_1(\mathcal E_e)=3C_e+b_ef$,  such 
that $c_1(\cA_3) = 12C_e+(4b_e-3)f = \Oc_{\FF_e}(12, 4b_e-3)$  and which are given as cokernels of appropriate injective vector bundle 
maps, (see \cite[Thm.\;4.1,\;(4.1) and Thm.\;6.2]{ant}). 
They moreover satisfies $Ext^2(\cA_3,\cA_3) = h^2(\cA_3 \otimes \cA_3^{\vee}) =0$ (cf. \cite[Lemma\;6.3]{ant}). From \cite[Prop.\;6.4]{ant}, any such bundle $\cA_3$  
is slope-stable,  hence simple i.e. $h^0(\cA_3 \otimes \cA_3^{\vee})=1$. These bundles belong to the moduli space $\mathcal M_{c_1(\mathcal E_e)}^U(3,c_1,c_2)$ of rank-3 vector bundles of given Chern classes $c_1$, $c_2$, which are Ulrich w.r.t. $c_1(\mathcal E_e)$, which is smooth, irreducible, of dimension 
$$\dim (\mathcal M_{c_1(\mathcal E_e)}^U(3,c_1,c_2)) = h^1(\cA_3 \otimes \cA_3^{\vee}) = 1 - \chi(\cA_3 \otimes \cA_3^{\vee}) = 2 (6b_e - 9e -4)$$(cf. \cite[Prop.\;6.4]{ant}). All points of $\mathcal M_{c_1(\mathcal E_e)}^U(3,c_1,c_2)$ correspond to slope--stable Ulrich bundles, since on $\FF_e$ there are no Ulrich line bundles w.r.t. $c_1(\mathcal E_e) = 3 C_e + b_e f = \Oc_{\FF_e}(3,b_e)$ from \cite[Thm.\;2.1]{a-c-mr} (cf. Theorem \ref{thm:stab}-(b) or 
\cite[Sect.3,\;(3.2)]{b}). 

Using Theorem \ref{pullback}, we therefore first consider on $\FF_e$ bundles $\cH_3:=  \cA_3(-c_1({\mathcal E}_e))$, which are of rank 3, with first Chen class 
$$c_1(\cH_3) = c_1(\cA_3) - 3 c_1(\mathcal E_e) = 3C_e + (b_e-3)f = \Oc_{\FF_e}(3,b_e-3),$$and then, on $X_e$, we take 
\begin{equation}\label{eq:cojo1}
\cV_3 = \xi \otimes \varphi^*(\cH_3),
\end{equation}which are rank-3 vector bundles on $X_e$ which are Ulrich w.r.t. $\xi$, as it follows from Theorem \ref{pullback}, 
whose first Chern class $c_1 := c_1 (\cV_3)$ is as in \eqref{eq:chern3e}. From 
{\bf Main Theorem}--(a), when $e>0$, there are no Ulrich line bundles w.r.t. $\xi$ thus, once again by Theorem \ref{thm:stab}-(b) or 
\cite[Sect.3,\;(3.2)]{b}, $\cV_3$ is also slope-stable w.r.t. $\xi$, whose slope is as in \eqref{eq:slope3e}; in particular it is also simple. 
Moreover, by Leray's isomorphism, projection formula and independence on the twists, one has 
\begin{equation}\label{eq:cojo0}
h^j(X_e, \cV_3\otimes \cV_3^{\vee}) = h^j(\FF_e, \cH_3\otimes \cH_3^{\vee}) = h^j(\FF_e, \cA_3\otimes \cA_3^{\vee}), \; 0 \leqslant j \leqslant 3,
\end{equation}which implies that bundles $\cV_3$ fill--up a smooth modular component $\mathcal M(3)$ of the moduli space of rank-3 vector bundles 
on $X_e$, which are Ulrich w.r.t. $\xi$, whose first Chern class and slope w.r.t. $\xi$ are as in \eqref{eq:chern3e}, \eqref{eq:slope3e}, respectively, and 
whose dimension is  
$$\dim (\mathcal M(3)) = \dim (\mathcal M_{c_1(\mathcal E_e)}^U(3,c_1,c_2)) = 2 (6b_e - 9e -4)$$as in \cite[Prop.\;6.4]{ant}. 
\end{proof}

Together with $\mathcal M(3)$ as in Theorem \ref{thm:antonelli3}, consider also the irreducible scheme $\mathcal M(2)$ to be the component $\mathcal M$ as in Theorem \ref{thm:rk 2 vctB e>0}. Recall that 
$\mathcal M(2)$ is generically smooth, of dimension $\dim(\mathcal M(2)) = 6b_e-9e-3$ and that the general member $[\cU_2] \in \mathcal M(2)$ corresponds to a rank-$2$ vector bundle on $X_e$ which is Ulrich and slope-stable w.r.t. $\xi$, of slope $\mu(\cU_2) = 8 b_e - k_e - 12 e - 3$ and whose first Chern class is $c_1(\cU_2)=2\xi+\varphi^*\Oc_{\FF_e}(1,b_e-e-2)$.

As done in the even rank case, if we assume by induction that we have constructed, for some integer $h \geqslant 2$, an irreducible scheme $\mathcal M(2h-1)$ which is a generically smooth 
modular component of the moduli space of vector bundles of rank $r = 2 h-1$ on $X_e$, which are Ulrich w.r.t. $\xi$, whose first Chern class is as in {\bf Main Theorem}--(c) 
and whose general point $[\cU_{2h-1}] \in  \mathcal M(2h-1)$ is slope-stable, with slope $\mu(\cU_{2h-1}) = 8 b_e - k_e - 12 e - 3$ w.r.t. $\xi$, we may therefore inductively define 
$\mathcal M(2h+1)$ to be the (possibly empty a priori) component of the moduli space of Ulrich bundles on $(X_e,\xi)$ containing bundles $\cF_{2h+1}$ 
that are  non--trivial extensions of the form
\begin{equation} \label{eq:estensioneodd}
0 \to \cU_2 \to \cF_{2h+1} \to \cU_{2h-1} \to 0, 
\end{equation}
with $[\cU_2] \in \mathcal M(2)$, $[\cU_{2h-1}] \in \mathcal M(2h-1)$ general, and we let $\mathcal M(2h +1)^{\rm ext}$ denote the locus in $\mathcal M(2h+1)$ of bundles that are  
non-trivial  extensions of the form \eqref{eq:estensioneodd}. 

In the next results we will prove that non-trivial extensions as in \eqref{eq:estensioneodd} always exist (cf. the proof of Proposition \ref{prop:h+1odd}) 
and that $\mathcal M(2h+1)^{\rm ext} \neq \emptyset$, so in particular $\mathcal M(2h+1) \neq \emptyset$, for any $h \geqslant 2$.

All vector bundles $\cF_{2h+1}$, $h \geqslant 2$, recursively defined as in \eqref{eq:estensioneodd} are of rank $2h+1$ and 
Ulrich w.r.t. $\xi$, since extensions of bundles which are Ulrich w.r.t. $\xi$ are again Ulrich w.r.t. $\xi$. Their first Chern class is given by 
$$c_1(\cF_{2h+1}) = c_1(\cU_2) + c_1(\cU_{2h-1}) = c_1(\cU_3) + (h-1) c_1(\cU_2),$$i.e.  
\begin{equation} \label{eq:c1rcasoeodd}
 c_1(\cF_{2h+1}): = (2h+1)\xi+ \varphi^*\Oc_{\FF_e}(3, b_e -3) + \varphi^*\Oc_{\FF_e}(h-1,(h-1)(b_e- e - 2)),
 \end{equation}where $h = \frac{r-1}{2}$, whose slope w.r.t. $\xi$ is  
	\begin{equation}\label{eq:slopercasoeodd}
	\mu(\cF_{2h+1}) = 8 b_e - k_e - 12 e - 3. 
	\end{equation} From Theorem \ref{thm:stab}-(a), any such bundle $\cF_{2h+1}$ is strictly semistable and slope-semistable, being extension of Ulrich bundles of the same slope w.r.t. $\xi$.

\begin{lemma} \label{lemma:induodd}
Let $h \geqslant 1$ be an integer and  assume  $\mathcal M (2k+1) \neq \emptyset$ for all $1 \leqslant k \leqslant h$. Then
 \begin{itemize}
 \item[(i)] $h^j(\cU_{2h+1} \otimes {\cU_2}^{\vee})=h^j(\cU_2 \otimes \cU_{2h+1}^{\vee})=0$ for $j=2,3$,
  \item[(ii)] $\chi(\cU_{2h+1} \otimes {\cU_2}^{\vee})= (h-1) (4 + 9 e - 6 b_e) + (15e - 10b_e +3)$ whereas \linebreak
	$\chi(\cU_{2} \otimes {\cU}^{\vee}_{2h+1})= (h-1) (4 + 9 e - 6 b_e) + (12e - 8b_e -3)$, 
\item[(iii)] $h^j(\cU_{2h+1} \otimes \cU^{\vee}_{2h+1})=0$ for $j=2,3$,
 \item[(iv)] $\chi(\cU_{2h+1} \otimes \cU^{\vee}_{2h+1})= 1 + 9(h-1) (3e - 2 b_e) + ((h-1)^2 +2) (4 + 9e - 6 b_e)$. 
\end{itemize}
\end{lemma}

\begin{proof} We prove it by induction on $h$. We start with $h=1$, namely $r=3$. 
In this case, (iii) and (iv)  follow from Theorem \ref{thm:antonelli3}. Indeed, take $\cV_3 = \xi \otimes \varphi^*(\mathcal H_3)$ so, 
by \eqref{eq:cojo0} and semi--continuity, one has 
$$h^j(X_e, \cU_{3} \otimes \cU_{3}^{\vee}) = h^j(X_e, \cV_{3} \otimes \cV_{3}^{\vee}) = h^j(\FF_e, \cA_{3} \otimes \cA_{3}^{\vee}) = 0,\; 2 \leqslant j\leqslant 3,$$which proves (iii) for $h=1$; moreover, 
by stability and semi--continuity, one has $h^0(X_e, \cU_{3} \otimes \cU_{3}^{\vee}) = h^0(X_e, \cV_{3} \otimes \cV_{3}^{\vee}) = 1$, therefore from the vanishings above one has  
$$\chi(X_e, \cU_{3} \otimes \cU_{3}^{\vee}) = 1 - h^1(X_e, \cU_{3} \otimes \cU_{3}^{\vee})  = 1 - \dim (\mathcal M(3)) = 1 - 2 (6b_e - 9e - 4) = 1 + 2(4 + 9e - 6b_e),$$which also 
proves (iv) for $h=1$. 

As for (i), by semi--continuity, one has $$h^j (X_e, \cU_3 \otimes \cU^{\vee}_2) \leqslant h^j (X_e, \cV_3 \otimes \cV^{\vee}_2) \;\; {\rm and} \;\;  
h^j (X_e, \cU_2 \otimes \cU^{\vee}_3) \leqslant h^j (X_e, \cV_2 \otimes \cV^{\vee}_3),$$for any $0 \leqslant j \leqslant 3$, where 
$\cV_2 = \cV$ as in the proof of Theorem \ref{thm:rk 2 vctB e>0}, i.e. 
$\cV_2 = \xi \otimes \varphi^*(\mathcal H)$, where $c_1(\mathcal H) = C_e + (b_e-e-2)f = {\mathcal O}_{\FF_e}(1, b_e-e-2)$.  
From Leray's isomorphism, projection formula and invariance under twists, it follows that, for any $0 \leqslant j \leqslant 3$, one has  
$$h^j(X_e, \cV_3 \otimes \cV_2^{\vee}) = h^j(\FF_e, \mathcal H_3 \otimes \mathcal H^{\vee}) = h^j(\FF_e, \cA_3 \otimes \mathcal H_1^{\vee})$$and 
$$h^j(X_e, \cV_2 \otimes \cV_3^{\vee}) = h^j(\FF_e, \mathcal H \otimes \mathcal H_3^{\vee}) = h^j(\FF_e, \mathcal H_1 \otimes \cA_3^{\vee}),$$where 
$\cA_3 = \mathcal H_3(c_1(\mathcal E_e))$ is the rank-3 vector bundle on $\FF_e$ as in the proof of Theorem \ref{thm:antonelli3} whereas $\mathcal H_1= \mathcal H(c_1(\mathcal E_e))$ 
is the rank-2 vector bundle on $\FF_e$ as in \eqref{extensionsuFe}, which are both Ulrich w.r.t. $c_1(\mathcal E_e) = {\mathcal O}_{\FF_e}(3, b_e)$. Therefore, for 
dimension reasons, $h^3(\FF_e, \cA_3 \otimes \mathcal H_1^{\vee}) = h^3(\FF_e, \mathcal H_1 \otimes \cA_3^{\vee}) = 0$, which implies therefore 
$h^3 (X_e, \cU_3 \otimes \cU^{\vee}_2) = h^3(X_e, \cU_2 \otimes \cU^{\vee}_3) = 0$, as desired. 

To prove that  $h^2(\FF_e, \cA_3 \otimes \mathcal H_1^{\vee})=0$ recall that, from \cite[Prop.\;6.4]{ant}, $\cA_3$ arises as the cokernel of a general injective vector bundle map of the form
\begin{equation}\label{eq:cojocaz}
0 \to \mathcal O_{\FF_e}(2, b_e - e -1)^{\oplus \gamma} \to \mathcal O_{\FF_e}(2, b_e - e)^{\oplus \delta} \oplus \mathcal O_{\FF_e}(3, b_e-1)^{\oplus \tau} \to \cA_3 \to 0,
\end{equation}for suitable positive integers $\gamma,\;\delta,\,\tau$ as in \cite[Thm.\;4.1]{ant}, in our situation  one can see that  $\gamma=b_e-3e+3,\;\delta=b-3e,\,\tau=6$. Tensoring \eqref{eq:cojocaz} by $\mathcal H_1^{\vee}$ gives that 
$$h^2(\FF_e, \cA_3 \otimes \mathcal H_1^{\vee}) \leqslant \delta\,h^2(\FF_e,  \mathcal H_1^{\vee} \otimes \mathcal O_{\FF_e}(2, b_e - e)) + \tau\,h^2(\FF_e,  
\mathcal H_1^{\vee} \otimes \mathcal O_{\FF_e}(3, b_e - 1)),$$
therefore to prove that $ h^2(\FF_e, \cA_3 \otimes \mathcal H_1^{\vee}) =0$ 
it is enough to show that 
\begin{equation}\label{eq:cojo4}
h^2(\FF_e,  \mathcal H_1^{\vee} \otimes \mathcal O_{\FF_e}(2, b_e - e)) = h^2(\FF_e,  \mathcal H_1^{\vee} \otimes \mathcal O_{\FF_e}(3, b_e - 1)) = 0.
\end{equation}
Taking into account that $\mathcal H_1$ is of rank 2, i.e. $\mathcal H^{\vee}_1 \cong \mathcal H_1 (- c_1(\mathcal H_1))$, where 
$c_1(\mathcal H_1) = C_e + (b_e - e - 2)f = \mathcal O_{\FF_e}(1, b_e - e -2)$, from \eqref{extensionsuFe} one has that $\mathcal H_1 (- c_1(\mathcal H_1))$ fits 
in 
\begin{equation}\label{eq:cojo3}
0 \to \mathcal O_{\FF_e}(2, e+1) \to \mathcal H_1 (- c_1(\mathcal H_1)) \to I_Z \otimes \mathcal O_{\FF_e}(3, b_e+1)\to 0,
\end{equation}where $Z \subset \FF_e$ is a general zero-dimensional subscheme of $\FF_e$ of length $\ell(Z) = 2b_e - 3e$. Thus,  from \eqref{eq:cojo3}, to prove 
\eqref{eq:cojo4}  it is enough to prove that

$$h^2(\mathcal O_{\FF_e}(4, b_e+1)) = h^2(\mathcal O_{\FF_e}(5, b_e+e)) = h^2(I_Z \otimes \mathcal O_{\FF_e}(5, 2b_e-e+1)) = h^2(I_Z \otimes \mathcal O_{\FF_e}(6, 2b_e))=0,$$
which trivially hold true from either Serre duality on $\FF_e$ or from the use of the exact sequence 
$0 \to I_Z \to \mathcal O_{\FF_e} \to \mathcal O_Z \to 0$. This shows that $h^2(\FF_e, \cA_3 \otimes \mathcal H_1^{\vee})=0$. 

To prove instead that $h^2(\FF_e, \mathcal H_1 \otimes \cA_3^{\vee})=0$, one considers the dual exact sequence of \eqref{eq:cojocaz}, i.e. 
\[
0 \to \cA_3^{\vee} \to \mathcal O_{\FF_e}(-2, - b_e + e)^{\oplus \delta} \oplus \mathcal O_{\FF_e}(-3, - b_e+1)^{\oplus \tau} \to \mathcal O_{\FF_e}(-2, - b_e + e + 1)^{\oplus \gamma}  \to 0,
\]and tensor it by $\mathcal H_1$, which gives
\begin{eqnarray*}
h^2(\FF_e, \mathcal H_1 \otimes \cA_3^{\vee}) \leqslant & \delta\,h^2(\mathcal H_1 \otimes \mathcal O_{\FF_e}(-2, -b_e + e)) + \tau\,
h^2(\mathcal H_1\otimes \mathcal O_{\FF_e}(-3, - b_e + 1)) + \\ 
& \gamma\,h^1(\mathcal H_1 \otimes \mathcal O_{\FF_e}(-2, -b_e + e +1)).
\end{eqnarray*} From \eqref{extensionsuFe}, it is enough to prove 
$$h^2(\mathcal O_{\FF_e}(1, e-1)) = h^2(\mathcal O_{\FF_e}) = h^1 (\mathcal O_{\FF_e}(1, e)) = 0,$$and 
$$h^2(I_Z \otimes \mathcal O_{\FF_e}(2, b_e-1)) = h^2(I_Z \otimes \mathcal O_{\FF_e}(1, b_e-e)) = h^1(I_Z \otimes \mathcal O_{\FF_e}(2, b_e)) =0.$$The vanishings 
of the $h^2$'s easily follow from the same reasoning as above; by Leray's isomorphism and projection formula one gets 
$h^1 (\mathcal O_{\FF_e}(1, e)) = h^1(\mathbb{P}^1, \mathcal O_{\mathbb P^1}(e) \oplus \mathcal O_{\mathbb P^1}) = 0$. At last, if we take into account that 
$Z$ is a zero-dimensional subscheme of length $\ell(Z) = 2b_e - 3e$ of general points on $\FF_e$, then $Z$ imposes independent conditions on the linear system 
$|\mathcal O_{\FF_e}(2, b_e)|$ on $\FF_e$, which is of dimension $3 b_e - 3e + 2 > \ell(Z) = 2b_e - 3e$ since $b_e \geqslant 3e+2$ by \eqref{eq:rem:assAB}; 
this means that $h^1(I_Z \otimes \mathcal O_{\FF_e}(2, b_e)) = h^1(\mathcal O_{\FF_e}(2, b_e))$ and the latter is zero by standard computations. This shows that 
$h^2(\FF_e, \mathcal H_1 \otimes \cA_3^{\vee})=0$, which concludes the proof of (i) for $h=1$.

Finally, to prove (ii) for $h=1$, from invariance of $\chi$ in irreducible families and from above one has 
$$\chi(X_e, \cU_3 \otimes \cU_2^{\vee}) = \chi(X_e, \cV_3 \otimes \cV_2^{\vee}) = \chi(\FF_e, \mathcal H_3 \otimes \mathcal H^{\vee}) = \chi(\FF_e, \cA_3 \otimes \mathcal H_1^{\vee})$$and 
$$\chi(X_e, \cU_2 \otimes \cU_3^{\vee}) =  \chi(X_e, \cV_3 \otimes \cV_2^{\vee}) = \chi(\FF_e, \mathcal H \otimes \mathcal H_3^{\vee}) = \chi(\FF_e, \mathcal H_1 \otimes \cA_3^{\vee}).$$Since 
$\cA_3$ and $\mathcal H_1$ are both Ulrich bundles w.r.t. $c_1(\mathcal E_e) = 3 C_e + b_e f= \mathcal O_{\FF_e}(3, b_e)$ on $\FF_e$, let us consider 
the smooth projective model $(S, \mathcal O_S(1)) \cong (\FF_e, \mathcal O_{\FF_e}(3, b_e))$, which is a surface, in a suitable projective space, of 
degree $d := \deg(S) = (c_1(\mathcal E_e))^2 = 6 b_e - 9e$. Thus, from \cite[Prop.\;2.12]{c-h-g-s}, one has 
$$\chi(\FF_e, \cA_3 \otimes \mathcal H_1^{\vee}) = 3 c_1(\mathcal H_1) \cdot K_S - c_1(\cA_3) \cdot c_1(\mathcal H_1) + 6 (2d-2)$$and 
$$\chi(\FF_e, \mathcal H_1 \otimes \cA_3^{\vee}) = 2 c_1(\cA_3) \cdot K_S - c_1(\cA_3) \cdot c_1(\mathcal H_1) + 6 (2d-2),$$where 
$p_a(S) =1$ and $K_S = - 2 C_e - (e+2)f$. 
Using that $ c_1(\mathcal H_1) = 7 C_e + (3b_e - e - 2) f$ and $c_1(\cA_3) = 12 C_e + (4b_e -3)f$, one gets
$$c_1(\cA_3) \cdot c_1(\mathcal H_1) = 64b_e - 96 e - 45,\; c_1(\cA_3) \cdot K_S = 12 e - 8b_e - 18,$$
$$c_1(\mathcal H_1) \cdot K_S = 9e - 6 b_e - 10, \; 6 (2d-2) = 12(d-1) = 72b_e - 108e -12;$$plugging these computations in the previous formulas, one gets 
$$\chi(\FF_e, \cA_3 \otimes \mathcal H_1^{\vee}) = 15 e - 10 b_e +3 \;\; {\rm and} \;\; \chi(\FF_e, \mathcal H_1 \otimes \cA_3^{\vee}) = 12 e - 8 b_e -3,$$which concludes the proof of 
(ii) for $h=1$.

We now prove by induction that all statements hold true also for any integer $h \geqslant 2$. Assume therefore that they are satisfied for all positive integers $k$ such that 
$1 \leqslant k \leqslant h-1$.

  (i) Let $j \in \{2,3\}$.  By specialization and  \eqref{eq:estensioneodd} tensored with  $ {\cU_2}^{\vee}$, we have \color{black}
  \[ h^j(\cU_{2h+1} \otimes {\cU_2}^{\vee}) \leqslant h^j(\cF_{2h+1}\otimes {\cU_2}^{\vee}) \leqslant h^j(\cU_2\otimes {\cU_2}^{\vee}) + h^j(\cU_{2h-1}\otimes {\cU_2}^{\vee}),\]
  and the latter are $0$ by induction. Similarly, by specialization and using the dual of  \eqref{eq:estensioneodd} tensored with  $ {\cU_2}$ we have
  \[ h^j(\cU_2 \otimes {\cU^{\vee}_{2h+1}}) \leqslant h^j(\cU_2 \otimes {\cF^{\vee}_{2h+1}}) \leqslant h^j(\cU_2\otimes {\cU_2}^{\vee}) + h^j(\cU_2\otimes {\cU^{\vee}_{2h-1}}),\] which are again $0$ by induction.

(ii) By specialization, \eqref{eq:estensioneodd}  tensored with  $ {\cU_2}^{\vee}$ and induction we have 
\begin{eqnarray*}
  \chi(\cU_{2h+1} \otimes {\cU_2}^{\vee}) & = & \chi(\cF_{2h+1}\otimes {\cU_2}^{\vee}) = \chi(\cU_2\otimes {\cU_2}^{\vee}) + \chi(\cU_{2h-1}\otimes {\cU_2}^{\vee}) \\
  & =  & (4 + 9e - 6b_e) + (h-2)(4 + 9e - 6b_e) + (15 e - 10 b_e +3).
  \end{eqnarray*} 
	Likewise, by specialization, the dual of \eqref{eq:estensioneodd} and induction, we have 
	
	\begin{eqnarray*}
  \chi(\cU_2 \otimes {\cU}^{\vee}_{2h+1}) & = & \chi( \cU_2 \otimes {\mathcal F}^{\vee}_{2h+1}) = \chi(\cU_2\otimes {\cU_2}^{\vee}) + 
	\chi(\cU_2 \otimes {\cU}^{\vee}_{2h-1}) \\
  & =  & (4 + 9e - 6b_e) + (h-2)(4 + 9e - 6b_e) + (12 e - 8 b_e -3).
  \end{eqnarray*} 
	
  (iii) Let $j = 2,3$; by specialization, \eqref{eq:estensioneodd} and its dual we have
  \begin{eqnarray*}
    h^j(\cU_{2h+1} \otimes \cU_{2h+1}^{\vee}) & \leqslant & h^j(\cF_{2h+1} \otimes \cF_{2h+1}^{\vee}) \leqslant \\
                              & \leqslant &   h^j(\cU_2 \otimes \cF_{2h+1}^{\vee})+ h^j(\cU_{2h-1} \otimes \cF_{2h+1}^{\vee}) \\
                        & \leqslant & h^j(\cU_2 \otimes {\cU_2}^{\vee})+h^j(\cU_2 \otimes \cU_{2h-1}^{\vee})+ \\
                        &  &    +   h^j(\cU_{2h-1} \otimes {\cU_2}^{\vee}) + h^j(\cU_{2h-1} \otimes \cU_{2h-1}^{\vee}),
    \end{eqnarray*}
which are all $0$ by induction. 

(iv) By specialization, \eqref{eq:estensioneodd} and its dual we have
  \begin{eqnarray*}
    \chi(\cU_{2h+1} \otimes \cU^{\vee}_{2h+1}) & = & \chi(\cF_{2h+1} \otimes \cF^{\vee}_{2h+1}) =
                                 \chi(\cU_2 \otimes \cF^{\vee}_{2h+1})+\chi(\cU_{2h-1} \otimes \cF^{\vee}_{2h+1}) \\
                   & = & \chi(\cU_2 \otimes {\cU_2}^{\vee})+\chi(\cU_2 \otimes \cU^{\vee}_{2h-1})+ \chi(\cU_{2h-1} \otimes {\cU_2}^{\vee}) + \chi(\cU_{2h-1} \otimes \cU^{\vee}_{2h-1}).
    \end{eqnarray*} By induction, this equals
    \begin{eqnarray*} 
	(4 + 9e - 6b_e) + (h-2)(4 + 9e - 6b_e)+ (h-2)(4 + 9e - 6b_e) + (27 e - 18 b_e) + 1 + \nonumber \\
	+ 9 (h-2) (3e - 2 b_e) + ((h-2)^2 + 2) (4 + 9e - 6 b_e) = \nonumber \\
		= 1 + 9(h-1) (3e - 2 b_e) + (4 + 9 e - 6 b_e) ( 1 + h-2 + h-2 + (h-2)^2 +2) = \nonumber \\	
	 = 1 + 9(h-1) (3e - 2b_e) + ((h-1)^2 + 2) (4 + 9e - 6 b_e).  \nonumber 	
	 \end{eqnarray*}
   \end{proof}
\begin{prop} \label{prop:h+1odd}
  For all integers $h \geqslant 1$ the scheme $\mathcal M (2h+1)$ is not empty
and its general member $[\cU_{2h+1}] \in \mathcal M (2h+1)$ corresponds to a rank-$(2h+1)$ vector bundle $\cU_{2h+1}$ which is Ulrich w.r.t. $\xi$ and which satisfies 
$$c_1(\cU_{2h+1})= (2h+1) \xi+ \varphi^*\Oc_{\FF_e}(3,b_e-3)) + \varphi^*\Oc_{\FF_e}(h-1,(h-1) (b_e - e -2))$$and 
$h^j(\cU_{2h+1} \otimes \cU_{2h+1}^{\vee})=0$, $j=2,3$. 
\end{prop}

\begin{proof} We prove this by induction on $h$, the case $h=1$ being satisfied by the choice of $\mathcal M (3)$ as in Theorem \ref{thm:antonelli3}. Therefore, let $h \geqslant 2$; for general $[\cU_{2h-1}] \in
  \mathcal M (2h-1)$ and $[\cU_2] \in
  \mathcal M (2)$, one has 
	\begin{eqnarray*}
	\dim ({\rm Ext}^1(\cU_{2h-1},\cU_2)) = h^1(\cU_2 \otimes \cU_{2h-1}^{\vee}).
	\end{eqnarray*}
	By Lemma \ref{lemma:induodd}-(i) we have that 
	$h^j(\cU_2 \otimes \cU_{2h-1}^{\vee}) = 0$, for $j=2,3$. Therefore 
	\begin{eqnarray*}
	\chi(\cU_2 \otimes \cU_{2h-1}^{\vee}) = h^0(\cU_2 \otimes \cU_{2h-1}^{\vee}) - h^1(\cU_2 \otimes \cU_{2h-1}^{\vee})
	\end{eqnarray*}
	so, by specialization and invariance of $\chi$ in irreducible families, we have
  \begin{eqnarray} \label{eq:dimextodd}
    \dim ({\rm Ext}^1(\cU_{2h-1},\cU_2)) & = & h^1(\cU_2 \otimes \cU_{2h-1}^{\vee})  \\ \nonumber
		                                 & = &  -\chi(\cU_2 \otimes \cU_{2h-1}^{\vee})+h^0(\cU_2 \otimes \cU_{2h-1}^{\vee}) \\ \nonumber
	                                   &\geqslant&  -\chi(\cU_2 \otimes \cU_{2h-1}^{\vee}) =  - \chi(\cU_2 \otimes \cF_{2h-1}^{\vee})  \\ \nonumber
																		 & = &  (h-2) (6b_e - 9e -4) + (8b_e - 12 e +3) >0,  \nonumber 
    \end{eqnarray} the latter equality following from Lemma \ref{lemma:induodd}-(ii) (with $h$ replaced by $h-1$) whereas the last strict inequality 
		following from $h \geqslant 2$ and $b_e \geqslant 3e +2$ by \eqref{eq:rem:assAB}. 
		
		The above computations prove that $\dim ({\rm Ext}^1(\cU_{2h-1},\cU_2)) >0$,  i.e. there exist non-trivial extensions as in 
		\eqref{eq:estensioneodd}, and that the scheme $\mathcal M (2h+1)^{\rm ext}$, and so also $\mathcal M (2h+1)$, is not empty.
		
		The members of $\mathcal M (2h+1)$ have rank $2h+1$ and first Chern class as in \eqref{eq:c1rcasoeodd}, since 
		$c_1 (\cU_{2h+1}) = c_1(\cF_{2h+1})$ being constant in $\mathcal M (2h+1)$. It is  immediate  that extensions of Ulrich bundles are still Ulrich, so the general member $[\cU_{2h+1}] \in 
		\mathcal M(2h+1)$ corresponds to an Ulrich bundle w.r.t. $\xi$. It also satisfies 
$h^j(\cU_{2h+1} \otimes \cU_{2h+1}^{\vee})=0$ for $j=2,3$ by Lemma \ref{lemma:induodd}-(iii).
\end{proof} 

We need to prove that the general member of $\mathcal M(2h+1)$ corresponds to a vector bundle which is slope--stable w.r.t. $\xi$, that $\mathcal M(2h+1)$ is generically smooth and  we need to compute the dimension at its general point $[\cU_{2h+1}]$. We will again prove all these facts by induction on $h$. Similarly as in the previous cases, we  need the following auxiliary result.

\begin{lemma} \label{lemma:uniquedest1odd}
Let $\cF_{2h+1}$ correspond to a general member of  $\mathcal M(2h+1)^{\rm ext}$, sitting in an extension like \eqref{eq:estensioneodd}.
Assume furthermore that $\cU_2$ and $\cU_{2h-1}$ are slope--stable.  Let $\cD$ be a destabilizing subsheaf of 
$\cF_{2h+1}$. Then $\cD^{\vee} \cong  {\cU_2}^{\vee}$ and $\left(\cF_{2h+1}/\cD\right)^{\vee} \cong \cU_{2h-1}^{\vee}$. 
\end{lemma}

\begin{proof} The proof is identical to that of Lemma \ref{lemma:uniquedest1}, so the reader is referred therein. 
\end{proof}

\begin{prop} \label{prop:h+1-IIodd} For all integers $h \geqslant 1$ the scheme $\mathcal M (2h+1)$ is not empty, generically smooth 
of dimension 
\begin{eqnarray*} \dim (\mathcal M(2h+1)) = ((h-1)^2 +2) (6b_e - 9e -4) + 9(h-1) (2b_e - 3e).
\end{eqnarray*}
Its general member corresponds to a slope-stable bundle $\cU_{2h+1}$ 
whose slope w.r.t. $\xi$ is $$\mu(\cU_{2h+1}) = 8b_e - k_e - 12 e -3.$$Furthermore, $\mathcal M (2h+1)$ properly contains the locally closed 
subscheme $\mathcal M (2h+1)^{\rm ext}$, namely $\dim(\mathcal M (2h+1)^{\rm ext}) < \dim(\mathcal M (2h+1))$. 
\end{prop}

\begin{proof} We prove this by induction on $h$, the case $h=1$ being satisfied by $\mathcal M(3)$ as in Theorem \ref{thm:antonelli3}, where in such a case $\mathcal M(3)^{\rm ext} = \emptyset$.

Let therefore $h \geqslant 2$ and assume that we have proved the statement for all positive integers $k \leqslant h-1$;  we will prove it for $h$.

The slope of the members of  $\mathcal M (2)$ and $\mathcal M (2h-1)$ are both equal to $ 8b_e - k_e - 12 e -3$ as in \eqref{eq:slopercasoeodd}. Thus, by \cite[Lemma\;4.2]{c-h-g-s}, the general member $[\cF_{2h+1}] \in \mathcal M (2h+1)^{\rm ext}$ (which is not empty by Proposition \ref{prop:h+1odd}) corresponds to a simple bundle. Hence, by semi-continuity, 
also the general member $[\cU_{2h+1}]\in \mathcal M (2h+1)$ corresponds to a simple bundle which also satisfies $h^j(\cU_{2h+1} \otimes \cU^{\vee}_{2h+1})=0$, $j=2,3$, by Lemma \ref{lemma:induodd}-(iii).
 
Therefore  $\mathcal M (2h+1)$ is smooth at $[\cU_{2h+1}]$ (see, e.g., \cite[Prop. 2.10]{c-h-g-s}) with
\begin{eqnarray} \label{eq:dimUh+1odd}
  \dim(\mathcal M (2h+1))&=&h^1(\cU_{2h+1} \otimes \cU^{\vee}_{2h+1})\\ \nonumber
 &=&-\chi(\cU_{2h+1} \otimes \cU^{\vee}_{2h+1})+h^0(\cU_{2h+1} \otimes \cU_{2h+1}^{\vee}) \\
\nonumber  
&=& ((h-1)^2 +2) (6b_e - 9e -4)+ 9(h-1) (2b_e - 3e), 
\end{eqnarray} using the facts that $h^0(\cU_{2h+1} \otimes \cU^{\vee}_{2h})=1$ as $\cU_{2h+1}$ is simple, and the computation of 
$\chi(\cU_{2h+1} \otimes \cU^{\vee}_{2h+1})$ in Lemma \ref{lemma:induodd}-(iv). This proves that $\mathcal M(2h+1)$ is generically smooth of the stated dimension. 

Finally, we prove that $\cU_{2h+1}$ general is slope--stable and that \linebreak $\dim(\mathcal M (2h+1)^{\rm ext}) < \dim(\mathcal M (2h+1))$.
 If $\cU_{2h+1}$ general were not slope-stable then we could find a one-parameter family of bundles 
$\{\cU_{2h+1}^{(t)}\}$ over the disc $\Delta$ such that $\cU_{2h+1}^{(t)}$ is a general member of $\mathcal M (2h+1)$ for $t \neq 0$ and $\cU_{2h+1}^{(0)}$ lies in $\mathcal M (2h+1)^{\rm ext}$, and such that we have a destabilizing sequence
\begin{equation} \label{eq:destat1odd} 
    0 \to \cD^{(t)} \to \cU_{2h+1}^{(t)} \to \cG^{(t)} \to 0
  \end{equation} for $t \neq 0$, which we can take to be saturated, that is, such that $\cG^{(t)}$ is torsion free, whence so that $\cD^{(t)}$ and $\cG^{(t)}$ are (Ulrich) vector bundles  
	(see \cite[Thm. 2.9]{c-h-g-s} or \cite[(3.2)]{b}). The limit of $\Pp(\cG^{(t)}) \subset \Pp(\cU_{2h+1}^{(t)})$ defines a subvariety of $\Pp(\cU_{2h+1}^{(0)})$ of the same dimension as 
	$\Pp(\cG^{(t)})$, whence a coherent sheaf $\cG^{(0)}$ of rank ${\rm rk} (\cG^{(t)})$ with a surjection $\cU_{2h+1}^{(0)} \twoheadrightarrow \cG^{(0)}$. 
	Denoting by $\cD^{(0)}$ its kernel, we have ${\rm rk} (\cD^{(0)})= {\rm rk} (\cD^{(t)})$ and 
	$c_1(\cD^{(0)})=c_1(\cD^{(t)})$. Hence, \eqref{eq:destat1odd} specializes to a destabilizing sequence for $t=0$. 
	
Lemma \ref{lemma:uniquedest1odd} yields  that ${\cD^{(0)}}^{\vee}$ (resp., ${\cG^{(0)}}^{\vee}$) 
is the dual of a member of $\mathcal M (3)$ (resp., of $\mathcal M (2h+1)$). It follows that
${\cD^{(t)}}^{\vee}$ (resp., ${\cG^{(t)}}^{\vee}$)
is a deformation of the dual of a member of $\mathcal M (2))$ (resp., of $\mathcal M (2h+1)$), whence that ${\cD^{(t)}}$ (resp., ${\cG^{(t)}}$) is a deformation of a member of $\mathcal M (2)$ (resp., $\mathcal M (2h+1)$), as both are locally free. It follows that 
$[\cU_{2h+1}^{(t)}] \in \mathcal M (2h+1)^{\rm ext}$ for $t \neq 0$. Thus,
\begin{equation} \label{eq:sonugualiodd}
  \mathcal M (2h+1)^{\rm ext}=\mathcal M (2h+1).
\end{equation}

On the other hand we have
\begin{equation} \label{eq:dimext2odd}
  \dim (\mathcal M (2h+1)^{\rm ext}) \leqslant \dim (\Pp(\Ext^1(\cU_{2h-1},\cU_2))) +\dim (\mathcal M (2h-1)) +\dim (\mathcal M (2)),
\end{equation}
for $[\cU_{2h-1}] \in \mathcal M (2h-1)$ and $[\cU_2]\in \mathcal M (2)$ general.  Because  $\cU_2$  is  slope--stable and also $\cU_{2h-1}$  is slope--stable by induction, of the same slope,  we have 
$h^0(\cU_2 \otimes \cU_{2h-1}^{\vee})=0$. Thus, \eqref{eq:dimextodd} gives 
\[  h^1(\cU_2 \otimes \cU^{\vee}_{2h-1})  =  -\chi(\cU_2 \otimes \cU^{\vee}_{2h-1}) = (h-2) (6b_e - 9e -4) + (8b_e - 12e +3). 
\]Hence, from \eqref{eq:dimext2odd}, using also \eqref{eq:dimUh+1odd} and the fact that $\dim(\mathcal M(2)) = (6b_e - 9e -3)$ 
from Theorem \ref{thm:rk 2 vctB e>0}, one has 
\begin{eqnarray*}
  \dim (\mathcal M (2h+1)^{\rm ext}) & \leqslant & (h-2)(6b_e - 9e -4) + (8b_e - 12 e +3) - 1 + \\
	&  & + ((h-2)^2 + 2) (6b_e - 9e - 4) + 9 (h-2) (2b_e - 3e) +  (6b_e - 9e -3) \\
  & < & ((h-1)^2 +2)(6b_e - 9e -4)+ 9(h-1) (2b_e - 3e) =\dim (\mathcal M (2h+1)), 
\end{eqnarray*} as $h \geqslant 2$. The previous inequality shows that 
$\dim (\mathcal M (2h+1)^{\rm ext}) < \dim (\mathcal M (2h+1))$, as stated; in particular \eqref{eq:sonugualiodd} is a contradiction, 
which forces also $\cU_{2h+1}$ general to be slope-stable.
\end{proof}

The collection of the previous results in even and in odd ranks, respectively, gives the following:

\begin{theo}\label{thm:generale}  Let  $(X_e, \xi) \cong \scrollcal{E_e}$ be a $3$-fold  scroll over $\FF_e$, with $e>0$ and $\mathcal E_e$ as in Assumptions \ref{ass:AB}. 
Let $\varphi: X_e \to \FF_e$ be the scroll map and $F$ be the $\varphi$-fibre. Let $r \geqslant 2$ be any integer. 

Then the moduli space of rank-$r$ vector bundles $\cU_r$ on $X_e$ which are Ulrich w.r.t. $\xi$ and with first Chern class
\begin{eqnarray*}c_1(\cU_r) =
    \begin{cases}
      r \xi + \varphi^*\Oc_{\FF_e} (3, b_e -3) + \varphi^*\Oc_{\FF_e}\left(\frac{r-3}{2}, \frac{(r-3)}{2} (b_e -  e- 2\right), & \mbox{if $r$ is odd}, \\
      r \xi + \varphi^*\Oc_{\FF_e}\left(\frac{r}{2}, \frac{r}{2}(b_e-e -2) \right), & \mbox{if $r$ is even}.
    \end{cases}\end{eqnarray*}
is not empty and it contains a generically smooth component $\mathcal M(r)$ of dimension 
		\begin{eqnarray*}\dim (\mathcal M(r) ) = \begin{cases} \left( \frac{(r-3)^2}{4} + 2 \right) (6 b_e - 9e -4) + \frac{9}{2} (r-3) (2b_e - 3e), & \mbox{if $r$ is odd}, \\
			 \frac{r^2}{4} (6b_e- 9e-4) +1 , & \mbox{if $r$ is even}.
    \end{cases}
    \end{eqnarray*} 
		The general member $[\cU_r] \in \mathcal M(r)$  corresponds to a  slope-stable vector bundle, of slope w.r.t. $\xi$ given by 
$\mu(\cU_r) = 8 b_e - k_e - 12 e - 3$. 
\end{theo}
\begin{proof} In the even case, the statement directly follows from Theorem  \ref{thm:rk 2 vctB e>0}, 
\eqref{eq:c1rcasoe}, \eqref{eq:slopercasoe} and from Propositions \ref{prop:h+1}, \ref{prop:h+1-II}. For odd cases, the statement follows from Theorem  \ref{thm:antonelli3}, 
\eqref{eq:c1rcasoeodd}, \eqref{eq:slopercasoeodd} and from Propositions \ref{prop:h+1odd}, \ref{prop:h+1-IIodd}
\end{proof}

\section{Final remarks on Ulrichness over $\FF_e$, $e \geqslant 0$}\label{S:final} As a direct consequence of {\bf Main Theorem}, {\bf Main Corollary}, Theorem \ref{pullback} and the one--to--one correspondence in \cite[Proposition 6.2]{f-lc-pl}, one has the following result concerning moduli spaces of rank-$r$ vector bundles on Hirzebruch surfaces $\FF_e$, for any $r \geqslant 1$ and any $e \geqslant 0$, which are Ulrich w.r.t. the very ample line bundle $c_1(\mathcal E_e) = 3 C_e + b_ef$, with $b_e \geqslant 3e+2$ as in \eqref{eq:rem:assAB} (the case $r=1,2,3$ 
already known by \cite{a-c-mr,cas,ant}). 

\begin{theo}\label{thm:UlrichFe} For any integer $e \geqslant 0$, consider the Hirzebruch surface 
$\mathbb{F}_e$ and let $\Oc_{\FF_e}(\alpha,\beta)$ denote the line bundle 
$\alpha C_e + \beta f$ on $\mathbb{F}_e$, where $C_e$ and $f$ are the generators 
of ${\rm Num}(\mathbb{F}_e)$.

Consider the very ample polarization $c_1(\E_e)= 3 C_e + b_e f$, where $b_e \geqslant 3e+2$. Then: 

\smallskip

\noindent 
(a) $\FF_e$  does  not support any Ulrich line bundle  w.r.t. $c_1(\E_e)$ unless $e = 0$. In this latter case, the unique line bundles on 
$\FF_0$ which are Ulrich w.r.t. $c_1(\E_e)$ are 
$$\mathcal L_1 :=\Oc_{\FF_0}(5,b_0-1) \; {\rm and} \; \mathcal L_2 := \Oc_{\FF_0}(2,2b_0-1).$$

\smallskip 

\noindent
(b) Set $e=0$ and let $r \geqslant2$ be any integer. Then the moduli space of rank-$r$ vector bundles $\cH_r$ on $\FF_0$ which are Ulrich w.r.t. $c_1(\E_0)$ and with first Chern class
\begin{eqnarray*}c_1(\cH_r) =
    \begin{cases}
      \Oc_{\FF_0}(3(r+1), (r+1) b_0 -3) \otimes  \Oc_{\FF_0}\left(\frac{r-3}{2}, \frac{(r-3)}{2}(b_0-2)\right), & \mbox{if $r$ is odd}, \\
      \Oc_{\FF_0}(3r, r b_0) \otimes \Oc_{\FF_0}\left(\frac{r}{2}, \frac{r}{2}(b_0-2)\right), & \mbox{if $r$ is even}.
    \end{cases}\end{eqnarray*}
    is not empty and it contains a generically smooth component $\mathcal M_{\FF_0}(r)$ of dimension 
		\begin{eqnarray*}\dim (\mathcal M_{\FF_0}(r) ) = \begin{cases} \frac{(r^2 -1)}{4}(6 b_0 -4), & \mbox{if $r$ is odd}, \\
			 \frac{r^2}{4} (6b_0-4) +1 , & \mbox{if $r$ is even}.
    \end{cases}
    \end{eqnarray*}
    The general point $[\cH_r] \in \mathcal M_{\FF_0}(r)$ corresponds to a  slope-stable vector bundle.

\smallskip

\noindent
(c) When $e >0$, let $r \geqslant 2$ be any integer. Then the moduli space of rank-$r$ vector bundles $\cH_{r}$ on $X_e$ which are Ulrich w.r.t. $c_1(\E_e)$ and with first Chern class
\begin{eqnarray*}c_1(\cH_r) =
    \begin{cases}
     \Oc_{\FF_e}(3 (r+1), (r+1)b_e-3) \otimes \Oc_{\FF_e}\left(\frac{r-3}{2}, \frac{(r-3)}{2}(b_e - e -2)\right), & \mbox{if $r$ is odd}, \\
      \Oc_{\FF_e}(3 r, r b_e) \otimes  \Oc_{\FF_e}\left(\frac{r}{2}, \frac{r}{2}(b_e-e-2)\right), & \mbox{if $r$ is even}.
    \end{cases}\end{eqnarray*}
    is not empty and it contains a generically smooth component $\mathcal M_{\FF_e}(r)$ of dimension 
		\begin{eqnarray*}\dim (\mathcal M_{\FF_e}(r) ) = \begin{cases} \left(\frac{(r -3)^2}{4}+ 2 \right)(6 b_e - 9e -4) + \frac{9}{2}(r-3) (2b_e-3e), & \mbox{if $r$ is odd}, \\
			 \frac{r^2}{4} (6b_e- 9e-4) +1 , & \mbox{if $r$ is even}.
    \end{cases}
    \end{eqnarray*} The general point $[\cH_{r}] \in \mathcal M_{\FF_e}(r)$ corresponds to a  slope-stable vector bundle. 
\end{theo}

\begin{proof} (a) When $e >0$, the fact that $\FF_e$ does not support line bundles which are Ulrich w.r.t. $c_1(\E_e) = 3 C_e + b_ef$ is proved in \cite[Thm.\;2.1]{a-c-mr}; however, this fact is also a direct consequence of {\bf Main Theorem}--(a) and Theorem \ref{pullback}. 

For $e=0$  observe that the line bundles $L_1$ and $L_2$ as in {\bf Main Theorem}--(a), which are Ulrich w.r.t. $\xi$ on $X_0$, give rise, by push forward and by Theorem \ref{pullback}, to Ulrich line bundles  w.r.t $c_1(\E_0)=3C_0+b_0 f$ on $\FF_0$. Indeed
\begin{eqnarray*}
L_1\otimes \xi^{\vee} = \varphi^{*}\Oc_{\FF_0}(2,-1)\, \, \mbox{and} \,\, L_2\otimes \xi^{\vee} = \varphi^{*}\Oc_{\FF_0}(-1,b_0-1),
\end{eqnarray*}
 thus 
\begin{eqnarray*}
\varphi_{*}(L_1\otimes \xi^{\vee})=\Oc_{\FF_0}(2,-1)\, \, \mbox{and} \,\,\varphi_{*}(L_2\otimes \xi^{\vee})=\Oc_{\FF_0}(-1,b_0-1),
\end{eqnarray*} which give 
\begin{eqnarray*}
\mathcal L_1 := \varphi_{*}(L_1\otimes \xi^{\vee})\otimes \Oc_{\FF_0}(c_1(\E_0))=  \Oc_{\FF_0}(5,b_0-1) \\
 \mathcal L_2 := \varphi_{*}(L_2\otimes \xi^{\vee})\otimes \Oc_{\FF_0}(c_1(\E_0))= \Oc_{\FF_0}(2,2b_0-1)
\end{eqnarray*}that are the only Ulrich line bundles on $\FF_0$ w.r.t. $c_1(\E_e) =3C_0+b_0 f$, according to \cite[Example 2.3]{cas} and \cite[Proposition 4.4]{ant}.

\smallskip

\noindent
(b) As for any rank $r \geqslant 2$ in the case $e=0$, observe that vector bundles $\cU_r$ on $X_0$ as in {\bf Main Theorem}--(b) when restricted to a general fibre $F$ of $\varphi: X_0 \to \FF_0$ are such that $\cU_r|_F \cong \cO_{\Pp^1}(1)^{\oplus r}$. To see this, one proceeds by induction. By \eqref{extension1}, $L_1|_F \cong \cO_{\Pp^1}(1) \cong L_2|_F$, thus \linebreak 
$\cF_1| _F \in {\rm Ext}^1(\cO_{\Pp^1}(1),\cO_{\Pp^1}(1))=H^1(\cO_{\Pp^1})=(0)$, hence
$\cF_1| _F \cong \cO_{\Pp^1}(1)^{\oplus 2}$, which is the most balanced splitting.
Since $\cU_2=\cU$ as in Theorem \ref{prop:rk 2 simple Ulrich vctB e=0;I} 
is a deformation of  $\cF_1$ then $\cU_2 |_F \cong \cO_{\Pp^1}(1)^{\oplus 2}$.  Assume by induction, that for some $r \geqslant 3$ one has $\cU_{r-1} |_F \cong \cO_{\Pp^1}(1)^{\oplus r-1}$ for $[\cU_{r-1}] \in \cM(r-1)$ general.
Then by \eqref{eq:estensioneL}, one has
\begin{eqnarray*}
0 \to \cU_{r-1} |_F =\cO_{\Pp^1}(1)^{\oplus r-1} \to \cF_r |_F\to L_{\epsilon_r}|_F=\cO_{\Pp^1}(1) \to 0
 \end{eqnarray*}
 and, once again, since ${\rm Ext}^1(\cO_{\Pp^1}(1), \cO_{\Pp^1}(1)^{\oplus r-1} )=0$ then $\cF_r |_F\cong \cO_{\Pp^1}(1)^{\oplus r}$. 

Thus by \cite[Theorem 6.1]{f-lc-pl}, using the diagram 
 
 \begin{equation*}
    \label{composition}
    \xymatrix@-2ex{
    \tilde{S} \, \, \ar@{^{(}->}[r]^i \ar[rd]_{\varphi'} & X_0\ar[d]^{\varphi} \\
                                      & {\FF_0}
    }
\end{equation*}
therein, where $\varphi':   \tilde{S} \to \FF_0$ is the blow-up  at $c_2(\E_0)$ points on $\FF_0$ and where $\sum E_i$  is the  $\varphi'$-exceptional divisor, one has that $\varphi_*(\cU_r \otimes i_*(\cO_{\tilde{S}}(\sum E_i)))$ is a rank $r$ vector bundle on $\FF_0$ which is Ulrich w.r.t. $c_1(\E_0)=3C_0+b_0 f $. More precisely 
\cite[Proposition 6.2]{f-lc-pl} gives rise to a component of the moduli space of Ulrich bundles of rank $r$ on $\FF_0$, which are Ulrich w.r.t.  $c_1(\E_0)=3C_0+b_0 f$, of first Chern class 
\begin{eqnarray*}
c_1(\varphi_*(\cU_r \otimes i_*(\cO_{\tilde{S}}(\sum E_i)))).
 \end{eqnarray*} On the other hand, if we set $\cH_r := \varphi_*(\cU_r \otimes i_*(\cO_{\tilde{S}}(\sum E_i)))$, 
the one-to-one correspondence in \cite[Proposition 6.2]{f-lc-pl} asserts that $\cU_r \cong \xi \otimes \varphi^*(\cH_r(-c_1(\E_0)))$. From {\bf Main Theorem}-(b) one knows 
$c_1(\cU_r \otimes  \xi^{\vee})$ and, since $c_1(\cU_r \otimes  \xi^{\vee}) =  c_1(\cH_r(-c_1(\E_0))) =  c_1(\cH_r)-r c_1(\E_0)$, it follows that 
$c_1 (\cH_r)$ is as stated; pairs $(r, c_1(\cH_r))$ are {\em Ulrich admissible pairs} w.r.t. $3C_0+b_0 f$ in the sense of \cite[Def.\;5.1]{ant}.

By  the one-to-one correspondence in \cite[Proposition 6.2]{f-lc-pl} and {\bf Main Theorem}-(b), we have therefore the existence of moduli spaces 
$\cM_{\FF_0}(r)$ of Ulrich bundles on $\FF_0$ w.r.t. $ 3 C_0 + b_0 f$, of any rank $r\geqslant 2$, of first Chern class and dimension as stated. 
Moreover since $\cU_r \cong \xi \otimes \varphi^*(\cH_r(-c_1(\E_0)))$ then
\begin{eqnarray*}
h^j(X_0, \cU_r \otimes \cU_r^{\vee})&=&h^j(X_0,  \varphi^*(\cH_r(-c_1(\E_0)))\otimes \varphi^*(\cH_r(-c_1(\E_0)))^{\vee})\\ \nonumber
&=& h^j(X_0,  \varphi^*(\cH_r\otimes \cH_r^{\vee}))=h^j(\FF_0,  \cH_r\otimes  \cH_r^{\vee}),  
\end{eqnarray*}i.e. $h^2(\FF_0,  \cH_r\otimes  \cH_r^{\vee}) =0$ and $h^0(\FF_0,  \cH_r\otimes  \cH_r^{\vee}) =1$, 
namely the component  $\cM_{\FF_0}(r)$ is also generically smooth and its general point $[\cH_r] \in  \cM_{\FF_0}(r)$ corresponds to a simple bundle. 

It is also clear that $\cH_r$ is slope-stable w.r.t. $3 C_0 + b_0f$: if not, by the one-to-one correspondence given by \cite[Proposition 6.2]{f-lc-pl}, any destabilizing rank-$k$ 
Ulrich sub-bundle $\mathcal D_k$ of $ \mathcal \cH_r$ general, for some $1 \leqslant k \leqslant r-1$, would give rise to a rank-$k$ 
vector bundle $\xi \otimes \varphi^*(\mathcal D_k(- c_1 (\mathcal E_0))$ which is Ulrich on $X_0$ w.r.t. $\xi$, by Theorem \ref{pullback}, and which would be a 
destabilizing sub-bundle on $\cU_r$ general in $\mathcal M(r)$ on $X_0$, contradicting {\bf Main Theorem}--(b).  

When $r=2,3$, the previous arguments are in accordance with \cite[Prop.\;6.4]{ant}, for the polarization $3C_0+b_0 f$, and $c_1$ as stated.

\smallskip

\noindent
(c) Similar arguments as in (b), but for the case for $e>0$, are obtained by using {\bf Main Theorem}--(c), Theorem \ref{pullback} and the one-to-one correspondence in \cite[Proposition 6.2]{f-lc-pl}. 
\end{proof}
	
From Theorem \ref{thm:UlrichFe} it follows that the pairs $(\mathbb{F}_e,  3 C_e + b_e f)$ are Ulrich wild and this is in accordance with \cite[Lemma 5.2]{cas}. 

\vspace{2mm}

We want to stress that, when $r=2,\;3$, we cannot deduce the irreducibility of the moduli spaces of bundles $\cU_r$ as in {\bf Main Theorem}--(b) or (c) from the correspondence in  \cite[Proposition 6.2]{f-lc-pl} and irreducibility results on $\FF_e$, $e \geqslant 0$, given by \cite[Theorem 4.7]{CM} (cf. also \cite[Propositions\;6.1,\;6.4]{ant}). Indeed, in principle, there could exist other components, different from ours $\cM(r)$, $r=2, 3$, where the general Ulrich bundle therein does not split as $\cO_{\Pp^1}(1)^{\oplus r}$ on the general $\varphi$--fiber.

  %
  %

\end{document}